\documentclass[a4paper,11pt,leqno]{article}

\usepackage[latin1]{inputenc}
\usepackage[T1]{fontenc} 
\usepackage[english]{babel}
\usepackage{verbatim}
\usepackage{calligra}
\usepackage{enumerate}
\usepackage{dsfont}
\usepackage{amsfonts}
\usepackage{amsmath}
\usepackage{amsthm}
\usepackage{amssymb}
\usepackage{mathtools}
\usepackage{mathrsfs}
\usepackage{color}
\usepackage{graphicx}
\usepackage{bm}
\usepackage{srcltx}

\usepackage{tikz}
\usepackage[retainorgcmds]{IEEEtrantools}

\usepackage{graphicx}		
\usepackage[hypcap=false]{caption}

\DeclareMathOperator*{\tend}{\longrightarrow}

\renewcommand{\leq}{\leqslant}
\renewcommand{\geq}{\geqslant}

\renewcommand{\div}{\operatorname{div}}

\newcommand{\orange}{\color{orange}}

\theoremstyle{definition}
\newtheorem{defin}{Definition}[section]

\newtheorem{rmk}[defin]{Remark}

\theoremstyle{plane}
\newtheorem{thm}[defin]{Theorem}
\newtheorem{prop}[defin]{Proposition}
\newtheorem{cor}[defin]{Corollary}
\newtheorem{lemma}[defin]{Lemma}

\newcommand{\tsl}{\textsl}

\newcommand{\mbb}{\mathbb}

\newcommand{\mc}{\mathcal}
\newcommand{\mf}{\mathfrak}

\newcommand{\veps}{\varepsilon}
\newcommand{\what}{\widehat}
\newcommand{\wtilde}{\widetilde}
\newcommand{\vphi}{\varphi}

\newcommand{\g}{\gamma}

\newcommand{\s}{\sigma}
\renewcommand{\t}{\tau}

\newcommand{\de}{\delta}
\renewcommand{\o}{\omega}
\newcommand{\wu}{\wtilde{u}}
\newcommand{\wOm}{\wtilde{\Omega}}
\newcommand{\wom}{\wtilde{\omega}}
\newcommand{\bstar}{\beta_*}

\newcommand{\lan}{\langle}
\newcommand{\ran}{\rangle}

\newcommand{\R}{\mathbb{R}}

\newcommand{\N}{\mathbb{N}}
\newcommand{\Z}{\mathbb{Z}}

\renewcommand{\div}{{\rm div}\,}

\newcommand{\curl}{{\rm curl}\,}

\newcommand{\supp}{{\rm Supp}\,}

\newcommand{\resc}{{\rm resc}}

\newcommand{\dx}{ \, {\rm d} x}
\newcommand{\dt}{ \, {\rm d} t}

\newcommand{\bt}{\beta}
\newcommand{\weak}{{\rm w}}

\allowdisplaybreaks

\def\d{\partial}
\def\div{{\rm div}\,}
\newcommand{\dd}{{\rm d}}

\begin{document}

\newcommand{\fra}[1]{\textcolor{blue}{#1}}

\title{\textsc{\Large{\textbf{Global existence for non-homogeneous incompressible inviscid fluids in presence of Ekman pumping}}}}

\author{\normalsize\textsl{Marco Bravin}$\,^1\qquad$ and $\qquad$
\textsl{Francesco Fanelli}$\,^{2a,b,c}$ \vspace{.5cm} \\
\footnotesize{$\,^{1}\;$ \textsc{Universidad de Cantabria, E.T.S. de Ingenieros Industriales y de Telecomunicac\'ion}}  \vspace{.1cm} \\
{\footnotesize \it Departmento de Matem\'atica Aplicada y Ciencias de la Computac\'ion}  \vspace{.1cm}\\
{\footnotesize Avd. Los Castros 44, 39005 Santander, SPAIN} \vspace{.3cm} \\
\footnotesize{$\,^{2a}\;$ \textsc{BCAM -- Basque Center for Applied Mathematics}} \\ 
{\footnotesize Alameda de Mazarredo 14, E-48009 Bilbao, Basque Country, SPAIN} \vspace{.2cm} \\
\footnotesize{$\,^{2b}\;$ \textsc{Ikerbasque -- Basque Foundation for Science}} \\  
{\footnotesize Plaza Euskadi 5, E-48009 Bilbao, Basque Country, SPAIN} \vspace{.2cm} \\
\footnotesize{$\,^{2c}\;$ \textsc{Universit\'e Claude Bernard Lyon 1}, {\it Institut Camille Jordan -- UMR 5208}} \\ 
{\footnotesize 43 blvd. du 11 novembre 1918, F-69622 Villeurbanne cedex, FRANCE} \vspace{.3cm} \\
\footnotesize{Email addresses: $\,^{1}\;$\ttfamily{marco.bravin@unican.es}}, $\;$
\footnotesize{$\,^{2}\;$\ttfamily{ffanelli@bcamath.org}}
\vspace{.2cm}
}

\date\today

\maketitle

\subsubsection*{Abstract}
{\footnotesize 

In this paper, we study the global solvability of the density-dependent incompressible Euler equations,
supplemented with a damping term of the form $ \mathfrak{D}_{\alpha}^{\gamma}(\rho, u) = \alpha \rho^{\gamma} u $, where $\alpha>0$ and $ \gamma \in \{0,1\} $.
To some extent, this system can be seen as a simplified model describing the mean dynamics in the ocean; from this perspective,
the damping term can be interpreted as a term encoding the effects of the celebrated Ekman pumping in the system.

On the one hand, in the general case of space dimension $d\geq 2$, we establish global well-posedness in the Besov spaces framework,
under a non-linear smallness condition involving the size of the initial velocity field $u_0$, of the initial non-homogeneity $\rho_0-1$ and of the damping coefficient $\alpha$.
On the other hand, in the specific situation of planar motions and damping term with $\g=1$,
we exhibit a second smallness condition implying global existence, which in particular yields global well-posedness for arbitrarily
large initial velocity fields, provided the initial density variations $\rho_0-1$ are small enough.
The formulated smallness conditions rely only on the endpoint Besov norm $B^1_{\infty,1}$ of the initial datum, whereas, as a byproduct of our analysis,
we derive exponential decay of the velocity field and of the pressure gradient in the high regularity norms $B^s_{p,r}$.

}

\paragraph*{\small 2020 Mathematics Subject Classification:}{\footnotesize 35Q35 
(primary);
35A01, 
35B65, 
76B03 
(secondary).}

\paragraph*{\small Keywords: }{\footnotesize incompressible Euler equations; density variations; damping term; global in time well-posedness.}


\section{Introduction} \label{s:intro}

Ocean dynamics is an intricate phenomenon, determined by several complementary, and sometimes conflicting, factors \cite{CR-B, Ped}.
Because of the large scales involved in the process, the two predominant ones are
the action of the Coriolis force, which tends to drive
the flow to a purely planar ($2$-D, horizontal) configuration, and the action of gravity, which instead tends to stratify the density on the vertical (\tsl{i.e.}, orthogonal
to the previously mentioned effect) direction.
Thus, if we can neglect temperature and salinity fluctuations in a first approximation, we cannot do so for the variations of the density.

Another important feature of the motion of water currents in oceans is the so-called Ekman pumping effect. This is a mechanism of dissipation of kinetic energy, which
originates from a boundary layer (the celebrated Ekman layer) close to the oceanic bottom, but involves in fact a global dynamics, affecting also the mean flow
in the bulk. On a rigorous ground, the mathematical analysis of Ekman boundary layers was initiated by the pioneering works \cite{Gr-Masm} (for the well-prepared data case)
and \cite{Masm} (for ill-prepared initial data). 
In those works, it is proven that, in the fast rotation regime, an Ekman pumping term appears in the target momentum equation. This Ekman pumping term takes the form of a 
linear damping term $\alpha u$, where $\alpha>0$.
We refer to the book \cite{C-D-G-G} for a complete overview of the subject of (homogeneous) fluids in fast rotation, as well as for a study of Ekman layers.
We also refer to \cite{B-F-P} for a description of Ekman boundary layers for compressible flows in presence of stratification, and to
\cite{Brav-F} for a new approach to the derivation of the Ekman pumping effect (without a description of the Ekman layers, though),
implemented in the context of non-homogeneous incompressible fluids.

\medbreak
In this work, we are interested in the (global) well-posedness issue for a model which may be seen as a rough approximation of ocean dynamics
in the regime of fast rotation.
Notice that the fluid can be assumed to be incompressible, at least at a fairly good level of approximation.
In addition, at large scales, the Reynolds number of the flow is typically very high (see \tsl{e.g.} the discussion in Section 4.4 of \cite{CR-B}),
so viscous effects may be discarded.

Motivated by the previous considerations, and restricting our attention to the (mean) motion in the bulk,
it is thus natural to consider the dynamics of an incompressible inviscid fluid
presenting density variations and energy dissipation through a damping term.

\subsection{The equations under study} \label{ss:equations}

At the mathematical level, the system we are interested in is described by the so-called \emph{non-homogeneous incompressible Euler equations}.
If we denote by $\rho=\rho(t,x)\geq0$ the density of the fluid, by $u=u(t,x)\in\R^d$, with $d\geq2$, its velocity field and by $\Pi=\Pi(t,x)\in\R$ its pressure field,
those equations can be written in the following way:
\begin{equation} \label{eq:dd-E}
\left\{\begin{array}{l}
\d_t\rho\,+\,u\cdot\nabla\rho\,=\,0 \\[1ex]
\rho\,\d_tu\,+\,\rho\,u\cdot\nabla u\,+\,\nabla\Pi\,+\,\alpha\,\rho^\g\,u\,=\,0 \\[1ex]
\div u\,=\,0\,.
       \end{array}
\right.
\end{equation}
Here, $\alpha>0$ is a positive parameter and $\g\in\{0,1\}$. The term
\begin{equation} \label{eq:damping}
\mf D^\g_\alpha(\rho,u)\,:=\,\alpha\,\rho^\g\,u 
\end{equation}
is thus a damping term, whose role is, precisely, to encode the Ekman pumping effect of ocean dynamics. On the other hand, for the sake of simplicity, in
the system above we discard remainders coming from to the (fast) Earth rotation.

Notice that a somewhat similar version of the previous system, but with a static density profile $\rho=\rho(x)$, was derived in \cite{Brav-F}
from a singular limit problem for non-homogeneous incompressible viscous fluids in thin domains, in the fast rotation regime.
Although the target system coming from fast rotating fluids should be only two-dimensional, for the sake of generality in this work we set system
\eqref{eq:dd-E} in
\[
 \R_+\times\R^d\,,\qquad\qquad \mbox{ with }\qquad d\geq2\,.
\]
We are interested in the initial value problem related to equations \eqref{eq:dd-E}, thus we supplement them with an initial datum
\begin{equation} \label{eq:in-datum}
 \big(\rho,u\big)_{|t=0}\,=\,\big(\rho_0,u_0\big)\qquad\qquad \mbox{ such that }\qquad 
 \div u_0\,=\,0\,.
\end{equation}
Throughout this work, we assume \emph{absence of vacuum}. This means that there exist two positive constants $0<\rho_*\leq \rho^*$ such that
\begin{equation} \label{eq:vacuum}
 0<\rho_*\leq\rho_0\leq\rho^*\,.
\end{equation}

\subsection{Previous results} \label{ss:previous}

When $\alpha=0$, we observe that the previous system reduces to the non-homogeneous incompressible Euler equations, namely
\[ 
\left\{\begin{array}{l}
\d_t\rho\,+\,u\cdot\nabla\rho\,=\,0 \\[1ex]
\rho\,\d_tu\,+\,\rho\,u\cdot\nabla u\,+\,\nabla\Pi\,=\,0 \\[1ex]
\div u\,=\,0\,.
       \end{array}
\right. \leqno{\rm (ddE)}
\] 
This system can be seen as a generalisation of the classical Euler system (obtained when $\rho\equiv1$) to the case of fluids presenting density variations.

On the mathematical side, we notice that, despite the presence of the (non-local) pressure term, equations (ddE) retain the structure of a transport system,
at least under the assumption \eqref{eq:vacuum} of absence of vacuum. Thus, their well-posedness can be addressed in any reasonable functional spaces
continuously embedded in the set of globally Lipschitz functions $W^{1,\infty}$. For instance, papers \cite{D1, D:F} established local well-posedness
in the class of Besov spaces $B^s_{p,r}=B^s_{p,r}(\R^d)$ verifying $B^s_{p,r}\hookrightarrow W^{1,\infty}$,
up to the endpoint case $p=+\infty$ and $s=r=1$.
Recall that the previous embedding holds true if and only if the triplet of indices $(s,p,r)\in\R\times[1,+\infty]\times[1,+\infty]$
satisfies the condition
\begin{equation} \label{cond:Lipschitz}
s\,>\,1\,+\,\frac{d}{p}\qquad\qquad\quad \mbox{ or }\qquad\qquad\quad s\,=\,1\,+\,\frac{d}{p}\quad \mbox{ and }\quad r\,=\,1\,.
\end{equation}
Paper \cite{F_2012} proved instead well-posedness under tangential (also dubbed ``striated'' and ``conormal'') regularity assumptions \tsl{\`a la Chemin}.
We refer to the introduction of those papers for a more complete overview of the literature concerning well-posedness of equations (ddE).

It is worth to point out that all those results hold true only \emph{locally in time}, even in the two-dimensional case. This is a purely heterogeneity effect, inasmuch
as variations of the density are responsible for vorticity production, thus causing the breaking of all the techniques available from the classical case $\rho\equiv1$
to prove global well-posedness.
As a matter of fact, global existence of regular solutions for equations (ddE) still remains as an open problem in general (however, see the recent preprint
\cite{CWZZ} for a study about global asymptotic stability around the Couette flow).

\medbreak
The study of the case $\alpha>0$, namely of system \eqref{eq:dd-E}, seems instead new in the literature, even though 
some previous results are in fact available for systems arising from fluid dynamics in presence of a damping term.

For instance, the compressible counterpart of equations \eqref{eq:dd-E}, where one drops the condition $\div u=0$ and lets the pressure term
$\Pi=\Pi(\rho)$ depend on the density $\rho$, was first investigated in \cite{S-T-W}. There, the authors proved that, for small initial data, the presence
of the damping term removes the appearing of finite time singularities and enables to prove global existence results. Fine results
(global well-posedness and long time dynamics) about the damped compressible Euler system have recently been obtained in the series of papers
\cite{CB:D2, CB:D3, CB:Z2}, as particular cases of a general theory dealing with
partially dissipative hyperbolic systems in a critical functional framework, a theory originating from the work \cite{B:Z} of Beauchard and Zuazua.

It is worth mentioning that, in all the works quoted above, the damping term was of the same form as $\mf D^\g_\alpha$ in \eqref{eq:damping}, with $\g=1$.
Roughly speaking, in this case and for small density perturbations around a constant state, one can then divide the momentum equation by $\rho$
and obtain a damped equation for $u$, for which it is then possible to derive an exponential decay in \tsl{e.g.} $L^2$.
Observe that, despite this simple heuristics, the fact that this decay implies global well-posedness is not straightforward \tsl{a priori}, and is indeed a remarkable
property due to the structure of the hyperbolic system: the dissipation does not act on all the components of the solution, but only on some of them.
We refer to \cite{B:Z} for a more in-depth discussion about this point.

From this point of view, our study about the incompressible case \eqref{eq:dd-E} will suffer of the same difficulty. Nonetheless,
the introduction of non-local effects due to the pressure function determines the loss of the hyperbolic structure which was underlying the compressible equations.
In other words, the coupling between mass and momentum conservation equations becomes weaker in the incompressible setting, thus making
the approach of the above mentioned \cite{CB:D2, CB:D3} out of use here and, at the same time, implying the appearance of new
difficulties in the analysis.

\medbreak
On the other hand, previous results about fluid systems with damping in the velocity equation exist also in the incompressible case, albeit
(to the best of our knowledge) only in the homogeneous or quasi-homogeneous situation. By this, we mean that
the systems under study were obtained under the so-called Boussinesq approximation;
as a consequence, the equation for the velocity field $u$ did not involve any heterogeneity in the convective term $\d_tu+u\cdot\nabla u$.
For instance, this is the framework considered in \cite{B-CB-P}, in the context of the Inviscid Porus Medium equation, and in \cite{Ca-Co-L}
in the context of the $2$-D Boussinesq system. Both those works address the question of the asymptotic stability of stably stratified solutions.
We remark that, as a consiquence of the Boussinesq approximation, at principal order the velocity field $u$ satisfies an Euler-type equation with pressure $\nabla\Pi$.
Correspondingly, the damping term considered in those works is of the form $\mf D^\g_\alpha$ with $\g=0$, namely it reduces to $\alpha\,u$ as in the case
of homogeneous fluid flows.
It follows that an application of the curl operator immediately
erases the presence of the pressure term and gives a damping effect on the vorticity function.

Because of the presence of density variations in equations \eqref{eq:dd-E}, our case will be very different from the ones discussed in \cite{B-CB-P, Ca-Co-L},
even for the choice $\g=1$. In particular, the pressure gradient cannot be eliminated from the equations; at the same time,
it is not clear, in our context, that one can define a ``damped mode'' as in \cite{B-CB-P} (see also works \cite{B:Z, CB:D2, CB:D3} for the compressible Euler equations)
and give a precise description of the asymptotic behaviour of the solutions when $t\tend+\infty$. We refer to the next subsection for more details about this
issue.

\subsection{Main results and strategy of the proof: an overview} \label{ss:overview}

In the present paper, we address the question of the global solvability for the non-homogeneous incompressible Euler equations with damping,
namely of system \eqref{eq:dd-E}-\eqref{eq:damping} with $\alpha>0$, in both cases $\g=0$ and $\g=1$.
Under suitable smallness assumptions (see more comments below), we are able to establish global in time well-posedness in the framework of non-homogeneous Besov spaces
$B^s_{p,r}$ satisfying either one of the two conditions appearing in \eqref{cond:Lipschitz}. Recall that each one of those conditions guarantees the embedding property
of $B^s_{p,r}$ in the space of globally Lipschitz functions.

Here, we want to give an overview of our main results and of the strategy of the proof. We refer to Section \ref{s:results} for the precise statements.

\paragraph*{The results.}
As already said, we prove global existence and uniqueness of solutions under suitable smallness assumptions. As a byproduct of our proof,
we also establish exponential decay of the velocity field $u$ and of the pressure term $\nabla\Pi$ in the high regularity norms.
For the sake of clarity, let us divide our discussion into two cases, depending on the type of results we obtain. 

The first class of results concerns the case of general space dimension $d\geq2$;
the precise statements are Theorem \ref{th:g=1_d} (for the case $\g=1$) and Theorem \ref{th:g=0_d} (for the case $\g=0$).
In those results, we formulate a \emph{non-linear} smallness assumption guaranteeing global existence,
where the term ``non-linear'' refers to the dependence of the condition on the norms of the initial
datum $\big(\rho_0,u_0\big)$ and on the damping coefficient $\alpha$. Roughly speaking, the meaning of such a condition is to require
either the norm of the initial datum to be small, or the damping coefficient to be sufficiently large, in order to deduce global in time results. 
Interestingly, the norms entering into play in the smallness conditions are only of lower order, resting on the $B^1_{\infty,1}$ norm
of the heterogeneity $\rho_0-1$ and $u$ (actually, when $\g=0$ one must require a control on the full $L^2\cap B^1_{\infty,1}$ norm of $u$).
Notice that, in the case $\g=0$, the smallness assumption reads in a slightly more complicated way, as the analysis of the system becomes more involved than in the case
$\g=1$.

The second class of results, instead, concerns the case $d=2$; however, we are forced to consider $\g=1$ here.
Under these assumptions,  we establish in Theorem \ref{th:g=1_2} global existence under a different smallness condition, resting
\emph{on the size of the non-homogeneity $\rho_0-1$ only}.
More precisely, we prove the existence of a unique global solution, provided $\rho_0-1$ is small enough in a suitable norm.
The smallness condition on $\rho_0-1$ depends on the size of the damping coefficient $\alpha>0$ and of the initial velocity field $u_0$,
which can be chosen arbitrarily large.

Unfortunately, when $d=2$ and $\g=0$ our method of proof breaks down and it is not clear whether or not a similar result can be obtained also in this case.

\paragraph*{Approach.}
The strategy of the proof of our main results is based on three main steps:
first, proving a local existence result (this is a quite simple adaptation of the theory of \cite{D1, D:F} for the case $\alpha=0$);
second, proving a continuation criterion of Beale-Kato-Majda type, resting only on the $L^\infty$ norm of $\nabla u$;
finally, controlling the low regularity norms globally in time.
On the one hand, this is the key in order to get a smallness assumption which depends only
on the low regularity norms of the initial datum. On the other hand, we point out that this approach yields exponential decay of the solution
(more precisely, of $u$ and $\nabla\Pi$)
only in the low regularity norms, and in principle not in the high regularity ones, as one would expect.
Actually, the sought decay for the higher regularity norms is obtained \tsl{a posteriori}, using the decay of low regularity norms and fine estimates, which
are similar in spirit to the ones which are necessary to establish the continuation criterion.

As already remarked above, it turns out that the case $\g=1$ is slightly easier than the case $\g=0$. The reason for this is that, when $\g=1$ and
in absence of vacuum \eqref{eq:vacuum}, one can divide the momentum equation by $\rho$ and find an equation of Euler type for the
rescaled velocity field $\wu\,:=\,e^{\alpha t}\,u$: after setting $\wtilde\Pi\,:=\,e^{\alpha t}\,\Pi$, one has
\[
 \d_t\wu\,+\,e^{-\alpha t}\,\wu\cdot\nabla\wu\,=\,\,-\,\frac{1}{\rho}\,\nabla\wtilde\Pi\,.
\]
Notice that the exponential growth of the right-hand side is only a fake: taking the divergence of the previous equation, one discovers that (as it is the case
in the incompressible Euler equations) the pressure gradient $\nabla\wtilde\Pi$ is a quadratic term in $\wu$. As a consequence, the right-hand side
of the previous equation actually decays exponentially in time. In the special case $d=2$ and $\g=1$, one has to complement the previous observations with
improved transport estimates in Besov spaces of regularity index $s=0$, originally due to Vishik \cite{Vis} (see also \cite{HK}). Controlling
the $B^0_{\infty,1}$ norm of the vorticity \emph{linearly} in the Lipschitz norm of the velocity field is then the key (together with the absence of
the stretching term) to get a smallness condition depending upon the size of $\rho_0-1$ only.

When $\g=0$, instead, the equation for the velocity field $u$ becomes
\[
 \d_tu\,+\,u\cdot\nabla u\,+\,\frac{1}{\rho}\,\nabla\Pi\,+\,\frac{\alpha}{\rho}\,u\,=\,0\,,
\]
from which we can deduce an exponential decay with sharp rate $-\alpha_*t$ only in $L^2$, where we have defined $\alpha_*\,:=\,\alpha/\rho^*$, $\rho^*$ being the maximum
value of the density, recall \eqref{eq:vacuum}. When passing to the estimates of high regularity norms, instead, the presence of a variable coefficient $1/\rho$
in the damping term is a source of difficulty, because it creates remainders, which must be treated as forcing term though they do not possess any decay in time.
The basic idea to overcome this difficulty is then to split the damping term as follows, for any $0<\beta\leq \alpha_*$:
\[
 \frac{\alpha}{\rho}\,u\,=\,\beta\,u\,+\,\left(\frac{\alpha}{\rho}\,-\,\beta\right)\,u\,.
\]
From the first term on the right, we can deduce exponential decay similarly to the previous case $\g=1$, whereas for the second term on the right we use the positivity
of its coefficient to keep it into the left-hand side of the estimates and, finally, discard it.
In doing so, however, we cannot forget about commutator terms arising from the estimates for the Besov norms (either endpoint norms in $B^1_{\infty,1}$, or higher regularity
norms in $B^s_{p,r}$): the basic idea is to implement a systematic use of sharp commutator estimates and interpolation in order to make always a (lower order) norm
of the velocity appear, which in turn can be controlled by the $L^2$ norm, for which one disposes of the sharp decay.

To conlude this part, let us mention 
a second major difficulty arising in the case $\g=0$, namely the presence of a \emph{linear} term in the pressure equation.
Because of this term, when $\g=0$ it is no more true that the pressure gradient is quadratic in the velocity field, but only linear, thus preventing us
from using the argument of the case $\g=1$ and making the analysis much more delicate.
In fact, it turns out that the same idea depicted above (based on fine commutator estimates and interpolation, to make an $L^2$ norm of $u$ appear)
allows us to handle this second obstacle and get the desired global in time result also in the instance $\g=0$.

\subsection*{Organisation of the paper}

The paper is organised in the following way. In the next section, we state our main results. In Section \ref{s:tools}, we collect some tools from harmonic analysis,
in particular the Littlewood-Paley theory, paradifferential calculus and transport estimates in Besov spaces.
In Section \ref{s:local} we study local well-posedness for system \eqref{eq:dd-E} and establish a continuation criterion
in terms of the Lipschitz norm of the velocity field only. The analysis of this section is carried out for both cases $\g=1$ and $\g=0$ in parallel,
as no essential differences arise.
Section \ref{s:g=1} is devoted to the proof of the global well-posedness statements related to the case $\g=1$, namely Theorem \ref{th:g=1_d} (in general space dimension) and
Theorem \ref{th:g=1_2} (related to the case $d=2$). Finally, in Section  \ref{s:g=0} we tackle the proof of Theorem \ref{th:g=0_d}, namely of global well-posedness 
for the case $ \g = 0 $.

\subsection*{Notation} \label{s:not}

Before presenting the main results of the work, let us fix some notations. For an interval $I\subset \R$ and a Banach space $X$,
we denote by $\mc C\big(I;X\big)$ the space of continuous bounded functions on $I$ with values in $X$. For any $p\in[1,+\infty]$,
the symbol $L^p\big(I;X\big)$ stands for the space of measurable functions on $I$ such that the map $t\mapsto \left\|f(t)\right\|_{X}$ belongs to $L^p(I)$. We also define $\mc C_b\big(I;X\big)\,=\,\mc C\big(I;X\big)\cap L^\infty\big(I;X\big)$.
Similarly, if $Q\subset \R^{1+d}=\R_t\times\R^d_x$, we set $\mc C_b(Q)\,=\,\mc C(Q)\cap L^\infty(Q)$.

In addition, let $Y$ be the predual of $X$ and denote by $\lan\cdot,\cdot\ran_{X\times Y}$ the duality pairing between $Y$
and $X=Y^*$.
Then we set $\mc C_\weak\big(I;X\big)$ the space of functions $f$ from $I$ to $X$ which are weakly continuous, namely such that, for any $\phi\in Y$,
the map $t\mapsto \lan f(t),\phi \ran_{X\times Y}$ is continuous over $I$.
We use the same notation for scalar-valued and vector-valued functions.


\section*{Acknowledgements}

{\small

The authors are indebted to the anonymous referees, whose careful reading and interesting remarks allowed to improve the presentation of the paper as well as
the formulation of some statements, and to clarify some aspects of our results.

The work of the second author has been partially supported by the project CRISIS (ANR-20-CE40-0020-01), operated by the French National Research Agency (ANR).

}

\section{Main results} \label{s:results}

In this section, we give the precise statement of our main results.
We start by discussing the case $\g=1$, which looks physically more pertinent and, to some extent, simpler. Then, we consider the case $\g=0$,
which is technically more involved. Finally, we focus on the special case of space dimension $d=2$ and choice $\g=1$, for which we are able to give an improvement 
of the previous results.

\subsection{The case ${\g=1}$ when ${d\geq2}$} \label{ss:g=1_d}
Our first result states global well-posedness under a non-linear smallness condition linking the norm of the initial datum and the damping term, and holds
true in any space dimension.

\begin{thm} \label{th:g=1_d}
Let $d\geq 2$ and fix $\g=1$ in system \eqref{eq:dd-E}. Let us fix indices $(s,p,r)\in \R\times[1,+\infty]\times[1,+\infty]$ such that $p\geq2$ and condition
\eqref{cond:Lipschitz} is satisfied. Consider two constants $0<\rho_*<\rho^*$.

There exists a constant $K>0$, only depending on $\rho_*$ and $\rho^*$, and an exponent $\eta=\eta(d)>0$, only depending on the space
dimension $d\geq2$, such that the following property holds true.

Take any initial datum $\big(\rho_0,u_0\big)$ verifying \eqref{eq:in-datum}-\eqref{eq:vacuum} with exactly the previous constants $\rho_*$ and $\rho^*$,
such that $u_0\in L^2\cap B^s_{p,r}$ and $\nabla\rho_0\in B^{s-1}_{p,r}$. Take $\alpha>0$ in equations \eqref{eq:dd-E}. Assume that
\begin{equation} \label{eq:smallness}
\frac{1}{\alpha}\,\left\| u_0\right\|_{B^{1}_{\infty,1}}\,
\exp\left(\left(1\,+\,\left\|\rho_0\,-\,1\right\|^\eta_{B^1_{\infty,1}}\right)\, e^{K}\,\left(\frac{1}{\alpha}\,\| u_0 \|_{L^2} \,+\,1 \right)\right)\,<\,2
\end{equation}

Then, there exists a unique global in time solution $\big(\rho,u,\nabla\Pi\big)$ to system \eqref{eq:dd-E},
related to the initial datum $\big(\rho_0,u_0\big)$, such that:
\begin{enumerate}[(i)]
 \item $\rho\in \mc C_b\big(\R_+\times\R^d\big)$, with $\rho_*\leq\rho(t)\leq\rho^*$ for any $t\in\R_+$, with in addition
$\nabla\rho\in\mc C_\weak\big(\R_+;B^{s-1}_{p,r}\big)$ in the case $p<+\infty$, $\rho\in\mc C_\weak\big(\R_+;B^s_{\infty,r}\big)$ if $p=+\infty$;
\item $u\in\mc C^1\big(\R_+;L^2\big)\cap \mc C_\weak\big(\R_+;B^s_{p,r}\big)$;
\item $\nabla\Pi\in\mc C\big(\R_+;L^2\big)\cap \mc C_\weak\big(\R_+;B^s_{p,r}\big)$.
\end{enumerate}
Moreover the map
 \begin{equation}
 \label{decay:est:1}
t\,\mapsto\, e^{\alpha\,t} \, \left\|u(t)\right\|_{L^2\cap B^s_{p,r}} +  e^{2\,\alpha\,t}\,\left\|\nabla\Pi(t)\right\|_{L^2\cap B^s_{\infty,r}} \qquad \mbox{ is bounded over }\ \R_+\,.
\end{equation}

When $r<+\infty$, the continuity in time with values in Besov spaces holds with respect to the strong topology.
\end{thm}

Observe that the precise value $2$ on the right-hand side of \eqref{eq:smallness} is taken just for convenience, in order to
give a quantitative version of the smallness condition which is required to deduce global well-posedness. However, 
one can take any other constant, up to modifying the value of the constant $K$.

Next, let us formulate a remark about the previous smallness conditions \eqref{eq:smallness}.
\begin{rmk} \label{r:smallness}
Conditions \eqref{eq:smallness} rest only on $\alpha$, on the $B^1_{\infty,1}$ norm of $\rho-1$ and on the $L^2\cap B^1_{\infty,1}$ norm of
$u_0$. In particular, the initial data may have large norms in higher regularity spaces.
\end{rmk}

From Theorem \ref{th:g=1_d}, we immediately deduce the following statement, which tells us that, for any initial datum $\big(\rho_0,u_0\big)$ as above,
then global existence holds, provided the damping coefficient $\alpha$ is large enough.

\begin{cor} \label{c:small-alpha}
Let $d\geq 2$ and fix $\g=1$ in system \eqref{eq:dd-E}. Fix indices $(s,p,r)\in \R\times[1,+\infty]\times[1,+\infty]$ such that $p\geq2$ and condition
\eqref{cond:Lipschitz} is satisfied. Consider two constants $0<\rho_*<\rho^*$ and take an initial datum $\big(\rho_0,u_0\big)$ verifying conditions \eqref{eq:in-datum}
and \eqref{eq:vacuum}, with $u_0\in L^2\cap B^s_{p,r}$ and $\nabla\rho_0\in B^{s-1}_{p,r}$.

Then, there exists $\alpha_0\,=\,\alpha_0\left(\rho_*,\rho^*,d,\|\rho_0-1\|_{B^1_{\infty,1}},\|u_0\|_{L^2\cap B^1_{\infty,1}}\right)\,>\,0$, depending only
on the quantities in the brackets, such that, for any $\alpha>\alpha_0$, there exists a unique global in time solution $\big(\rho,u,\nabla\Pi\big)$
to system \eqref{eq:dd-E}, related to the initial datum $\big(\rho_0,u_0\big)$ and satisfying the regularity properties (i)-(ii)-(iii)
stated in Theorem \ref{th:g=1_d}.
\end{cor}

Of course, also the ``dual'' statement holds true, in which one fixes the value of $\alpha$ and requires the smallness of the norm of the initial datum.
The interesting point of condition \eqref{eq:smallness} is that it imposes smallness only on the $B^1_{\infty,1}$ norm of the initial velocity,
once one fixes the size of the heterogeneity and of the initial kinetic energy.

\begin{cor}
\label{cor:only:besov}
Let $d\geq 2$ and take $\g=1$ and $\alpha>0$ in system \eqref{eq:dd-E}. 
Fix indices $(s,p,r)\in \R\times[1,+\infty]\times[1,+\infty]$ such that $p\geq2$ and condition
\eqref{cond:Lipschitz} is satisfied. For $0<\rho_*<\rho^*$, consider an initial datum $\big(\rho_0,u_0\big)$ verifying the assumptions
of Corollary \ref{c:small-alpha}.

Then, there exists $ c_0\,=\, c_0\left(\rho_*,\rho^*,d,\alpha,\|\rho_0-1\|_{B^1_{\infty,1}},\|u_0\|_{L^2 }\right)\,>\,0$, depending only
on the quantities in the brackets, such that, if $ \| u_0 \|_{B^{1}_{\infty, 1}} < c_0$, there exists a unique global in time solution $\big(\rho,u,\nabla\Pi\big)$
to system \eqref{eq:dd-E}, related to the initial datum $\big(\rho_0,u_0\big)$ and satisfying the regularity properties (i)-(ii)-(iii) of Theorem \ref{th:g=1_d},
together with the decay estimate \eqref{decay:est:1}.
\end{cor}

\subsection{The case ${\g=0}$ for ${d\geq2}$} \label{ss:g=0_d}
In this subsection, we still focus on the case of general dimension $d\geq2$, but for the choice $\g=0$ in the damping coefficient.
We immediately notice that this case is much more involved than the previous one $\g=1$. As a matter of fact, a linear term survives
in the analysis of the pressure gradient, which then is no more expected to be quadratic in $u$, thus to possess a better decay (compare with the decay
property \eqref{decay:est:1} above).

Nonetheless, by using positivity properties of the damping term and careful decay estimates,
one can prove a global well-posedness result also in this case, provided some smallness conditions, more involved than \eqref{eq:smallness},
are satisfied. The precise statement reads as follows.

 \begin{thm} \label{th:g=0_d}
Let $d\geq 2$ and fix $\g=0$ in system \eqref{eq:dd-E}. Let us fix indices $(s,p,r)\in \R\times[1,+\infty]\times[1,+\infty]$ such that $p\geq2$ and condition
\eqref{cond:Lipschitz} is satisfied. Consider two constants $0<\rho_*<\rho^*$.

There exists a constant $K>0$, only depending on $\rho_*$ and $\rho^*$ and on the space dimension $d\geq2$,
and an exponent $\eta=\eta(d)$ only depending on $d$, such that the following property holds true.

Take any initial datum $\big(\rho_0,u_0\big)$ verifying \eqref{eq:in-datum}-\eqref{eq:vacuum} with exactly the previous constants $\rho_*$ and $\rho^*$,
such that $u_0\in L^2\cap B^s_{p,r}$ and $\nabla\rho_0\in B^{s-1}_{p,r}$. Take $\alpha>0$ in equations \eqref{eq:dd-E}.
Define
\[
\mf R\,:=\,1+\left\|\rho_0-1\right\|^\eta_{B^1_{\infty,1}}
\]
and assume that
\[
\frac{\|u_0\|_{L^2\cap B^1_{\infty,1}}}{\alpha} \,K\,\mf R^3\, e^{K \mf R} \,  <\,4\,.
\]

Then, there exists a unique global in time solution $\big(\rho,u,\nabla\Pi\big)$ to system \eqref{eq:dd-E},
related to the initial datum $\big(\rho_0,u_0\big)$, such that the regularity properties (i)-(ii)-(iii) stated in Theorem \ref{th:g=1_d} still hold true.
Moreover, after defining $\theta_0\in\,]0,1[\,$ as the interpolation exponent 
such that $\| f \|_{B^{s-1}_{p,r}}\, \leq\, C\,\| f \|_{L^2}^{\theta_0}\,\| f \|^{1-\theta_0}_{B^{s}_{p,r}}$
and setting
\[
 \beta_0\, :=\, \frac{\theta_0}{1+\theta_0}\; \frac{\alpha}{\rho^*}\,,
\]
one has that the map
\begin{equation}
\label{decay:est:2}
t\,\mapsto\, e^{\beta_0 \,t} \, \left\|u(t)\right\|_{L^2\cap B^s_{p,r}} +  e^{\beta_0 \,t}\,\left\|\nabla\Pi(t)\right\|_{L^2\cap B^s_{\infty,r}} \qquad \mbox{ is bounded over }\ \R_+\,.
\end{equation}
\end{thm}

Of course, the same comments formulated after Theorem \ref{th:g=1_d}, together with Remark \ref{r:smallness} and Corollary \ref{c:small-alpha},
apply also to the previous theorem, obtained in the case $\g=0$. However, for the sake of conciseness, we avoid to formulate them explicitly also here.

\subsection{The case $d=2$ and ${\g=1}$} \label{ss:g=1_2}

Now, we focus on the special case of planar flows, namely $d=2$.
Observe that, for any $\alpha\geq0$, in the limit  $\rho\to1$ equations \eqref{eq:dd-E} formally converge to the classical (homogeneous)
incompressible Euler equations (with relaxation term in the case $\alpha>0$), which are globally well-posed when $d=2$.
Then, it is natural to wonder if, in this case, Theorem \ref{th:g=1_d} can be improved by imposing a smallness condition on the size of
the non-homgeneity $\rho_0-1$ only.

In this section, we give a positive answer to the previous question, but only in the case $\g=1$. The key ingredient is the use of improved transport estimates
\cite{Vis, HK} in Besov spaces of regularity index $s=0$ to control the growth of the vorticity, together with the absence of the stretching term
in its transport equation.
In the case $\g=0$, the problem is that
it is not clear how to reproduce the fundamental cancellation mentioned at the beginning of Subsection \ref{ss:g=0_d} when estimating
the norm of the vorticity in the \emph{critical} space $B^0_{\infty,1}$.

For $\g=1$, the precise result is contained in the following statement.

\begin{thm} \label{th:g=1_2}
Let $d=2$ and fix $\g=1$ in system \eqref{eq:dd-E}. Let us fix indices $(s,p,r)\in \R\times[1,+\infty]\times[1,+\infty]$ such that $p\geq2$ and condition
\eqref{cond:Lipschitz} is satisfied. Fix two constants $0<\rho_*<\rho^*$.

There exists a constant $K>0$, only depending on $\rho_*$ and $\rho^*$, such that the following property holds true.

Take any initial datum $\big(\rho_0,u_0\big)$ verifying \eqref{eq:in-datum}-\eqref{eq:vacuum} with exactly the given constants $0<\rho_*<\rho^*$,
such that $u_0\in L^2\cap B^s_{p,r}$ and $\nabla\rho_0\in B^{s-1}_{p,r}$. In addition, define the function $\Phi_K$ by
\[
\forall\,z\geq0\,,\qquad\qquad
\Phi_K(z)\,:=\,K \,e^{K\, z}\,\left(e^{K\left(\exp\big(K \, z \big)-1\right)} -1 \right)
\]
and assume that
\begin{equation} \label{eq:small_d=2}
\left\|\rho_0\,-\,1\right\|_{B^1_{\infty,1}}\,\left(1\,+\,\left\|\rho_0\,-\,1\right\|^\eta_{B^1_{\infty,1}}\right)\,
\Phi_K\left(\frac{\left\|u_0\right\|_{L^2\cap B^1_{\infty,1}}}{\alpha}\right)
\,<\,4\,,
\end{equation}
where the exponent $\eta>5$ can be chosen\footnote{The exponent $\eta$ comes from an interpolation argument, see inequality \eqref{est:Pi-B_g=1} below.
It only depends on the space dimension, which is fixed here $d=2$.} as close to $5$ as one wants.

Then, there exists a unique global in time solution $\big(\rho,u,\nabla\Pi\big)$ to system \eqref{eq:dd-E},
related to the initial datum $\big(\rho_0,u_0\big)$
and satisfying the regularity properties (i)-(ii)-(iii) stated in Theorem \ref{th:g=1_d}, together with the decay estimate \eqref{decay:est:1} therein.
\end{thm}

Once again, we remark that there is no specific reason to take the critical value $4$ in the inequality defining the smallness condition of the previous statement.
Next, we formulate a remark, similar in spirit to Remark \ref{r:smallness}.

\begin{rmk} \label{r:smallness_d=2}
%
Also in this case, only the low regularity norms $\|\rho_0-1\|_{B^1_{\infty,1}}$ and $\|u_0\|_{L^2\cap B^1_{\infty,1}}$ are involved in the smallness condition.
\end{rmk}

The smallness condition \eqref{eq:small_d=2} has to be interpreted as a smallness condition over the non-homogeneity
$\rho_0-1$: even in the case when the velocity field is large and the damping coefficient $\alpha>0$ is small,
if the non-homogeneity is small enough, then the solution is globally defined.
More precisely, as an immediate consequence of Theorem \ref{th:g=1_2}, we obtain the next result, in the same vein of Corollary \ref{c:small-alpha} above.

\begin{cor} \label{c:small-alpha_d=2}
Let $d= 2$ and fix $\g=1$ and some $\alpha>0$ in system \eqref{eq:dd-E}. Take indices $(s,p,r)\in \R\times[1,+\infty]\times[1,+\infty]$ such that $p\geq2$ and condition
\eqref{cond:Lipschitz} is satisfied. Consider two constants $0<\rho_*<\rho^*$.

For any initial velocity field $u_0\in L^2\cap B^s_{p,r}$, with $\div u_0=0$, there exists a positive constant
$C_{0}\,=\,C_{0}\left(\rho_*,\rho^*,\alpha,\left\|u_0\right\|_{L^2\cap B^1_{\infty,1}}\right)>0$,
depending only on the quantities in the brackets, such that the following fact holds true.

For any initial density $\rho_0$ verifying condition \eqref{eq:vacuum} and  $\nabla\rho_0\in B^{s-1}_{p,r}$, and such that
\[
\left\|\rho_0-1\right\|_{B^1_{\infty,1}}\,\leq\, C_{0}\,,
\]
there exists a unique global in time solution $\big(\rho,u,\nabla\Pi\big)$
to system \eqref{eq:dd-E}, related to the initial datum $\big(\rho_0,u_0\big)$ and satisfying the regularity properties (i)-(ii)-(iii) stated in Theorem \ref{th:g=1_d},
together with the decay property \eqref{decay:est:1}.
\end{cor}

We conclude this part with a remark on condition \eqref{eq:small_d=2}.
\begin{rmk} \label{r:small-second-rem}
Despite \eqref{eq:small_d=2} has to be read as a smallness condition resting on the size (in $B^1_{\infty,1}$) of the non-homogeneity $\rho_0-1$,
we observe that Theorem \ref{th:g=1_2} remains coherent with the claim of Theorem \ref{th:g=1_d}. Indeed,
condition \eqref{eq:small_d=2} implies the global in time well-posedness for possibly large $\rho_0-1$, provided the quantity
$\left\|u_0\right\|_{L^2\cap B^1_{\infty,1}}/\alpha$ is small enough:
this is an immediate consequence of the fact that $\lim_{z\to0^+}\Phi_K(z)\,=\,0$.
\end{rmk}

\subsection{A few comments on the main results} \label{ss:comments}

The rest of the work is devoted to the proof of the previous statements. 
Before continuing, let us make a few comments on the statements presented above.

\bigbreak
\paragraph*{Scaling.} We observe that equations \eqref{eq:dd-E} are invariant under a $2$-parameter family of scalings.
More precisely, for $\g\in\{0,1\}$, assume to dispose of a solution $\big(\rho,u,\nabla\Pi\big)$ to \eqref{eq:dd-E} with $\alpha=1$.
For $\alpha>0$, $R>0$ and $\ell>0$, define the rescaled solution $\big(\rho_\resc, u_\resc, \nabla\Pi_\resc\big)$ as
\begin{equation} \label{eq:scaling}
\left\{
\begin{array}{l}
\rho_\resc(t,x)\,=\,R\,\rho(R^{\g-1}\alpha t, \ell x)\,, \\[2ex]
u_\resc(t,x)\,=\,\dfrac{R^{\g-1}\,\alpha}{\ell}\,u(R^{\g-1}\alpha t, \ell x)\,, \\[2ex]
\Pi_\resc(t,x)\,=\,\dfrac{R^{2\g-1}\,\alpha^2}{\ell^2}\,\Pi(R^{\g-1}\alpha t, \ell x)\,.
\end{array}
\right.
\end{equation}
Then, for a fixed $\alpha>0$, the rescaled solution $\big(\rho_\resc, u_\resc, \nabla\Pi_\resc\big)$
solves, for any choice of parameters $R>0$ and $\ell>0$, system \eqref{eq:dd-E} with the same value of $\g$ and with the given $\alpha>0$.

Notice that the $L^1_T(L^\infty)$ norm of $\nabla u$ (coming into play, for instance, in the continuation criterion of Lemma \ref{lem:cont:pri})
is invariant for such family of scalings. We now discuss the invariance of the smallness conditions of Theorems \ref{th:g=1_d}, \ref{th:g=0_d} and \ref{th:g=1_2}
under a suitable choice of the scaling parameters.

\paragraph*{Relation between the smallness conditions and the scaling.}
In \eqref{eq:scaling}, the presence of the factor $R>0$ in front of the density looks quite unnatural, as its value drastically modifies the size of the non-homogeneity.
At the same time, in our study we need to work with non-homogeneous Besov spaces, which are not well-behaved with respect to a rescaling of the space variable
because of the presence of the low-frequency terms.

The above considerations motivate us to fix the choice $R=\ell=1$ in \eqref{eq:scaling}. Therefore, given a solution $\big(\rho,u,\nabla\Pi\big)$ as above and fixed
$\alpha>0$,
we define the rescaled solution $\big(\rho_\alpha,u_\alpha,\nabla\Pi_\alpha\big)$ as
\begin{equation*} 
\left\{
\begin{array}{l}
\rho_\alpha(t,x)\,=\,\rho(\alpha t, x)\,, \\[1ex]
u_\alpha(t,x)\,:=\,\alpha\,u(\alpha t, x)\,, \\[1ex]
\Pi_\alpha(t,x)\,=\,\alpha^2\,\Pi(\alpha t, x)\,,
\end{array}
\right.
\end{equation*}
which then solves \eqref{eq:dd-E} with damping coefficient equal to $\alpha$.

This having been pointed out, we notice that the smallness conditions formulated in Theorems \ref{th:g=1_d}, \ref{th:g=0_d} and \ref{th:g=1_2}
are invariant with respect to the above scaling. Besides, this indicates that our smallness conditions are somehow sharp.

\medbreak
Interestingly, we observe that, if we were to replace the non-homogeneous Besov norms
appearing in Theorems \ref{th:g=1_d}, \ref{th:g=0_d} and \ref{th:g=1_2} with the corresponding homogeneous ones
(and if we were to forget about the $L^2$ norm of $u_0$ appearing in the smallness conditions), we could allow for more freedom in the chosen scaling.

More precisely, the restriction $\ell=1$ would not be needed in Theorems \ref{th:g=1_d} and \ref{th:g=0_d}, as the smallness conditions there
would be invariant under the scaling \eqref{eq:scaling} with $R=1$ (actually, for any $R>0$ in the case of Theorem \ref{th:g=1_d}).
In addition, in the situation of Theorem \ref{th:g=1_2} (hence, in particular, $\g=1$) one could take any $\ell>0$ and $R=\ell^{-1}$ in \eqref{eq:scaling}
and see that the smallness condition of that statement would remain invariant for that choice of the scaling.

\paragraph*{Vorticity.}
Let us mention that, although the vorticity of the fluid does not appear in the formulation of the main results,
this quantity plays an important role in our analysis. Therefore, let us spend some lines to introduce it.

Given an incompressible velocity field $u$, we define the vorticity matrix  $ \Omega $ of $u$ as twice the skew-symmetric part of its Jacobian matrix,
namely $\Omega\,:=\,Du\,-\,\nabla u$,
where $Du$ is the Jacobian matrix of $u$ and $\nabla u\,=\,^tDu$ is its transpose matrix.
We adopt the convention of identifying $\Omega$ with the vector field $\o\,:=\,\nabla\times u$ when $d=3$ and, when $d=2$,
with the scalar field $\o\,:=\,\d_1u^2\,-\,\d_2u^1$.

\medbreak
This having been recalled, we can move on and collect, in the next section, some preliminary material which is needed in our study.


\section{Toolbox} \label{s:tools}

The present section is intended to be a functional toolbox, which we stock with all the material which is needed in the course of the proof of the main results.
We start by recalling some basic facts of Littlewood-Paley theory in $\R^d$, with $d\geq1$,
the definition of non-homogeneous Besov spaces and some interpolation inequalities of Gagliardo-Nirenberg type. Then, in Subsection \ref{ss:para}
we introduce some elements of paradifferential calculus, namely paraproduct decomposition and estimates for composition of Besov functions.
Finally, in Subsection \ref{ss:tools-est} we recall some well-known estimates for smooth solutions to transport and elliptic equations.

\subsection{Littlewood-Paley theory and Besov spaces} \label{ss:LP-Sobolev}

We recall here the main ideas of Littlewood-Paley theory in $\R^d$. These are classical results, for which,
if not otherwise specified, we refer to Chapter 2 of \cite{BCD} for details.

To begin with,
we fix a smooth radial function $\chi$ supported in the ball $B(0,2)$, equal to $1$ in a neighborhood of $B(0,1)$
and such that $r\mapsto\chi(r\,e)$ is nonincreasing over $\R_+$ for all unitary vectors $e\in\R^d$. Set
$\varphi\left(\xi\right)=\chi\left(\xi\right)-\chi\left(2\xi\right)$ and
$\vphi_j(\xi):=\vphi(2^{-j}\xi)$ for all $j\geq0$.
The dyadic blocks $(\Delta_j)_{j\in\Z}$ are defined by\footnote{Throughout, we agree  that  $f(D)$ stands for 
the pseudo-differential operator $u\mapsto\mc{F}^{-1}[f(\xi)\,\what u(\xi)]$.} 
$$
\Delta_j\,:=\,0\quad\mbox{ if }\; j\leq-2,\qquad\Delta_{-1}\,:=\,\chi(D)\qquad\mbox{ and }\qquad
\Delta_j\,:=\,\varphi(2^{-j}D)\quad \mbox{ if }\;  j\geq0\,.
$$
We  also introduce the following low frequency cut-off operator:
\begin{equation} \label{eq:S_j}
S_ju\,:=\,\chi(2^{-j}D)\,=\,\sum_{k\leq j-1}\Delta_{k}\qquad\mbox{ for }\qquad j\geq0\,.
\end{equation}
Note that $S_j$ and $\Delta_j$ are convolution operators by $L^1$ kernels, whose norms are independent of the index $j$. Thus, they act continuously
from $L^p$ into itself, for any $p\in[1,+\infty]$.

Remark that the function $\chi$ can be chosen so that, for all $\xi\in\R^d$, one has the equality $\chi(\xi)+\sum_{j\geq0}\vphi_j(\xi)=1$.
Based on this partition of unity on the Fourier space, we obtain the so-called non-homogeneous ``Littlewood-Paley decomposition'' of tempered distributions:
\[
\forall\,u\,\in\,\mc S'\,,\qquad\qquad\qquad u\,=\,\sum_{j\geq-1}\Delta_ju\qquad \mbox{ in the sense of }\quad \mc S'\,.
\]

Let us also mention the so-called \emph{Bernstein inequalities}, which explain the way derivatives act on spectrally localised functions.
They play a key role in the whole theory.

\begin{lemma} \label{l:bern}
Let  $0<r<R$.   A constant $C>0$ exists so that, for any integer $k\geq0$, any couple $(p,q)$ 
in $[1,+\infty]^2$, with  $p\leq q$,  and any function $u\in L^p$, for all $\lambda>0$ we have:
\begin{align*}
{\supp}\, \widehat u\, \subset\,   B(0,\lambda R)\,=\,\big\{\xi\in\R^d\,\big|\,|\xi|\leq\lambda R \big\}\qquad
\Longrightarrow\qquad
\|\nabla^k u\|_{L^q}\, \leq\,
 C^{k+1}\,\lambda^{k+d\left(\frac{1}{p}-\frac{1}{q}\right)}\,\|u\|_{L^p} \\[1ex]
{\supp}\, \widehat u   \, \subset\, \big\{\xi\in\R^d\,\big|\, r\lambda\leq|\xi|\leq R\lambda\big\}
\quad\Longrightarrow\quad C^{-k-1}\,\lambda^k\|u\|_{L^p}\,\leq\,
\|\nabla^k u\|_{L^p}\,
\leq\,C^{k+1} \, \lambda^k\|u\|_{L^p}\,.
\end{align*}
\end{lemma}

By use of Littlewood-Paley decomposition, we can define the class of Besov spaces.
\begin{defin} \label{d:B}
  Let $s\in\R$ and $1\leq p,r\leq+\infty$. The \emph{non-homogeneous Besov space}
$B^{s}_{p,r}\,=\,B^s_{p,r}(\R^d)$ is defined as the subset of tempered distributions $u$ for which
$$
\|u\|_{B^{s}_{p,r}}\,:=\,
\left\|\left(2^{js}\,\|\Delta_ju\|_{L^p}\right)_{j\geq -1}\right\|_{\ell^r}\,<\,+\infty\,.
$$
\end{defin}
Besov spaces are interpolation spaces between the more classical Sobolev spaces. In fact, for any $k\in\N$ and~$p\in[1,+\infty]$,
we have the following chain of continuous embeddings:
$$
B^k_{p,1}\hookrightarrow W^{k,p}\hookrightarrow B^k_{p,\infty}\,,
$$
where  $W^{k,p}$ stands for the classical Sobolev space of $L^p$ functions with all the derivatives up to the order $k$ in $L^p$.
When $1<p<+\infty$, we can refine the previous result (see Theorems 2.40 and 2.41 in \cite{BCD}): we have
$B^k_{p, \min (p, 2)}\hookrightarrow W^{k,p}\hookrightarrow B^k_{p, \max(p, 2)}$.
In particular, for all $s\in\R$, we deduce the equivalence $B^s_{2,2}\equiv H^s$, with equivalence of norms. 
In addition, for any $k\in\N$ and any $\veps\in\,]0,1[\,$, one has the equivalence $B^{k+\veps}_{\infty,\infty}\equiv \mc C^{k,\veps}$
with the classical H\"older classes, with equivalence of norms.

As an immediate consequence of the first Bernstein inequality, one gets the following embedding result.
\begin{prop}\label{p:embed}
The space $B^{s_1}_{p_1,r_1}$ is continuously embedded in the space $B^{s_2}_{p_2,r_2}$ for all indices satisfying $p_1\,\leq\,p_2$ and
$$
s_2\,<\,s_1-d\left(\frac{1}{p_1}-\frac{1}{p_2}\right)\qquad\qquad\mbox{ or }\qquad\qquad
s_2\,=\,s_1-d\left(\frac{1}{p_1}-\frac{1}{p_2}\right)\;\;\mbox{ and }\;\;r_1\,\leq\,r_2\,. 
$$
\end{prop}

In particular, we get the following chain of continuous embeddings:
\begin{equation*}
B^s_{p,r} \hookrightarrow B^{s - \frac{d}{p}}_{\infty, r} \hookrightarrow B^0_{\infty, 1} \hookrightarrow L^\infty\,,
\end{equation*}
provided that the triplet $(s, p, r) \in \mathbb{R} \times [1, +\infty]^2$ satisfies the condition
\begin{equation} \label{eq:AnnLInfty}
s > \frac{d}{p} \qquad\qquad \text{ or } \qquad\qquad s = \frac{d}{p}\quad \text{ and }\quad r = 1\,.
\end{equation}
Similarly, as announced at the beginning of the manuscript, for $(s,p,r)$ verifying \eqref{cond:Lipschitz}, we get $B^s_{p,r}\,\hookrightarrow\,W^{1,\infty}$.

To conclude this part, we recall some useful inequalities by means of Littlewood-Paley decomposition.
They are all consequences of a unique simple technique, which consists in cutting a tempered distribution into low and high frequencies and then optimising
the size $N$ at which realising the cut. 

\begin{lemma} \label{l:GN-ineq} 
Let $d$ be the space dimension and let $ p_0 > d $. For any function $ f \in \mathcal{S} $, the following Gagliardo-Nirenberg type inequalities hold true:
 \[
\| f \|_{L^\infty}\,\lesssim\,\| f\|_{L^2}^{\alpha_1}\,\|\nabla f \|^{1-\alpha_1}_{L^\infty} \qquad \mbox{ and }\qquad
\|\nabla f \|_{L^\infty}\,\lesssim\,\|\nabla f\|_{L^2}^{\alpha_2}\,\|\Delta f \|^{1-\alpha_2}_{L^{p_0}}\,,
\]
 for two suitable exponent $ \alpha_{1,2} \in \,]0,1[\, $.
The above inequalities extend also to vector-valued functions. In addition, for any vector field $ u \in \mathcal{S}  $ such that $ \div(u) = 0 $, one has
\[
\|u\|_{L^\infty}\,\lesssim\,\left\|u\right\|_{L^2}^{\alpha_3}\,\left\|\Omega\right\|_{L^{p_0}}^{1-\alpha_3}
\]
 where  $ \Omega\,:=\,Du\,-\,^t\nabla u $ denotes the vorticity matrix of $ u $ and $ \alpha_{3} \in \,]0,1[\, $ is a suitable   exponent.
\end{lemma}

\subsection{Paradifferential calculus} \label{ss:para}

We now apply Littlewood-Paley decomposition to state some useful results from paradifferential calculus. Again, we refer to Chapter 2
of \cite{BCD} for details and further results.

To begin with, let us introduce the \emph{paraproduct operator} (after J.-M. Bony \cite{Bony}).
Constructing the paraproduct operator relies on the observation that, 
formally, any product of two tempered distributions $u$ and $v$ may be decomposed into the sum
\begin{equation}\label{eq:bony}
u\,v\;=\;\mathcal{T}_uv\,+\,\mathcal{T}_vu\,+\,\mathcal{R}(u,v)\,,
\end{equation}
where we have defined
$$
\mathcal{T}_uv\,:=\,\sum_jS_{j-1}u\,\Delta_j v,\qquad\qquad\mbox{ and }\qquad\qquad
\mathcal{R}(u,v)\,:=\,\sum_j\sum_{|k-j|\leq1}\Delta_j u\,\Delta_{k}v\,.
$$
The above operator $\mc T$ is called ``paraproduct'' whereas
$\mc R$ is called ``remainder''.
The paraproduct and remainder operators have many nice continuity properties over Besov spaces, which are collected in the next statement.
\begin{prop}\label{p:op}
For any $(s,p,r)\in\R\times[1,+\infty]^2$ and $t>0$, the paraproduct operator 
$\mathcal{T}$ maps continuously $L^\infty\times B^s_{p,r}$ in $B^s_{p,r}$ and  $B^{-t}_{\infty,\infty}\times B^s_{p,r}$ in $B^{s-t}_{p,r}$.
Moreover, the following estimates hold:
\[
\|\mathcal{T}_uv\|_{B^s_{p,r}}\,\lesssim\,\|u\|_{L^\infty}\,\|\nabla v\|_{B^{s-1}_{p,r}}\qquad\mbox{ and }\qquad
\|\mathcal{T}_uv\|_{B^{s-t}_{p,r}}\,\lesssim\,\|u\|_{B^{-t}_{\infty,\infty}}\,\|\nabla v\|_{B^{s-1}_{p,r}}\,.
\]

For any $(s_1,p_1,r_1)$ and $(s_2,p_2,r_2)$ in $\R\times[1,+\infty]^2$ such that 
$s_1+s_2>0$, $1/p:=1/p_1+1/p_2\leq1$ and~$1/r:=1/r_1+1/r_2\leq1$,
the remainder operator $\mathcal{R}$ maps continuously~$B^{s_1}_{p_1,r_1}\times B^{s_2}_{p_2,r_2}$ into~$B^{s_1+s_2}_{p,r}$.
In the case $s_1+s_2=0$, if in addition $r=1$, the operator $\mathcal{R}$ is continuous from $B^{s_1}_{p_1,r_1}\times B^{s_2}_{p_2,r_2}$ with values
in $B^{0}_{p,\infty}$.
\end{prop}

The consequence of this proposition is that, when $s>0$, the product acts continuously on the spaces $B^s_{p,r}\cap L^\infty$;
in particular, the product is continuous from the space $B^s_{p,r}\times B^s_{p,r}$ into $B^s_{p,r}$
as long as condition \eqref{eq:AnnLInfty} holds with $s > 0$.
Moreover, in that case, we have the so-called \emph{tame estimates}.

\begin{cor}\label{c:tame}
Let $(s,p, r)\in\R\times[1,+\infty]^2$ be such that that $s > 0$. Then, for any $f$ and $g$ belonging to $L^\infty\cap B^s_{p,r}$, the product
$fg$ also belongs to that space and we have
\begin{equation}
\label{alg:prop:2}
\left\| f\,g \right\|_{B^s_{p,r}}\, \lesssim \,\| f \|_{L^\infty}\, \|g\|_{B^s_{p,r}}\, +\, \| f \|_{B^s_{p,r}} \,\| g \|_{L^\infty}\,.
\end{equation}
\end{cor}

We remark that in the spaces $B^0_{p,r}$ the product is \emph{not} continuous, as the continuity of the remainder operator $\mc R(f,g)$ with values in $B^0_{p,r}$ fails in this case. However, following the analysis in \cite{D:F}, we remark
some information can still be obtained in some special cases, which we now highlight. 
We focus on the case $B^0_{\infty,1}$, in the end the only relevant one for our scopes.

First of all, by requiring some better regularity for $g$, we can replace estimate \eqref{alg:prop:2} by \tsl{e.g.}
\begin{equation}
\label{algebra:B^0}
\| f g\|_{B^{0}_{\infty,r}}\, \lesssim\, \| f \|_{B^{0}_{\infty,1}} \,\| g\|_{B^{1/2}_{\infty,1}}\,.
\end{equation}
On the other hand, we can use some special structure in presence of differentiation. More precisely, for a function $ f \in B^{1}_{\infty,1}$ and a vector field
$ G \in B^{0}_{\infty,1} $ such that $ \div(G) = 0 $, we can use the Bony decomposition \eqref{eq:bony} and write
\begin{equation*}
\nabla f \cdot G\, =\, \div(f\,G)\, =\, \mathcal{T}_{\nabla f} G\, +\, \mathcal{T}_{G} \nabla f\, +\, \div\big(\mathcal{R}(f, G)\big)\,.
\end{equation*}
Thus, by using
\begin{equation*}
\left\| \div\big(\mathcal{R}(f, G)\big)\right\|_{B^{0}_{\infty,1}}\, \lesssim\, \| \mathcal{R}(f, G) \|_{B^{1}_{\infty,1}}\, \lesssim\,
\| f \|_{B^{1}_{\infty,1}}\, \|  G \|_{B^{0}_{\infty,1}}\,, 
\end{equation*}
it is easy to deduce that
\begin{equation}
\label{algebra:B^0_2}
\| \nabla f \cdot G \|_{B^{0}_{\infty,1}} \, \lesssim \, \| f \|_{B^{1}_{\infty,1}}\, \|  G \|_{B^{0}_{\infty,1}}\,.
\end{equation}
Inequalities \eqref{algebra:B^0} and \eqref{algebra:B^0_2} will be needed in the computations of Subsection \ref{ss:g=1_d=2}.

Next, we recall the following result, in the same spirit of the classical Meyer's \emph{paralinearisation} theorem.
It will be important when comparing the Sobolev norms of the density $\rho$ with the ones of its inverse $1/\rho$ and of $\log\rho$.
The proof can be found in \cite{D1}.
\begin{prop}\label{p:comp}
Let $I$ be an open  interval of $~\R$ and let $F:I\rightarrow\R$ be a smooth function. 

Then for all compact subset $J\subset I$, all $s>0$ and all $(p,r)\in[1,+\infty]$, there exists a constant $C>0$
such that, for any function $a$ valued in $J$ and with gradient in $B^{s-1}_{p,r}$,  we have
$\nabla(F\circ a)\in B^{s-1}_{p,r}$ together with the estimate
$$
\|\nabla(F\circ a)\|_{B^{s-1}_{p,r}}\,\leq\,C\,\|\nabla a\|_{B^{s-1}_{p,r}}\,.
$$
\end{prop}

Finally, we need a few commutator estimates. Recall that, given two operators $A$ and $B$, we denote by $[A,B]\,:=\,AB-BA$ their commutator operator.
The first commutator estimate corresponds to Lemma 2.100 of \cite{BCD} (see also Remark 2.101 therein).
\begin{lemma}\label{l:CommBCD}
Let $s>0$ and $(s,p,r)\in\R\times[1,+\infty]^2$ such that condition \eqref{cond:Lipschitz} is satisfied. 
Let $v$ be a vector field on $\R^d$, belonging to $B^s_{p,r}$. Then, for any $f\in B^s_{p,r}$, one has
\begin{equation*}
2^{js} \left\| \big[ v \cdot \nabla, \Delta_j \big] f  \right\|_{L^p}\,\lesssim\,
c_j\, \Big( \|\nabla v \|_{L^\infty}\, \| f \|_{B^s_{p,r}} \,+\, \|\nabla v \|_{B^{s-1}_{p,r}}\, \|\nabla f \|_{L^\infty} \Big)\,,
\end{equation*}
for some sequence $\big(c_j\big)_{j\geq -1}$ belonging to the unit ball of $\ell^r$. 
\end{lemma}

The second commutator estimate we need involves the commutator of a paraproduct operator with a Fourier multiplier. It is contained in Lemma 2.99 of \cite{BCD},
but we report it here for the reader convenience.
\begin{lemma}\label{l:ParaComm}
Let $\psi$ be a smooth function on $\mathbb{R}^d$, which is homogeneous of degree $m$ away from a neighbourhood of $\,0$.
Then, for any vector field $v$ such that $\nabla v \in L^\infty$ and for any $s\in\R$, one has:
\begin{equation*}
\forall\, f \in B^s_{p,r}\,, \qquad\qquad \left\| \big[ \mathcal{T}_v, \psi(D) \big] f \right\|_{B^{s-m+1}_{p,r}}\, \lesssim\,
\|\nabla v\|_{L^\infty} \|f\|_{B^s_{p,r}}\,.
\end{equation*}
\end{lemma}

The third commutator estimate, contained in the next statement, is used to control the damping term. Although this estimate seems classical,
we were not able to find a reference for it; so we briefly explain its proof in the Appendix.  

\begin{lemma}\label{com:damping}
Let $ (s,p,r) \in \mathbb{R} \times  [1,+\infty]^2 $ satisfying \eqref{cond:Lipschitz}.

Then, for any vector field $v$ on $\R^d$
such that $v\in B^{s-1}_{p,r} $ and any $ f \in L^{\infty}(\mathbb{R}^d)$ such that $ \nabla f \in B^{s-1}_{p,r} $, one has the estimate
\[
2^{js} \left\| \big[ f , \Delta_j \big] v  \right\|_{L^p} \,\lesssim\,
c_j\, \Big( \|f\|_{B^1_{\infty,1}} \, \| v \|_{B^ {s-1}_{p,r}} + \| \nabla f \|_{B^{s-1}_{p,r}}\,\| v \|_{L^\infty} \Big)\,, 
\]
for a suitable sequence $\big(c_j\big)_{j\geq -1}$ belonging to the unit ball of $\ell^r$. The (implicit) multiplicative constant intervening
in the previous inequality is ``universal'', namely it does not depend on the vector field $v$, nor on the function $f$.
\end{lemma}

\subsection{Elliptic equations with variable coefficients} \label{ss:elliptic}

Here, we turn our attention to the study of the following elliptic equation:
\begin{equation}\label{eq:elliptic}
-\,\div \left(\frac{1}{\rho}\,\nabla\Pi \right)\,=\,\div (F) \qquad\qquad\qquad \mbox{ in }\quad \R^d\,,
\end{equation}
where $ \rho = \rho (x) $ is a given regular bounded function satisfying 
\[ 
a_*\,:=\,\inf_{x\in\R^d} \frac{1}{\rho}\,>\,0\,.
\] 
We shall use  the following  result,
based on Lax-Milgram's theorem (this is Lemma 2 in \cite{D1}). 
\begin{lemma}\label{l:laxmilgram}
For all vector field $F$ with coefficients in $L^2$, there exists a tempered distribution $\Pi$,
unique up to  constant functions, such that  $\nabla\Pi\in L^2$ and  
equation \eqref{eq:elliptic} is satisfied. 
In addition, we have 
$$
\|\nabla\Pi\|_{L^2}\,\leq\, \frac{1}{a_*}\, \|F\|_{L^2}\,.
$$
\end{lemma}

Next, we discuss how to derive higher regularity estimates for equation \eqref{eq:elliptic} in Besov spaces. Propagation of higher regularity for the pressure
for the non-homogeneous incompressible Euler system \eqref{eq:dd-E} (actually, in the case $\alpha=0$)
has been obtained in \cite{D1, D:F, F_2012}. More precisely, one can prove that
\begin{equation}
\label{gen:pres:ine}
\|\nabla \Pi \|_{B^{s}_{p,r}}\, \lesssim\,\left( 1 +  \| \nabla \rho  \|_{B^{s-1}_{p,r}}^{\eta} \right)\,\|F\|_{L^2}\,+\,\left\| \rho\,\div(F)\right\|_{B^{s-1}_{p,r}}\,.
\end{equation}
However, in view of applications to the continuation criterion, which will play a key role in our analysis,
we have to use more precise estimates, in the same spirit of the bounds exhibited in \cite{F-GB-S, F-V} for a different (yet related) non-homogeneous fluid system.

Let us collect the needed estimates in the next statement, whose proof is postponed to Appendix \ref{a:app}.
It goes without saying that those estimates, together with an interpolation argument, imply \eqref{gen:pres:ine} above.

\begin{lemma}
\label{lem:pres:est:Besov}
Fix indices $(s,p,r)\in \R\times[1,+\infty]\times[1,+\infty]$ such that $p\geq2$ and condition \eqref{cond:Lipschitz} is satisfied. Then,
there exists an exponent $\eta>0$, only depending on the triplet $(s,p,r)$, such that the solution to equation \eqref{eq:elliptic} satisfies
the following estimate:
\begin{equation*}
\|\nabla \Pi \|_{B^{s}_{p,r}}\, \lesssim\,\left( 1 +  \| \nabla \rho  \|_{L^\infty}^{\eta} \right)\,\|F\|_{L^2}\,+\,
 \| \nabla \rho  \|_{B^{s-1}_{p,r}}\,\left\|\nabla\Pi\right\|_{L^\infty}\,+\, \left\| \rho\,  \div(F)  \right\|_{B^{s-1}_{p,r}}\,,
\end{equation*}
where the (implicit) multiplicative constant only depends on $ a_{*} $ and $ a^{*}\,:=\,\sup_{x\in\R^d}1/\rho(x) $.
\end{lemma}

\subsection{Transport estimates in Besov spaces} \label{ss:tools-est}

System \eqref{eq:dd-E} has essentially the structure of a coupling of transport equations, at least under condition \eqref{eq:vacuum}.
As a matter of fact, in our study we will make a great use of \tsl{a priori} estimates in Besov spaces for smooth solutions to linear transport equations,
which we collect in this subsection.
Most of the material presented here can be found in Chapter 3 of \cite{BCD}. 

Let us consider the initial value problem
\begin{equation}\label{eq:TV}
\left\{\begin{array}{l}
\partial_t f\, +\, v \cdot \nabla f\, = \,g \\[1.5ex]
f_{|t = 0}\, =\, f_0\,.
\end{array}\right.
\end{equation}
Throughout this part, the velocity field $v=v(t,x)$ will always be assumed to be divergence-free, \tsl{i.e.} $\div v = 0$,
and Lipschitz continuous with respect to the space variable.

The following statement contains the conclusions of Theorems 3.14 and 3.19 of \cite{BCD}, in the case when the transport field $v$ is Lipschitz-continuous.
\begin{thm}\label{th:transport}
Let $(s,p,r)\in\R\times[1,+\infty]^2$ satisfy the Lipschitz condition \eqref{cond:Lipschitz} and $T>0$ be fixed.
Assume that $v$ is a divergence-free vector field belonging to $L^1\big([0,T];B^s_{p,r}\big)$ such that, for some
$q > 1$ and $M > 0$, $v \in L^q\big([0,T];B^{-M}_{\infty, \infty}\big)$.
Finally, let $\s\in\R$ be such that
\[
\s\,\geq\,-\,d\,\min\left\{\frac{1}{p}\,,\,\frac{1}{p'}\right\}\qquad\qquad \mbox{ or, \ \ \ \ \  if }\ \div v=0\,,\quad 
\s\,\geq\,-\,1\,-\,d\,\min\left\{\frac{1}{p}\,,\,\frac{1}{p'}\right\}\,.
\]

Then, for any external force $g \in L^1\big([0,T];B^\s_{p,r}\big)$ and any initial datum $f_0 \in B^\s_{p,r}$, the transport equation \eqref{eq:TV} has a unique solution $f$ in the space:
\begin{itemize}
\item $\mc C\big([0,T];B^\s_{p,r}\big)$ if $r < +\infty$;
\item $\left( \bigcap_{\s'<\s} \mc C\big([0,T];B^{\s'}_{p,\infty}\big) \right) \cap \mc C_{\weak}\big([0,T];B^\s_{p, \infty}\big)$, in the case $r = +\infty$.
\end{itemize}
Moreover, after defining the function $V(t)$ as
\[
V(t)\,:=\,\int^t_0\left\|\nabla v(\t)\right\|_{B^{s-1}_{p,r}}\,\dd\t\,,
\]
the unique solution $f$ satisfies the following estimate, for a suitable universal constant $C>0$:
\begin{equation*} 
\forall\,t\in[0,T]\,,\qquad
\| f(t) \|_{B^\s_{p,r}}\, \leq\, e^{C\,V(t)}\,\left(\| f_0 \|_{B^\s_{p,r}} + \int_0^t e^{-C\,V(\tau)}\, \| g(\t) \|_{B^\s_{p,r}} \dd\t\right)\,.
\end{equation*}
In the case when $v=f$, the previous estimate holds true with $V'(t)\,=\,\left\|\nabla f(t)\right\|_{L^\infty}$.
\end{thm}

Next, let us recall improved transport estimates for Besov spaces whose index of regularity is $\s=0$. Those estimates, discovered by Vishik in \cite{Vis}
and later proved again \tsl{via} a different argument by Hmidi and Keraani in \cite{HK}, establish that the $B^0_{p,r}$ norm of the solution
grows \emph{linearly}, and not exponentially as in Theorem \ref{th:transport}, with the Lipschitz norm of the transport field $v$.
We refer to Theorem 3.18 of \cite{BCD} for the statement presented here.
\begin{thm}\label{th:B^0}
Assume that the velocity field $v$ satisfies $\nabla v \in L^1_T(L^\infty)$ and $\div v=0$. Let $(p,r) \in [1, +\infty]^2$.

Then there exists a constant $C = C(d)>0$, only depending on the space dimension $d\geq2$, such that, for any solution $f$ to problem \eqref{eq:TV} in
$\mc C\big([0,T];B^0_{p,r}\big)$, or in $\mc C_w\big([0,T];B^0_{p,\infty}\big)$ if $r=+\infty$, we have
\[ 
\| f \|_{L^\infty_T(B^0_{p, r})}\, \leq\, C\, \left( \| f_0 \|_{B^0_{p, r}}\, +\, \| g \|_{L^1_T(B^0_{p, r})}\right)\;
\left( 1+\int_0^T\| \nabla v(\tau) \|_{L^\infty}\,{\rm d} \tau \right)\,.
\] 
\end{thm}

Theorem \ref{th:B^0} plays a key role in the proof of Theorem \ref{th:g=1_2}, namely global existence when $\g=1$ and $d=2$, with a smallness condition
on the Besov norm of the non-homogeneity only.

\section{Local well-posedness} \label{s:local}

In this section, as a first step towards the proof of the global existence results,
we investigate local well-posedness of the Cauchy problem \eqref{eq:dd-E}-\eqref{eq:in-datum}, together with the validity of a suitable continuation criterion.
We state our main results of this part (see Proposition \ref{p:local} and Lemma \ref{lem:cont:pri}) for $\alpha\geq0$, but we remark that they actually hold true
for any $\alpha\in\R$. Indeed, the term $\alpha \rho^\g u$ is treated as a forcing term in the arguments below, so its sign does not play any role.

To begin with, we state the following local existence and uniqueness result.

\begin{prop} \label{p:local}
Let $\alpha \geq 0$ and $\g\in\{0,1\}$. Fix indices $(s,p,r)\in\R\times[1,+\infty]\times[1,+\infty]$ such that $p>1$ and one of the two conditions in \eqref{cond:Lipschitz}
is verified. Take an initial datum $\big(\rho_0,u_0\big)$ such that conditions \eqref{eq:in-datum} and \eqref{eq:vacuum} are satisfied, for two suitable constants
$0<\rho_*\leq\rho^*$. Assume in addition that $\nabla\rho_0\in B^{s-1}_{p,r}$ and that $u_0\in L^2\cap B^s_{p,r}$.

Then, there exist a positive time $T>0$ and a unique solution $\big(\rho,u,\nabla\Pi\big)$ to equations \eqref{eq:dd-E}-\eqref{eq:in-datum}-\eqref{eq:vacuum}
on $[0,T]\times\R^d$ such that:
\begin{enumerate}[(i)]
 \item $\rho\in \mc C_b\big([0,T]\times\R^d\big)$, with $\rho_*\leq\rho(t)\leq\rho^*$ for any $t\in[0,T]$ and $\nabla\rho\in\mc C_{\weak}\big([0,T];B^{s-1}_{p,r}\big)$;
 \item $u\in \mc C^1\big([0,T];L^2\big)\,\cap\,\mc C_\weak\big([0,T];B^{s}_{p,r}\big)$;
 \item $\nabla \Pi\in \mc C\big([0,T];L^2\big)\,\cap\,\mc C_\weak\big([0,T];B^{s}_{p,r}\big)$.
\end{enumerate}
The time continuity with values in Besov spaces holds with respect to the strong topology if $r<+\infty$.
\end{prop}

The previous well-posedness statement is proved in \cite{D1, D:F} for $ \alpha = 0 $; for $ \alpha>0$ it easily follows from the same arguments of those papers,
for instance by looking at the damping term
$\mf D^\g_\alpha(\rho,u)\,=\,\alpha\,\rho^\g\,u$
as a forcing term in the equations, and implementing an iterative approximation procedure.

The goal of the present paper is to show that, in the case $\alpha>0$, the previous solutions are global, namely that $T=+\infty$ in Proposition \ref{p:local}. 
In this respect, we notice that a global $L^2$ bound for $u$ and global $L^\infty$ bounds for $\rho^{\pm1}$ are easily obtained \tsl{a priori} by simple energy and transport estimates for, respectively, the momentum equation and the mass conservation equation. Indeed, we have
\begin{equation} \label{est:rho-inf}
\forall\,t\in[0,T]\,,\qquad\qquad \rho_*\,\leq\,\rho(t)\,\leq\,\rho^*\,,
\end{equation}
together with the kinetic energy balance
\begin{equation}
\label{ene:est:glob:form}
\frac{1}{2}\,\frac{\dd}{\dt}\int_{\R^d}\rho\,|u|^2\,\dx\,+\,\int_{\R^d} \mathfrak{D}_{\alpha}^{\g}(\rho, u)\cdot u \,\dx\,=\,0\,.
\end{equation}
As $ \mathfrak{D}_{\alpha}^{\g}(\rho, u) \cdot u \geq 0$,
this latter relation in particular implies, together with \eqref{est:rho-inf}, that
\begin{equation} \label{est:u-L^2}
\forall\,t\in[0,T]\,,\qquad\qquad \left\|u(t)\right\|_{L^2}\,\lesssim\,\left\|u_0\right\|_{L^2}\,.
\end{equation}

It remains to show that also the high regularity norms remain globally bounded.
For this, the key point will be the following continuation criterion in terms of the velocity field \emph{only}. Such a criterion improves the ones
of \cite{D1, D:F} (stated for the case $\alpha=0$), which involved in addition a low regularity norm of the pressure gradient,
and extends the one of \cite{Korean} (again, stated for the case $\alpha=0$ only),
which considered only the case of high Sobolev regularity, to the class of Besov spaces with critical regularity. 

As already remarked at the beginning of this section, we point out that the next statement holds true for any $\alpha\in\R$, even though
we state it only for $\alpha\geq0$. In particular, it holds true also for the non-homogeneous incompressible Euler equations,
corresponding to the case $\alpha=0$.

\begin{lemma} \label{lem:cont:pri}
Let the assumptions of Proposition \ref{p:local} be in force.
In addition, assume that $p\geq 2$.
Let $\big(\rho,u,\nabla\Pi\big)$ be a local solution to system \eqref{eq:dd-E}-\eqref{eq:in-datum}-\eqref{eq:vacuum} related to the initial datum
$\big(\rho_0,u_0\big)$, defined on $[0,T^*[\,\times\R^d$ (for some time $T^*>0$) and such that the regularity properties (i)-(ii)-(iii)
stated in Proposition \ref{p:local} hold true for any $0<T<T^*$. In the case $p=+\infty$, assume moreover that
there exists a $p_0\in\,]d,+\infty[\,$ such that $\nabla u_0\in L^{p_0}$.

If $T^*<+\infty$ and
\begin{equation} \label{est:Du}
\int_0^{T^*} \left\|\nabla u(t) \right\|_{L^{\infty}}\,\dt\, <\, +\, \infty\,,
\end{equation} 
then $\big(\rho,u,\nabla\Pi\big)$ can be continued  beyond the time $T^*$ into a solution of \eqref{eq:dd-E}-\eqref{eq:in-datum}-\eqref{eq:vacuum}
with the same regularity.
\end{lemma}
 
\begin{proof}
Let us restrict our attention to the case $\alpha>0$ and $\g=0$, which is, to some extent,
the most difficult one, as it involves the presence of an additional term in the pressure equation. For obtaining the result in the case $\alpha=0$,
one can simply discard that additional term from the analysis. 
 
Without loss of generality, we can fix the value $\alpha=1$ in the argument below.

The method of the proof is classical. It consists in showing that, under the integral condition \eqref{est:Du}, one has
\begin{equation} \label{est:to-prove}
 \sup_{t\in[0,T^*[}\left(\left\|\nabla\rho(t)\right\|_{B^{s-1}_{p,r}}\,+\,\left\|u(t)\right\|_{B^s_{p,r}}\,+\,
\left\|\nabla\Pi(t)\right\|_{L^2\cap B^s_{p,r}}\right)\,<\,+\infty\,.
\end{equation}
We recall that  global in time bounds for $u$ in $L^2$ and for the $L^\infty$ norms of $\rho^{\pm1}$ are available in, respectively,
\eqref{est:u-L^2} and \eqref{est:rho-inf}.

First of all, we differentiate the mass equation with respect to $x^j$, for $j=1,2$, to get
\begin{equation} \label{eq:D-rho}
\d_t\d_j\rho\,+\,u\cdot\nabla \d_j\rho\,=\,-\,\d_ju\cdot\nabla\rho\,.
\end{equation}
A classical argument (see the details \tsl{e.g.} in Section 4 of \cite{D1}) allows one to obtain, for all $t\in[0,T^*[\,$, the inequality
\begin{equation}
	\label{lin:est:nab:a}
\left\|\nabla \rho\right\|_{B^{s-1}_{p,r}}\,\lesssim\,\left\|\nabla\rho_0\right\|_{B^{s-1}_{p,r}}\,+\,
\int^t_0\left(\|\nabla u \|_{L^{\infty}}\,\|\nabla \rho\|_{B^{s-1}_{p,r}}\, +\,\|\nabla \rho \|_{L^{\infty}}\, \|\nabla u\|_{B^{s-1}_{p,r}}\right)\,\dt\,,
\end{equation}
at least in the case $p<+\infty$, or $p=+\infty$ and $s>1$.
In the endpoint case $p=+\infty$ and $s=1$ (and then $r=1$ as well), we rather work with $\rho$ (see Section 3 of \cite{D:F}, this time) to get
\begin{equation}
	\label{lin:est:rho_inf}
\left\|\rho\right\|_{B^{1}_{\infty,1}}\,\lesssim\,\left\|\rho_0\right\|_{B^{1}_{\infty,1}}\,+\,
\int^t_0\left(\|\nabla u \|_{L^{\infty}}\,\|\rho\|_{B^{1}_{\infty,1}}\, +\,\|\nabla\rho \|_{L^{\infty}}\, \|u\|_{B^{1}_{\infty,1}}\right)\,\dt\,.
\end{equation}

Next, we consider the velocity field. Owing to the absence of vacuum, see \eqref{est:rho-inf}, we can divide the momentum equation by $\rho$ and recast it in the following
form (recall that we have fixed $\alpha=1$ and $\g=0$ at the beginning of the proof):
\begin{equation} \label{eq:mom_2}
\d_tu\,+\,u\cdot\nabla u\,+\,\frac{1}{\rho}\,\nabla\Pi\,+\,\frac{1}{\rho}\,u\,=\,0\,.
\end{equation}
By transport estimates, together with commutator and tame estimates, we can bound, for any $t\in[0,T^*[\,$, the norm of the velocity field as follows:
\begin{align}
	\label{lin:est:u}
\|u(t)\|_{B^{s}_{p,r}}\,&\lesssim\,\|u_0\|_{B^{s}_{p,r}}\,+\,\int^t_0\Bigg(\|\nabla u \|_{L^{\infty}}\,\|u\|_{B^{s}_{p,r}}\,+\,
\frac{1}{\rho_*}\,\|u\|_{B^s_{p,r}}\,+\,\frac{1}{\rho_*}\,\left\|\nabla\Pi\right\|_{B^s_{p,r}} \\
\nonumber
&\qquad\qquad\qquad\qquad\qquad\quad
+\,\|u\|_{L^\infty}\,\left\|\nabla\rho\right\|_{B^{s-1}_{p,r}}\,+\,\|\nabla\Pi\|_{L^\infty}\,\left\|\nabla\rho\right\|_{B^{s-1}_{p,r}}\Bigg)\,\dt\,.
\end{align}
For notational convenience, let us set
\[
 \mc N(t)\,:=\,\left\|\nabla\rho(t)\right\|_{B^{s-1}_{p,r}}\,+\,\left\|u(t)\right\|_{B^s_{p,r}}\,. 
\]
In addition, we notice that 
the Gagliardo-Nirenberg type inequalities of Lemma \ref{l:GN-ineq} imply
\begin{equation} \label{est:interpol_u-Pi}
\|u\|_{L^\infty}\,\lesssim\,\|u\|_{L^2}^{\alpha_1}\,\|\nabla u\|^{1-\alpha_1}_{L^\infty} \qquad \mbox{ and }\qquad
\|\nabla\Pi\|_{L^\infty}\,\lesssim\,\|\nabla\Pi\|_{L^2}^{\alpha_2}\,\|\Delta\Pi\|^{1-\alpha_2}_{L^{p_0}}\,,
\end{equation}
which hold true for two suitable exponents $\alpha_{1,2}\in\,]0,1[\,$, as soon as $p_0>d$. If the index $p$ from the statement verifies
$p>d$, we simply take $p_0=p$, if not we choose a different
$p_0\in\,]d,+\infty[\,$ (in the case $p=+\infty$, simply take the $p_0$ from the statement).
After an application of the Young inequality and summing up
the resulting expression with \eqref{lin:est:nab:a} (or with \eqref{lin:est:rho_inf} in case $p=+\infty$), we find, for any $t\in[0,T^*[\,$, the estimate
\[
 \mc N(t)\,\lesssim\,\mc N(0)+\int^t_0\Big(\big(1+\left\|\nabla\rho\right\|_{L^\infty}+\left\|\nabla u\right\|_{L^\infty}+
\left\|\nabla\Pi\right\|_{L^2}+\left\|\Delta\Pi\right\|_{L^{p_0}}\big)\mc N(\t)+\left\|\nabla\Pi\right\|_{B^s_{p,r}}\Big)\dd\t\,, 
\]
for an implicit multiplicative constant which depends also on $\rho_*$, $\rho^*$ and $\left\|u_0\right\|_{L^2}$.
Observe that, under our assumptions, one always has $B^{s-1}_{p,r}\,\hookrightarrow\,L^p\cap L^\infty$. Then, from equation \eqref{eq:D-rho}
and classical transport estimates, one deduces that, for all $q\in[p,+\infty]$ and all $t\in[0,T^*[\,$, one has
\begin{equation}
\label{est:D-rho_L^q_I}
\left\|\nabla\rho(t)\right\|_{L^q}\,\lesssim\,\left\|\nabla\rho_0\right\|_{L^q}\,+\,\int^t_0\left\|\nabla u(\t)\right\|_{L^\infty}\,
\left\|\nabla\rho(t)\right\|_{L^q}\dd\t\,.
\end{equation}
An application of the Gr\"onwall lemma thus yields, under condition \eqref{est:Du}, the bound
\begin{equation} \label{est:D-rho_q}
\sup_{t\in[0,T^*[}\left\|\nabla \rho(t)\right\|_{L^p\cap L^\infty}\,\leq\,C\,.
\end{equation}
Of course, this inequality only gives a bound on $\nabla\rho$ in $L^\infty$ in the endpoint case $p=+\infty$.

In addition, by computing the divergence of the momentum equation \eqref{eq:mom_2},
we discover that the pressure term $\Pi$ satisfies the elliptic equation
\begin{equation} \label{eq:D-Pi}
 -\,\div\left(\frac{1}{\rho}\,\nabla\Pi\right)\,=\,
\div(F)\,,\qquad\qquad \mbox{ with }\qquad F\,:=\,u\cdot\nabla u\,+\,\frac{1}{\rho}\,u\,.
\end{equation}
An application of Lemma \ref{l:laxmilgram} thus yields
\begin{equation} \label{est:Pi_L^2}
\frac{1}{\rho^*}\,\left\|\nabla\Pi\right\|_{L^2}\,\leq\,\left\|u\cdot\nabla u\,+\,\frac{1}{\rho}\,u\right\|_{L^2}\,\lesssim\,
\left\|u_0\right\|_{L^2}\,\left\|\nabla u\right\|_{L^\infty}\,+\,\frac{1}{\rho_*}\,\left\|u_0\right\|_{L^2}\,, 
\end{equation}
where we have also used the inequalities in \eqref{est:rho-inf} and \eqref{est:u-L^2}.

Inserting \eqref{est:D-rho_q} and \eqref{est:Pi_L^2} into the previous estimate for $\mc N(t)$, we infer that, if condition \eqref{est:Du} holds,
for any $t\in[0,T^*[\,$ one has
\begin{align} \nonumber
 \label{est:N_partial}
\mc N(t) \,&\lesssim\, \mc N(0)\,+\,\left(1+\| \nabla \rho_0\|_{L^{\infty}} + \left(1+\frac{1}{\rho_*}\right)\|u_0\|_{L^2}\right) \\
& \qquad \qquad  \qquad\qquad \times  \int^t_0\Big(\big(1+\left\|\nabla u\right\|_{L^\infty}+\left\|\Delta\Pi\right\|_{L^{p_0}}\big)\,\mc N(\t)\,+\,
\left\|\nabla\Pi\right\|_{B^s_{p,r}}\Big)\,\dd\t \\
\nonumber
&\lesssim\, \mc N(0)\,+\,\int^t_0\Big(\big(1+\left\|\nabla u\right\|_{L^\infty}+\left\|\Delta\Pi\right\|_{L^{p_0}}\big)\,\mc N(\t)\,+\,
\left\|\nabla\Pi\right\|_{B^s_{p,r}}\Big)\,\dd\t\,,
\end{align}
for a new multiplicative constant, also depending on $\|\nabla\rho_0\|_{L^\infty}$.

Next, let us bound the $L^{p_0}$ norm of $\Delta\Pi$. For this, we develop the derivatives in \eqref{eq:D-Pi} to find an equation for
$\Delta\Pi$: we get
\[ 
-\,\Delta\Pi\,=\,-\,\nabla\log\rho\cdot\nabla\Pi\,+\,\rho\,\nabla u:\nabla u \,-\,u\cdot\nabla\log\rho\,.
\] 
Thus, we deduce
\begin{align*}
 \left\|\Delta\Pi\right\|_{L^{p_0}}\,\leq\,\frac{1}{\rho_*}\,\left\|\nabla\rho\right\|_{L^{\infty}}\,\left\|\nabla\Pi\right\|_{L^{p_0}}\,+\,
 \rho^*\,\left\|\nabla u\right\|_{L^{p_0}}\,
\left\|\nabla u\right\|_{L^\infty}\,+\frac{1}{\rho_*}\,\|u\|_{L^{p_0}}\,\left\|\nabla\rho\right\|_{L^\infty}\,.
\end{align*}
First, we can use \eqref{est:D-rho_q} to see that, under assumption \eqref{est:Du}, the $L^\infty$ norm of $\nabla\rho$ is finite over $[0,T^*[\,$.
Next, recall that, with our choice of $p_0$, we always have $p\leq p_0<+\infty$ if $p$ is finite, $d<p_0<+\infty$ if $p=+\infty$ (which in particular implies that
$2<p_0<+\infty$).
Thus,  we can use interpolation between Lebesgue spaces and then the interpolation inequalities \eqref{est:interpol_u-Pi} to get, for suitable exponents
$\theta$ and $\theta_1$ belonging to $\,]0,1[\,$, the bound
\begin{align*}
 \left\|\Delta\Pi\right\|_{L^{p_0}}\,\leq\,\left\|\nabla\Pi\right\|_{L^2}^\theta\,
\left\|\Delta\Pi\right\|_{L^{p_0}}^{1-\theta}\,+\,
 \rho^*\,\left\|\nabla u\right\|_{L^{p_0}}\,
\left\|\nabla u\right\|_{L^\infty}\,+\frac{1}{\rho_*}\,\|u\|_{L^2}^{\theta_1}\,\|\nabla u\|^{1-\theta_1}_{L^\infty}\,.
\end{align*}
Thanks to an application of the Young inequality, we infer, pointwise in the interval $[0,T^*[\,$, the estimate
\[
 \left\|\Delta\Pi\right\|_{L^{p_0}}\,\lesssim\,1\,+\,\left\|\nabla u\right\|_{L^\infty}\,+\,\,\left\|\nabla u\right\|_{L^{p_0}}\,\left\|\nabla u\right\|_{L^\infty}\,.
\]
Since (as already remarked above) we always have $d<p_0<+\infty$ by our choice of $p_0$, by Calder\'on-Zygmund theory we can write the estimate
\[
\left\|\nabla u\right\|_{L^{p_0}}\,\lesssim\,\left\|\Omega\right\|_{L^{p_0}}\,.
\]
The previous bound implies that
\begin{equation} \label{est:Delta-Pi}
 \left\|\Delta\Pi\right\|_{L^{p_0}}\,\lesssim\,1\,+\,\left\|\nabla u\right\|_{L^\infty}\,+\,\left\|\nabla u\right\|_{L^\infty}\,\left\|\Omega\right\|_{L^{p_0}}\,.
\end{equation}

In turn, we have moved the problem of bounding $\left\|\Delta\Pi\right\|_{L^{p_0}}$ to the one of bounding $\Omega$ in the same norm. The advantage is that,
from \eqref{eq:mom_2}, we can compute an equation for the vorticity $\Omega$, which reads
\begin{equation*}
\d_t\Omega\,+\,u\cdot\nabla\Omega\,+\,\Omega\cdot D u\,+\,\nabla u\cdot\Omega\,+\,\nabla\left(\frac{1}{\rho}\right)\wedge\nabla\Pi\,+\,
\frac{1}{\rho}\,\Omega\,+\,\nabla\left(\frac{1}{\rho}\right)\wedge u\,=\,0\,,
\end{equation*}
where, given two vectors $v,w\in \R^d$, we have defined $v\wedge w$ to be the skew-symmetric matrix with components
\[
\big(v\wedge w\big)_{jk}\,=\,v^k\,w^j\,-\,v^j\,w^k\,.
\]
Notice that, in the specific cases $d=2$ and $d=3$, the equation for $\Omega$ actually simplifies; however, for the sake of generality, let us work
in any space dimension $d\geq2$.
Applying transport estimates in $L^{p_0}$ to the previous equation, we find, for any time $t\in[0,T^*[\,$, the inequality
\begin{align*}
\left\|\Omega(t)\right\|_{L^{p_0}}\,&\lesssim\,\left\|\Omega_0\right\|_{L^{p_0}}\,+\,\int^t_0\Big(\left\|\nabla u\right\|_{L^\infty}\,\left\|\Omega\right\|_{L^{p_0}}\,+\,
\left\|\nabla\rho\right\|_{L^\infty}\,\left\|\nabla\Pi\right\|_{L^{p_0}}\,+\,\left\|\nabla\rho\right\|_{L^\infty}\,\left\|u\right\|_{L^{p_0}}\Big)\,\dd\t \\
&\lesssim\,\left\|\Omega_0\right\|_{L^{p_0}}\,+\,\int^t_0\Big(\big(1+\left\|\nabla\rho\right\|_{L^\infty}\big)\,\left\|\nabla u\right\|_{L^\infty}\,
\left\|\Omega\right\|_{L^{p_0}}\,+\,
\left\|\nabla\rho\right\|_{L^\infty}\,\big(1+\left\|\nabla u\right\|_{L^\infty}\big)\Big)\,\dd\t\,,
\end{align*}
where, for passing from the first inequality to the second one, we have used interpolation inequalities between Lebesgue spaces and the ones from \eqref{est:interpol_u-Pi},
as previously done, combined with the Young inequality,
and estimates \eqref{est:Pi_L^2} and \eqref{est:Delta-Pi} for the pressure. The implicit multiplicative constant depends on $\rho_*$, $\rho^*$ and $\|u_0\|_{L^2}$,
of course.
Remark again that, owing to our assumptions, one has $B^{s-1}_{p,r}\,\hookrightarrow\,L^p\cap L^\infty$ and $p\leq p_0<+\infty$ when $p$ finite,
so the $L^{p_0}$ norm of the initial datum is finite; in the case $p=+\infty$, this simply comes from the additional assumption
$\nabla u_0\in L^{p_0}$ and Calder\'on-Zygmund theory. We also notice that, in the previous estimate,
we could neglect the contribution coming from the term $\frac{1}{\rho}\Omega$, because it has the right sign.

At this point, we remark that, under condition \eqref{est:Du}, estimate \eqref{est:D-rho_q} holds true, hence the $L^\infty$ norm of $\nabla\rho$ remains bounded
on $[0,T^*[\,$. Thus, an application of the Gr\"onwall inequality yields
\[
\forall\,t\in[0,T^*[\;,\qquad
\left\|\Omega(t)\right\|_{L^{p_0}}\,\lesssim\,\left(\left\|\Omega_0\right\|_{L^{p_0}}\,+\,\int^t_0\big(1+\left\|\nabla u\right\|_{L^\infty}\big)\,\dd\t\right)\,\exp\left(C\int^t_0\left\|\nabla u\right\|_{L^\infty}\,\dd\t\right)\,,
\]
for a suitable positive constant $C>0$.
Therefore, under condition \eqref{est:Du}, we find
\begin{equation} \label{est:vort-p_0}
 \sup_{t\in[0,T^*[}\left\|\Omega(t)\right\|_{L^{p_0}}\,\leq\,C\,.
\end{equation}
Now, thanks to \eqref{est:Delta-Pi}, from \eqref{est:Du} and \eqref{est:vort-p_0}, we deduce that
\begin{equation} \label{est:Pi-integral}
 \int^{T^*}_0\left\|\Delta\Pi(t)\right\|_{L^{p_0}}\,\dt\,<\,+\infty\,.
\end{equation}

As a last step of the argument, we need to bound the pressure term $\nabla\Pi$ in the Besov norm $B^s_{p,r}$, which correspond to the last term
appearing in \eqref{est:N_partial}. 
The starting point is the estimate of Lemma \ref{lem:pres:est:Besov}, which reads
\[
 \|\nabla \Pi \|_{B^{s}_{p,r}}\, \lesssim\,\left( 1 +  \| \nabla \rho  \|_{L^\infty}^{\eta} \right)\,\|F\|_{L^2}\,+\,
 \| \nabla \rho  \|_{B^{s-1}_{p,r}}\,\left\|\nabla\Pi\right\|_{L^\infty}\,+\, \left\| \rho\,  \div(F)  \right\|_{B^{s-1}_{p,r}}\,,
\]
where $F$ is defined in \eqref{eq:D-Pi}. By making use of \eqref{est:D-rho_q}, \eqref{est:Pi_L^2} and the interpolation inequality \eqref{est:interpol_u-Pi},
we obtain the bound
\begin{equation} \label{est:Pi-B_provvis}
 \|\nabla \Pi \|_{B^{s}_{p,r}}\, \lesssim\,1\,+\,\left\|\nabla u\right\|_{L^\infty}\,+\,
 \left(1+\left\|\nabla u\right\|_{L^\infty}+\left\|\Delta\Pi\right\|_{L^{p_0}}\right)\,\mc N\,+\, \left\| \rho\,  \div(F)  \right\|_{B^{s-1}_{p,r}}\,.
\end{equation}
We have now to find a control on the last term on the right-hand side of the previous inequality.

In order to avoid the appearing of a dangerous $\left\|\nabla u\right\|_{L^\infty}^2$ term, let us write
\begin{align*}
\rho\,  \div(F)\,&=\,\rho\,\nabla u:\nabla u\,+\,\rho\,\nabla\left(\frac{1}{\rho}\right)\cdot u\,=\,
\nabla\big(\rho\,u\big):\nabla u\,-\,\big(u\cdot\nabla u\big)\cdot\nabla\rho \,-\,u\cdot\nabla\log\rho\,.
\end{align*}
Hence, by a repeated use of the tame estimates, when $s>1$ we gather
\begin{align*}
\left\| \rho\,  \div(F)  \right\|_{B^{s-1}_{p,r}}\,&\lesssim\,\left\|\nabla u\right\|_{L^\infty}\,\left\|\nabla\big(\rho u\big)\right\|_{B^{s-1}_{p,r}}
\,+\,\left\|\nabla u\right\|_{B^{s-1}_{p,r}}\,\left\|\nabla\big(\rho u\big)\right\|_{L^\infty} \\
&\qquad\qquad \,+\,
\left\|u\cdot\nabla u\right\|_{B^{s-1}_{p,r}}\,\left\|\nabla\rho\right\|_{L^\infty}\,+\,
\left\|u\cdot\nabla u\right\|_{L^\infty}\,\left\|\nabla\rho\right\|_{B^{s-1}_{p,r}} \\
&\qquad\qquad\qquad\qquad \,+\,\|u\|_{L^\infty}\,\left\|\nabla\log\rho\right\|_{B^{s-1}_{p,r}}\,+\,
\|u\|_{B^{s-1}_{p,r}}\,\left\|\nabla\log\rho\right\|_{L^\infty}\,.
\end{align*}
Observe that, by using Leibniz rule and Bony's paraproduct decomposition, we can bound
\begin{align*}
\left\|\nabla\big(\rho u\big)\right\|_{B^{s-1}_{p,r}}\,&\leq\,\left\|u\,\nabla\rho\right\|_{B^{s-1}_{p,r}}\,+\,\left\|\rho\,\nabla u\right\|_{B^{s-1}_{p,r}} \\
&\lesssim\,\left\|\nabla\rho\right\|_{B^{s-1}_{p,r}}\,\left\|u\right\|_{L^\infty}\,+\,\left\|u\right\|_{B^{s-1}_{p,r}}\,\left\|\nabla\rho\right\|_{L^\infty}\,+\,
\left\|\nabla u\right\|_{B^{s-1}_{p,r}}\,+\,\left\|\mc T_{\nabla u}(\rho-1)\right\|_{B^{s-1}_{p,r}}\,,
\end{align*}
where, owing to Proposition \ref{p:op}, we can bound
\[
\left\|\mc T_{\nabla u}(\rho-1)\right\|_{B^{s-1}_{p,r}}\,\lesssim\,\left\|\nabla u\right\|_{B^{-1}_{\infty,\infty}}\,\left\|\nabla \rho\right\|_{B^{s-2}_{p,r}} \,\lesssim\,\left\|u\right\|_{L^\infty}\,\left\|\nabla \rho\right\|_{B^{s-1}_{p,r}}\,.
\]
In particular, this argument shows that we have indeed avoided the appearing of the term  $\left\|\nabla u\right\|_{L^\infty}^2$
in the continuation criterion.

If $s=1$ (and then $r=1$ as well), instead, the argument for bounding $\rho\,\div(F)$ is slightly more involved, as $B^0_{\infty,1}$ is \emph{not} an algebra.
However, we can use the fact that $\rho$ and $u$ actually belong to $B^{1}_{\infty,1}$ and that $B^{1/2}_{\infty,1}$ is still an interpolation space between
$B^{1}_{\infty,1}$ and $L^2$. Thus, arguing as done for arriving to \eqref{algebra:B^0} and \eqref{algebra:B^0_2}, we can bound
\begin{align*}
\left\|\nabla\big(\rho\,u\big):\nabla u\right\|_{B^0_{\infty,1}}\,&\lesssim\,\left\|\rho\,u\right\|_{B^1_{\infty,1}}\,\left\|\nabla u\right\|_{L^\infty}\,+\,
\left\|\nabla\big(\rho\,u\big)\right\|_{L^\infty}\,\left\|\nabla u\right\|_{B^0_{\infty,1}} \\
\left\|u\cdot\nabla\log\rho\right\|_{B^0_{\infty,1}}\,&\lesssim\,\left\|u\right\|_{B^{1/2}_{\infty,1}}\,\left\|\nabla\log\rho\right\|_{L^\infty}\,+\,
\|u\|_{L^\infty}\,\left\|\nabla\log\rho\right\|_{B^0_{\infty,1}}\,,
\end{align*}
where we have exploited the fact that $\nabla\big(\rho u\big):\nabla u\,=\,\div\big(\rho\,u\cdot\nabla u\big)$
and that $\nabla\log\rho\cdot u\,=\,\div\left(\log\rho\,u\right)$.
As for the last term, we have 
\[
\big(u\cdot\nabla u\big)\cdot\nabla\rho\,=\,\sum_{j,k}u^j\,\d_ju^k\,\d_k\rho\,,
\]
so we can argue as above to get
\begin{align*}
\left\|\big(u\cdot\nabla u\big)\cdot\nabla\rho\right\|_{B^0_{\infty,1}}\,&\lesssim\,
\left\|u\right\|_{B^{0}_{\infty,1}}\,\left\|\nabla u\cdot \nabla\rho\right\|_{L^\infty}\,+\,
\left\|u\right\|_{L^\infty}\,\left\|\nabla u\cdot \nabla\rho\right\|_{B^{0}_{\infty,1}}\,,
\end{align*}
where, owing to the fact that, for any $j$ fixed, $\sum_k\d_ju^k\,\d_k\rho\,=\,\div\big(\rho\,\d_j u\big)$, one has
\[
\left\|\nabla u\cdot \nabla\rho\right\|_{B^{0}_{\infty,1}}\,\lesssim\,\left\|\nabla u\right\|_{B^{0}_{\infty,1}}\,\left\|\nabla\rho\right\|_{L^\infty}\,+\,
\left\|\nabla u\right\|_{L^\infty}\,\left\|\rho\right\|_{B^1_{\infty,1}}\,.
\]

In any case, one arrive at the following bound for the Besov norm of $\rho\,\div(F)$:
\begin{align*}
\left\| \rho\,  \div(F)  \right\|_{B^{s-1}_{p,r}}\,&\lesssim\,
\Big(1+\left\|u\right\|_{L^\infty}+\left\|\nabla\rho\right\|_{L^\infty}\Big)\,\left\|\nabla u\right\|_{L^\infty}\,\mc N
 \\
&\qquad \,+\,\left\|u\right\|_{L^\infty}\,\left\|\nabla\rho\right\|_{L^\infty}\,\mc N\,+\,
\Big(\left\|u\right\|_{L^\infty}+\left\|\nabla\rho\right\|_{L^\infty}\Big)\,\mc N\,,
\end{align*}
where we have also used the paralinearisation result of Proposition \ref{p:comp}.
In order to control the $\|u\|_{L^\infty}$, we can apply the second interpolation inequality of Lemma \ref{l:GN-ineq}
to deduce the existence of a $\alpha_3\in\,]0,1[\,$
such that
\[
\|u\|_{L^\infty}\,\lesssim\,\left\|u\right\|_{L^2}^{\alpha_3}\,\left\|\Omega\right\|_{L^{p_0}}^{1-\alpha_3}\,,
\]
where $p_0$ is the same index appearing above (in particular, $p_0=p$ if $p>d$, while $p<p_0<+\infty$ otherwise).
In particular, by estimates \eqref{est:u-L^2} and \eqref{est:vort-p_0}, we deduce that, under condition \eqref{est:Du},
we have $\sup_{[0,T^*[\,}\left\|u\right\|_{L^\infty}\,<\,+\infty$.
Combining this property with \eqref{est:D-rho_q}, from \eqref{est:Pi-B_provvis} we infer the inequality
\begin{equation} \label{est:D-Pi_Besov}
\|\nabla \Pi \|_{B^{s}_{p,r}}\, \lesssim\,1\,+\,\left\|\nabla u\right\|_{L^\infty}\,+\,
 \left(1+\left\|\nabla u\right\|_{L^\infty}+\left\|\Delta\Pi\right\|_{L^{p_0}}\right)\,\mc N\,.
\end{equation}

At this point, we can plug the previous bound \eqref{est:D-Pi_Besov} into \eqref{est:N_partial} and get, under condition \eqref{est:Du}, the following inequality:
\[
\forall\,t\in[0,T^*[\;,\qquad\qquad \mc N(t)\,\lesssim\,\mc N(0)\,+\,\int^t_0\Big(1+\left\|\nabla u\right\|_{L^\infty}+\left\|\Delta\Pi\right\|_{L^{p_0}}\Big)\,
\mc N(\t)\,\dd\t\,.
\]
An application of Gr\"onwall inequality and the use of \eqref{est:Du} and \eqref{est:Pi-integral} then imply that
\[
\sup_{t\in[0,T^*[\,}\mc N(t)\,<\,+\infty\,.
\]
Thanks to the embedding $B^s_{p,r}\,\hookrightarrow\, W^{1,\infty}$, in turn we deduce that also $\left\|\nabla u\right\|_{L^\infty}$ is uniformly bounded in
$[0,T^*[\,$, thus also $\left\|\nabla\Pi\right\|_{L^2}$, owing to \eqref{est:Pi_L^2}, and $\left\|\Delta\Pi\right\|_{L^{p_0}}$, thanks to \eqref{est:Delta-Pi}.
Coming back to \eqref{est:D-Pi_Besov}, we finally gather that
$\sup_{[0,T^*[\,}\left\|\nabla\Pi\right\|_{B^s_{p,r}}\,<\,+\infty$. Hence, we infer that \eqref{est:to-prove} holds true, as sought.

From \eqref{est:to-prove}, one can use a standard argument (which we omit here) to complete the proof of the continuation criterion.
\end{proof}

We conclude this section with a remark, relating our continuation criterion of Lemma \ref{lem:cont:pri} to the celebrated Beale-Kato-Majda criterion.

\begin{rmk} \label{r:BKM}
It would be tempting to replace, in the subcritical case $s>0$, condition \eqref{est:Du} of the continuation criterion with the condition
\begin{equation} \label{est:BKM}
\int^{T^*}_0\left\|\Omega(t)\right\|_{L^\infty}\,\dt\,<\,+\,\infty\,,
\end{equation}
where $\Omega$ is the vorticity of the fluid. This would correspond to the equivalent of the celebrated Beale-Kato-Majda continuation criterion in the context
of the non-homogeneous incompressible Euler equations.

Unfortunately, we do not know, at present, whether our continuation criterion still holds true under the weaker assumption \eqref{est:BKM}. The problem
comes from the presence, in bound \eqref{lin:est:nab:a} for the density, of the factor $\|\nabla\rho\|_{L^\infty}$ multiplying
the higher order norm $\|\nabla u\|_{B^{s-1}_{p,r}}$. Indeed, owing to estimate \eqref{est:D-rho_L^q_I} and the well-known interpolation
inequality
\[
\left\|\nabla u\right\|_{L^\infty}\,\lesssim\,\left\|u\right\|_{L^2}\,+\,\left\|\Omega\right\|_{L^\infty}\,
\log\left(e\,+\,\frac{\left\|\Omega\right\|_{B^{s-1}_{p,r}}}{\left\|\Omega\right\|_{L^\infty}}\right)
\]
(see \tsl{e.g.} Proposition 7.7 of \cite{BCD}),
the term $\|\nabla\rho\|_{L^\infty}\,\|\nabla u\|_{B^{s-1}_{p,r}}$ can be controlled only by $\mc N^2$. This growth is, of course, not suitable
for obtaining the sought continuation criterion.
\end{rmk}

\section{Global well-posedness in the case $\g=1$} \label{s:g=1}

In this section, we carry out the proof of Theorems \ref{th:g=1_d} and \ref{th:g=1_2} in the case $\g=1$. 
This case is somehow simpler, inasmuch as, owing to assumption \eqref{eq:vacuum} of absence of vacuum, we can divide the momentum equation by $\rho$ and write it
in the following form:
\begin{equation*}
\d_tu\,+\,u\cdot\nabla u\,+\,\frac{1}{\rho}\,\nabla\Pi\,+\,\alpha\,u\,=\,0\,.
\end{equation*}
Thus, if we define the new velocity field $\wu$ and the new pressure function $\wtilde\Pi$ as
\begin{equation} \label{eq:change}
 \wu\,:=\,e^{\alpha\,t}\,u\qquad\qquad \mbox{ and }\qquad\qquad \wtilde\Pi\,:=\,e^{\alpha\,t}\,\Pi \,,
\end{equation}
then system \eqref{eq:dd-E} can be recasted in the following form:
\begin{equation} \label{eq:dd-E-wu}
\left\{\begin{array}{l}
\d_t\rho\,+\,e^{-\alpha t}\,\wu\cdot\nabla\rho\,=\,0 \\[1ex]
\rho\,\d_t\wu\,+\,\rho\,e^{-\alpha t}\,\wu\cdot\nabla\wu\,+\,\nabla\wtilde\Pi\,=\,0 \\[1ex]
\div\wu\,=\,0\,.
       \end{array}
\right.
\end{equation}
Of course, systems \eqref{eq:dd-E} and \eqref{eq:dd-E-wu} are equivalent in our functional framework.

Throughout this section, we consider a solution $\big(\rho,u,\nabla\Pi\big)$ to system \eqref{eq:dd-E}, related to an initial datum $\big(\rho_0,u_0\big)$ satisfying the
assumptions fixed in Theorems \ref{th:g=1_d} and \ref{th:g=1_2} above.
Let us call $T=T(\rho_0,u_0)>0$ the lifespan of that solution, so that $\big(\rho,u,\nabla\Pi\big)$ is defined on $[0,T[\,\times\R^d$.
The transformation \eqref{eq:change} produces a solution $\big(\rho,\wu,\nabla\wtilde\Pi\big)$ to system \eqref{eq:dd-E-wu}, related to the same initial datum,
defined on the same time interval.

Our goal is to prove that $T=+\infty$. Arguing by contradiction, in what follows we are going to assume that $T<+\infty$ instead.

\medbreak
Recall that, by the mass equation, we immediately get the bound  \eqref{est:rho-inf}, namely
\begin{equation} \label{est:rho-inf_RIP}
\forall\,t\in[0,T[\;,\qquad\qquad \rho_*\,\leq\,\rho(t)\,\leq\,\rho^*\,.
\end{equation}
Moreover, an energy estimate performed on the ``momentum equation'' of system \eqref{eq:dd-E-wu} implies that
\[
\frac{1}{2}\,\frac{\dd}{\dt}\big\|\sqrt{\rho}\;\wu\big\|_{L^2}^2\,=\,0\,,
\]
which implies, together with \eqref{eq:vacuum} and \eqref{est:rho-inf_RIP}, that
\begin{equation} \label{est:wu-en}
\forall\,t\in[0,T[\;,\qquad\qquad \left\|\wu(t)\right\|_{L^2}\,\lesssim\,\left\|u_0\right\|_{L^2}\qquad
\mbox{ (\tsl{i.e.} \ $\left\|u(t)\right\|_{L^2}\,\lesssim\,e^{-\alpha t}\,\left\|u_0\right\|_{L^2}$) }\,.
\end{equation}

Thus, we need to prove a global uniform (in time) bound also for the higher order regularity norms of the solution $\big(\rho,\wu,\nabla\wtilde\Pi\big)$. For this, we will
exploit the continuation criterion of Lemma \ref{lem:cont:pri} in a fundamental way:
we know that the Sobolev norms of $\big(\rho,\wu,\nabla\wtilde\Pi\big)$ remain bounded on $[0,T[\,$ if and only if the integral condition \eqref{est:Du} is verified.

Now, cutting into low and high frequencies, it is standard to bound
\[ 
\left\|\nabla u\right\|_{L^\infty}\,\lesssim\,\left\|u\right\|_{L^2}\,+\,\left\|\Omega\right\|_{B^0_{\infty,1}}\,,
\] 
where we recall that $\Omega\,:=\,Du\,-\,\nabla u$ denotes the vorticity of the fluid. 
Recall that we have assumed, by contradiction, that the lifespan $T$ verifies $T<+\infty$. Hence, owing to \eqref{est:wu-en}, the failure of
condition \eqref{est:Du} implies that
\begin{equation} \label{est:Omega_to-prove}
 \int^T_0\left\|\Omega(t)\right\|_{B^0_{\infty,1}}\,\dt\,=\,\int^T_0e^{-\alpha\,t}\;\left\|\wtilde\Omega(t)\right\|_{B^0_{\infty,1}}\,\dt\,=\,+\,\infty\,,
\end{equation}
where we have denoted by $\wtilde\Omega\,:=\,D\wu\,-\,\nabla\wu$ the vorticity of the divergence-free velocity field $\wu$.

From now on, for the sake of clarity, we split our argument into two parts:
in Subsection \ref{ss:g=1_d=2} we focus on the $2$-D situation and prove global existence under a smallness assumption on the non-homogeneity only
(namely, we prove Theorem \ref{th:g=1_2}), while in Subsection \ref{ss:g=1_d-geq3} we treat the general
case of space dimension $d\geq2$ and we prove Theorem \ref{th:g=1_d}.

\subsection{The case of planar flows: proof of Theorem \ref{th:g=1_2}} \label{ss:g=1_d=2}

To begin with, we consider the case of two-dimensional flows $d=2$. We aim at proving Theorem \ref{th:g=1_2}, namely global well-posedness
under a smallness condition relying only on the size of the initial non-homogeneity $\rho_0-1$.

Recall that, in the case $d=2$, the vorticities $\Omega$ and $\wOm$, related respectively to the divergence-free velocity fields $u$ and $\wu$,
can be identified with the scalar functions
\[
\o\,:=\,\curl(u)\,=\,\d_1u^2\,-\,\d_2u^1\qquad\qquad \mbox{ and }\qquad\qquad
\wom\,:=\,\curl(\wu)\,=\,\d_1\wu^2\,-\,\d_2\wu^1\,.
\]
Now, we implement the following strategy: we perform $B^0_{\infty,1}$ estimates for the vorticity $\wom$ and show that, under the assumptions of Theorem \ref{th:g=1_2},
relation \eqref{est:Omega_to-prove} cannot hold, thus reaching the sought contradiction.

The advantage of the two-dimensional situation is that, as is well-known, the equation for the vorticity simplifies, due to special cancellations which imply
the vanishing of the stretching term. More precisely, if we start from the momentum equation in \eqref{eq:dd-E-wu} and we compute an 
equation for $\wom$, as $d=2$ we find
\begin{equation} \label{eq:wom}
\d_t\wom\,+\,e^{-\alpha t}\,\wu\cdot\nabla\wom\,=\,-\,\nabla^\perp\left(\frac{1}{\rho}\right)\cdot\nabla\wtilde\Pi\,,
\end{equation}
where, for any vector field $v\in\R^2$, $v=\big(v^1,v^2\big)$, we have denoted by $v^\perp\,:=\,\big(-v^2,v^1\big)$ its rotation of angle $\pi/2$.

At this point, in order to bound $\wom$ in $B^0_{\infty,1}$, we can apply the improved transport estimates from Theorem \ref{th:B^0}
to equation \eqref{eq:wom}: we get, for any $t\in[0,T[\,$, the inequality
\begin{equation} \label{est:wom_part}
\left\|\wom(t)\right\|_{B^0_{\infty,1}}\,\lesssim\,\left(\left\|\o_0\right\|_{B^0_{\infty,1}}\,+\,
\int^t_0\left\|\nabla^\perp\left(\frac{1}{\rho}\right)\cdot\nabla\wtilde\Pi\right\|_{B^0_{\infty,1}}\,\dd\t\right)\,
\left(1\,+\,\int^t_0e^{-\alpha\t}\,\left\|\nabla\wu\right\|_{L^\infty}\,\dd\t\right)\,,
\end{equation}
where we have set $\o_0\,=\,\curl(u_0)$.

The bound of the non-linear term appearing in the right-hand side of the previous inequality may appear complicated, as the space $B^0_{\infty,1}$ is \emph{not} algebra.  However, we can rewrite
\[
\nabla^\perp\left(\frac{1}{\rho}\right)\cdot\nabla\wtilde\Pi\,= \nabla^\perp\left(\frac{1}{\rho}-1\right)\cdot\nabla\wtilde\Pi\, = \,\curl\left(\left(\frac{1}{\rho}-1 \right)\,\nabla\wtilde\Pi\right)
\]
and argue as for \eqref{algebra:B^0_2}. Using a careful paraproduct decomposition (see also Section 4 of \cite{D:F}), together with paralinearisation estimate, one can actually prove that
\begin{equation} \label{est:dens-Pi}
\left\|\nabla^\perp\left(\frac{1}{\rho}\right)\cdot\nabla\wtilde\Pi\right\|_{B^0_{\infty,1}}\,\lesssim\,
\left\|\rho-1\right\|_{B^1_{\infty,1}}\,\left\|\nabla\wtilde\Pi\right\|_{B^0_{\infty,1}}\,.
\end{equation}
On the other hand, applying classical Besov estimates to the transport equation for $\rho$ in \eqref{eq:dd-E-wu}, which of course holds true also for
the quantity $\rho-1$, we find
\begin{equation} \label{est:dens-Besov}
\left\|\rho(t)-1\right\|_{B^1_{\infty,1}}\,\lesssim\,\left\|\rho_0-1\right\|_{B^1_{\infty,1}}\,
\exp\left(C\int^t_0e^{-\alpha\t}\,\left\|\wu\right\|_{B^1_{\infty,1}}\,\dd\t\right)\,.
\end{equation}

Our next goal is to bound the $B^0_{\infty,1}$ norm of $\nabla\wtilde\Pi$. Actually, we are going to bound its $B^1_{\infty,1}$ norm, for it is not clear to us how to take
really advantage of this lower order requirement. Computing an equation for $\nabla\wtilde\Pi$ from the second equation in \eqref{eq:dd-E-wu} yields
\begin{equation} \label{eq:Pi_g=1}
-\div\left(\frac{1}{\rho}\,\nabla\wtilde\Pi\right)\,=\,\div\left(e^{-\alpha t}\,\wu\cdot\nabla\wu\right)\,.
\end{equation}
Before applying inequality \eqref{gen:pres:ine} with $ F = e^{-\alpha t}\,\wu\cdot\nabla\wu $, let us notice that
\begin{align*}
\|e^{-\alpha t}\,\wu\cdot\nabla\wu\|_{L^2}\, &\leq\, e^{-\alpha t}\, \|\wu\|_{L^2}\|\nabla\wu\|_{L^{\infty}} \\
\|\rho \div(e^{-\alpha t}\,\wu\cdot\nabla\wu)\|_{B^{0}_{\infty,  1}}\, &\lesssim\,e^{-\alpha t}\, \| \rho \|_{B^{1}_{\infty,1}}\,\|\wu\|_{B^{1}_{\infty, 1}}^2\,,
\end{align*}
where we used $\div\left(e^{-\alpha t}\,\wu\cdot\nabla\wu\right)\,=\,e^{-\alpha t}\,\nabla\wu:\nabla\wu$, together with \eqref{algebra:B^0}
and \eqref{algebra:B^0_2}. These estimates together with \eqref{gen:pres:ine} imply the bound
\begin{align}
\label{est:Pi-B_g=1}
\left\|\nabla\wtilde\Pi\right\|_{B^1_{\infty,1}}\,
&\lesssim\,e^{-\alpha t}\,\left(1\,+\,\|\rho-1\|_{B^1_{\infty,1}}^\eta\right)\,\Big(\left\|\wu\right\|_{L^2}\,\left\|\nabla\wu\right\|_{L^\infty}\,+\,
\left\|\wu\right\|_{B^1_{\infty,1}}^2\Big)\,.
\end{align}

For notational convenience, let us set
\[
\wtilde U(t)\,:=\,\left\|\wu(t)\right\|_{L^2}\,+\,\left\|\wom(t)\right\|_{B^0_{\infty,1}}\qquad\qquad \mbox{ and }\qquad\qquad
\wtilde U_0\,:=\,\left\|u_0\right\|_{L^2}\,+\,\left\|\o_0\right\|_{B^0_{\infty,1}}\,.
\]
Notice that, for any $t\in[0,T[\,$, we have that $\wtilde U(t)$ is equivalent to $\left\|\wu(t)\right\|_{L^2\cap B^1_{\infty,1}}$.
Then, we can plug inequalities \eqref{est:dens-Besov} and \eqref{est:Pi-B_g=1} into \eqref{est:dens-Pi} and use the resulting expression
in \eqref{est:wom_part}: thanks to the energy inequality \eqref{est:wu-en}, we get, for any $t\in[0,T[\,$, the bound
\begin{align} \label{est:U-til}
\wtilde U(t)\,\lesssim\,\left(\wtilde U_0\,+\,R_0\,
\int^t_0e^{C_0\int^\t_0e^{-\alpha\s}\wtilde U(\s)\,\dd\s}\,e^{-\alpha\t}\,\left(\wtilde U(\t)\right)^2\,\dd\t\right)\,
\left(1\,+\,\int^t_0e^{-\alpha\t}\,\wtilde U(\t)\,\dd\t\right)\,,
\end{align}
where $C_0>0$ in the exponential term is a suitable constant and where we have defined
\[
R_0\,:=\,\left\|\rho_0-1\right\|_{B^1_{\infty,1}}\,\left(1\,+\,\left\|\rho_0-1\right\|^\eta_{B^1_{\infty,1}}\right)\,.
\]

At this point, we notice that using a classical procedure (see for instance the definition of the time $T_1$ in the next subsection) would lead to a
smallness requirement \emph{both} on the size of the initial non-homogeneity and on $\wtilde U_0/\alpha$. In order to avoid the appearing of the latter condition,
a more subtle argument is required, which we are going to present now.

Let us define the time $T_0>0$ as
\begin{align*}
T_0\,&:=\,\sup\Bigg\{t\in[0,T[\;\Big|\quad R_0\,\int^t_0e^{C_0\int^\t_0e^{-\alpha\s}\wtilde U(\s)\,\dd\s}\,e^{-\alpha\t}\,
\left(\wtilde U(\t)\right)^2\,\dd\t\,\leq\,4\,\wtilde{U}_0\Bigg\}\,.
\end{align*}
Then, from the previous estimate \eqref{est:U-til}, we get that
\begin{equation} \label{est:U-til_2}
\forall\,t\in[0,T_0]\,,\qquad\qquad \wtilde U(t)\,\leq\,C_1\,\wtilde U_0\,\left(1\,+\,\int^t_0e^{-\alpha\t}\,\wtilde U(\t)\,\dd\t\right)\,,
\end{equation}
where the constant $C_1$ only depends on $\rho_*$ and $\rho^*$.
Thus, if we define the function
\[
 f(t)\,:=\,1\,+\,\int^t_0e^{-\alpha\t}\,\wtilde U(\t)\,\dd\t\,,
\]
in turn we deduce that
\[
\forall\,t\in[0,T_0]\,,\qquad\qquad f'(t)\,\leq\,C_1\,\wtilde U_0\,e^{-\alpha t}\,f(t)\,.
\]
An application of the Gr\"onwall inequality and estimate \eqref{est:U-til_2} then imply
\begin{equation} \label{est:up-to-T_0}
\forall\,t\in[0,T_0]\,,\qquad\qquad f(t)\,\leq\,e^{\frac{C_1}{\alpha}\,\wtilde U_0}\qquad \mbox{ and }\qquad 
\wtilde U(t)\,\leq\,C_1\,\wtilde U_0\,e^{\frac{C_1}{\alpha}\,\wtilde U_0}\,.
\end{equation}

Using now the definition of the time $T_0$, we see that
\[
4\,\wtilde U_0\,=\,R_0\int^{T_0}_0e^{C_0\int^t_0e^{-\alpha\t}\wtilde U(\t)\,\dd\t}\,e^{-\alpha t}\,\left(\wtilde U(t)\right)^2\,\dt\,=\,
R_0\int^{T_0}_0e^{C_0\big(f(t)-1\big)}\,f'(t)\,\wtilde U(t)\,\dt\,.
\]
Making use of the bounds from \eqref{est:up-to-T_0} and integrating the function $ e^{C_0\big(f(t)-1\big)}\,f'(t) $ on the interval $[0, T_0]$, we get
\begin{align*}
4\,\wtilde U_0\,&\leq\,R_0\,\frac{1}{C_0}\,\left(\,e^{C_0\big(f(T_0)-1\big)}-1\right)\,C_1\,\wtilde U_0\,e^{\frac{C_1}{\alpha}\,\wtilde U_0} \\
&\leq\,R_0\,\frac{C_1}{C_0}\left(\,e^{C_0\left(\exp\left(C_1\frac{\wtilde U_0}{\alpha}\right)-1\right)}-1\right) \, \wtilde U_0\,e^{C_1\frac{\wtilde U_0}{\alpha}}\,.
\end{align*}
Thus, after defining $M\,:=\,\max\big\{C_0\,,\,C_1\,,\,\frac{C_1}{C_0}\big\}$, if $R_0$ and $U_0$ are such that the inequality
\[
R_0\,M\,e^{M\,\frac{\wtilde U_0}{\alpha}}\,
\left(\,e^{M\left(\exp\left(M\frac{\wtilde U_0}{\alpha}\right)-1\right)}-1\right)\,<\,4
\]
is satisfied, we see that one must have $T_0\equiv T$, which in turn implies, thanks to \eqref{est:up-to-T_0}, that
\[
\sup_{t\in[0,T[\,}\wtilde U(t)\,\approx\,\sup_{t\in[0,T[\,}\left\|\wu\right\|_{L^2\cap B^1_{\infty,1}}\,<\,+\,\infty\,.
\]
This bound obviously contradicts \eqref{est:Omega_to-prove}, thus implying that $T=+\infty$, \tsl{i.e.} the constructed solution is global.

Observe that, as a byproduct of our argument, we get also that
\[
 \forall\,t\in\R_+\,,\qquad\qquad \left\|u(t)\right\|_{L^2\cap B^1_{\infty,1}}\,\lesssim\,e^{-\alpha t}\,\left\|u_0\right\|_{L^2\cap B^1_{\infty,1}}\,
e^{\frac{C_1}{\alpha}\,\left\|u_0\right\|_{L^2\cap B^1_{\infty,1}}}\,,
\]
thanks to inequalities \eqref{est:wu-en} and \eqref{est:up-to-T_0}, 
and that
\[
\forall\,t\in\R_+\,,\qquad\qquad \rho_*\,\leq\,\rho(t)\,\leq\,\rho^*\qquad \mbox{ and }\qquad
\left\|\rho(t)-1\right\|_{B^1_{\infty,1}}\,\lesssim\, \left\|\rho_0-1\right\|_{B^1_{\infty,1}}\,,
\]
owing to \eqref{est:dens-Besov}. Going back to \eqref{est:Pi-B_g=1}, we also see that
\[
 \forall\,t\in\R_+\,,\qquad\qquad \left\|\nabla\Pi(t)\right\|_{L^2\cap B^1_{\infty,1}}\,\lesssim\,e^{-2\alpha t}\,.
\]

With those bounds at our disposal, by working on system \eqref{eq:dd-E-wu}, we can propagate the Besov regularity of $\rho$, $\wu$ and $\nabla\wtilde\Pi$,
globally in time, and deduce the sought regularity properties.

This observation concludes the proof of Theorem \ref{th:g=1_2}, provided we show similar decay estimates also for the higher order Besov norms of the solution.
We postpone the proof of this latter fact to Subsection \ref{ss:higher_decay}.

\subsection{The general case: proof of Theorem \ref{th:g=1_d}} 
\label{ss:g=1_d-geq3}

Let us consider now the case of any space dimension $d\geq2$ and prove global well-posedness under a smallness condition on the norm of the initial datum.
Now, in order to avoid the appearing of the $L^2$ norm of $u_0$ in the smallness condition, it is more convenient to work with the velocity formulation of the equations,
rather than with the vorticity.

Similarly to what done in the previous section, by arguing by contradiction, we are going to show that, under the smallness assumption of Theorem \ref{th:g=1_d},
the blow-up \eqref{est:Omega_to-prove} cannot happen.
To see this, we perform transport estimates directly on the equation for the velocity field, namely
\[
 \d_t\wtilde u\,+\,e^{-\alpha\,t}\,\wtilde u\cdot\nabla \wtilde u\,+\,\frac{1}{\rho}\,\nabla\wtilde\Pi\,=\,0\,.
\]
By mimicking the proof of Theorem \ref{th:transport}, from the previous equation it is plain to get the following bound, for any $t\in[0,T[\,$:
\begin{align*}
\left\| \wu(t) \right\|_{B^{1}_{\infty,1}}\, &\lesssim\, \left\| u_0\right\|_{B^{1}_{\infty,1}}\, +\,
\int_0^t \left( \| \rho\|_{B^{1}_{\infty,1}}\, \left\| \nabla \wtilde \Pi \right\|_{B^{1}_{\infty,1}}\, + e^{-\alpha \t  } \,
\| \wu \|_{B^{1}_{\infty,1}}^2 \right) \, \dd\t\,, 
%
%
\end{align*}    
for an implicit multiplicative constant also depending on $\rho_*$ and $\rho^*$ appearing in \eqref{eq:vacuum}.

In order to estimate the pressure term, we can still use \eqref{est:Pi-B_g=1}, up to take a higher $\eta=\eta(d)$, depending only on the space dimension $d$.
After setting $ R_1\, :=\, 1\,+\,\left\| \rho_0-1\right\|_{B^{1}_{\infty,1}}^{\eta} $, we combine estimate \eqref{est:Pi-B_g=1} with \eqref{est:dens-Besov}
to get
\[
\left\|\nabla\wtilde\Pi\right\|_{B^1_{\infty,1}}\,\lesssim\,
e^{-\alpha\t}\,R_1\,\left(\left\|\wtilde u\right\|_{L^2}\,+\,\left\|\wtilde u\right\|_{B^1_{\infty,1}}\right)
\exp\left(\eta\,C\int^\t_0e^{-\alpha\s}\,\left\|\wtilde u\right\|_{B^1_{\infty,1}}\,\dd\s\right)\,,
\]
Up to change $\eta$ into $\eta+1$ in the definition of $R_1$, making use of \eqref{est:wu-en} we then obtain
\begin{align*}
\left\| \wu(t) \right\|_{B^{1}_{\infty,1}}\, &\lesssim\, 
\left\| u_0\right\|_{B^{1}_{\infty,1}}\, +\, R_1\, \int_0^t e^{-\alpha \t  }\, e^{C\int_0^{\t} e^{-\alpha\sigma}\,\|\wu\|_{B^{1}_{\infty,1}} \, \dd \s}
\left(\left\|u_0\right\|_{L^2}+ \left\| \wu \right\|_{B^{1}_{\infty,1}} \right)\,\left\| \wu \right\|_{B^{1}_{\infty,1}} \, \dd \t\,,
\end{align*}
and an application of the Gr\"onwall lemma finally yields, for any $t\in[0,T[\,$, the bound 
\begin{align}
\label{est:u-partial}
\left\| \wu(t) \right\|_{B^{1}_{\infty,1}}\, &\lesssim\,\left\| u_0\right\|_{B^{1}_{\infty,1}}\,
\exp \left( R_1\, \int_0^t e^{-\alpha \t  }\, e^{C\int_0^{\t} e^{-\alpha\sigma}\,\left\|\wu\right\|_{B^{1}_{\infty,1}} \, \dd \sigma}\,
\left(\left\|u_0\right\|_{L^2}\,+\, \left\| \wu \right\|_{B^{1}_{\infty,1}} \right) \, \dd \t \right)\,.
\end{align}    

With this estimate at hand, the rest of the argument follows the main lines of the one employed in Subsection \ref{ss:g=1_d=2} above.
Let us define the time $T_1$ as
\begin{align*}
{T}_1\, :=\,\sup\Bigg\{t\in[0,T[\;\Big|\quad \int^t_0e^{-\alpha\t}\,\| \wu(t) \|_{B^{1}_{\infty,1}} \,\dd\t\,\leq\,2 \Bigg\}\,.
\end{align*}
Then, from \eqref{est:u-partial} we deduce, for $ t \in [0,{T_1} ] $, the inequality
\begin{align}
\label{est:u-unif_T_1}
\left\| \wu(t) \right\|_{B^{1}_{\infty,1}}\, \lesssim \, \left\| u_0\right\|_{B^{1}_{\infty,1}}\,
\exp\left(R_1\, e^{2C}\,\left(\frac{1}{\alpha}\,\| u_0 \|_{L^2} \,+\,1 \right)\right)\,.
\end{align}
This in particular implies that, for all $ t \in [0, {T}_1] $, it holds
\begin{equation*}
\int^t_0e^{-\alpha\t}\,\left\| \wu(\t) \right\|_{B^{1}_{\infty,1}} \,\dd\t \,\lesssim \,
\frac{1}{\alpha}\,\left\| u_0\right\|_{B^{1}_{\infty,1}}\,
\exp\left(R_1\, e^{2C}\,\left(\frac{1}{\alpha}\,\| u_0 \|_{L^2} \,+\,1 \right)\right)\,.
\end{equation*}
Now, by definition of ${T}_1 $, we must have
\begin{equation*}
  2\,\leq\,\frac{1}{\alpha}\,\left\| u_0\right\|_{B^{1}_{\infty,1}}\,
\exp\left(R_1\, e^{2C}\,\left(\frac{1}{\alpha}\,\| u_0 \|_{L^2} \,+\,1 \right)\right)\,,
\end{equation*}
which, however, cannot be verified if the right-hand side is small enough.
This implies that $ {T}_1 = T $, hence 
\[
 \sup_{t\in[0,T[\,}\left\|\wtilde u(t)\right\|_{B^1_{\infty,1}}\,<\,+\,\infty\,,
\]
owing to \eqref{est:u-unif_T_1}, but this property is clearly in contradiction with \eqref{est:Omega_to-prove}.
Therefore, we infer that the constructed solution must be globally defined in time.

In addition, analogously to what done in Subsection \ref{ss:g=1_d=2}, we notice that our argument also implies the following bounds:
first of all,
\[
\forall\,t\in\R_+\,,\qquad\qquad 
\left\|u(t)\right\|_{L^2\cap B^1_{\infty,1}}\,\lesssim\,e^{-\alpha\,t}\,\left\|u_0\right\|_{L^2\cap B^1_{\infty,1}}\,,
\]
owing to inequalities \eqref{est:wu-en} and \eqref{est:u-unif_T_1}, and secondly
\[
\forall\,t\in\R_+\,,\qquad\qquad \rho_*\leq\rho(t)\leq\rho^*\qquad \mbox{ and }\qquad
\left\|\rho(t)-1\right\|_{B^1_{\infty,1}}\,\lesssim\,\left\|\rho_0-1\right\|_{B^1_{\infty,1}}\,,
\]
owing to \eqref{est:dens-Besov}. Similarly, we obtain the time-decay of the Besov norm of $\nabla\Pi$.
With those bounds at hand, by working on system \eqref{eq:dd-E-wu}, we can propagate the (higher order, in the case $s>1$)
Besov norms of $\rho$, $\wu$ and $\nabla\wtilde\Pi$, globally in time,
getting the sought regularity properties.

Remark that, in order to complete the proof to Theorem \ref{th:g=1_d}, it remains us to prove the time decays of the higher order Besov norms of the solution.
As it was the case in Subsection \ref{ss:g=1_d=2}, this fact will be shown in Subsection \ref{ss:higher_decay} below.

%

\subsection{Decay estimates for higher order Besov norms} \label{ss:higher_decay}

In this subsection, we prove that not only the $B^1_{\infty,1}$ norm of $u$ and $\nabla \Pi$ decay exponentially in time, but also
the higher order Besov norms $B^s_{p,r}$. This will conclude the proof of both Theorems \ref{th:g=1_d} and \ref{th:g=1_2}.

\begin{lemma}
\label{per:of:norm}
Let $d\geq2$, $\g=1$ and $\alpha>0$ fixed. Take indices $(s,p,r)\,\in\,\R\times[1,+\infty]\times[1,+\infty]$ such that $p\geq2$ and on of the two conditions in \eqref{cond:Lipschitz}
is verified.
Take an initial datum $\big(\rho_0,u_0\big)$ such that conditions \eqref{eq:in-datum}-\eqref{eq:vacuum} are satisfied, for two suitable constants
$0<\rho_*\leq\rho^*$. Assume that $u_0\in L^2\cap B^s_{p,r}$ and $\nabla\rho_0\in B^{s-1}_{p,r}$ and that, in addition, there exists a unique
global in time solution $\big(\rho,u,\nabla\Pi\big)$ to system \eqref{eq:dd-E} associated to that initial datum,
verifying the regularity properties (i)-(ii)-(iii) of Theorem \ref{th:g=1_d}.
Finally, assume that there exists a constant $C>0$ such that
\begin{equation} \label{ass:decay-low}
\forall\,t\geq 0\,,\qquad\qquad \left\|\nabla\rho(t)\right\|_{B^0_{\infty,1}}\,+\,e^{\alpha\,t}\,\left\|u(t)\right\|_{L^2\cap B^1_{\infty,1}}\,+\,e
^{2\,\alpha\,t}\,\left\|\nabla\Pi(t)\right\|_{L^2\cap B^1_{\infty,1}}\,\leq\,C\,.
\end{equation}

Then, for a possibly different constant $\wtilde C>0$, one has
\[
\forall\,t\geq 0\,,\qquad\qquad
\left\|\nabla\rho(t)\right\|_{B^{s-1}_{p,r}}\,+\,e^{\alpha\,t}\,\left\|u(t)\right\|_{L^2\cap B^s_{p,r}}\,+\,
e^{2\,\alpha\,t}\,\left\|\nabla\Pi(t)\right\|_{L^2\cap B^s_{p,r}}\,\leq\,\wtilde{C}\,.
\]
\end{lemma}

\begin{proof}
Without loss of generality, we can assume to have either $2\leq p<+\infty$, or $p=+\infty$ and $s>1$ (otherwise, the decay estimates have already been proved
in the previous Sections \ref{ss:g=1_d-geq3} and \ref{ss:g=1_d=2}). In order to prove the previous statement, we are going to repeat the main steps of the proof of the continuation criterion, trying to
making each time a lower order term appear, which possesses the right decay property by assumption.

We start by estimating the density gradient: by using the notation introduced in \eqref{eq:change}, from inequality \eqref{lin:est:nab:a} we gather
\begin{equation*}
\left\|\nabla \rho(t)\right\|_{B^{s-1}_{p,r}}\,\lesssim\,\left\|\nabla\rho_0\right\|_{B^{s-1}_{p,r}}\,+\,
\int^t_0 e^{-\alpha \t}\,\left( \left\|\nabla \wtilde{u} \right\|_{L^{\infty}}\,\|\nabla \rho\|_{B^{s-1}_{p,r}}\, +\,
\|\nabla \rho \|_{L^{\infty}}\, \left\|\nabla \wtilde{u}\right\|_{B^{s-1}_{p,r}}\right)\,\dt
\end{equation*}
for any time $t\geq0$. Notice that, as $s>1$ always, the space $B^{s-1}_{p,r}$ is an algebra.

By transport estimates, together with commutator and tame estimates, we can bound the norm of the velocity field, for any time $t\geq0$, as follows:
\begin{align*}
\left\|\wtilde{u}(t)\right\|_{B^{s}_{p,r}}\,&\lesssim\,
\left\|\wtilde{u}_0\right\|_{B^{s}_{p,r}}\,+\,\int^t_0 e^{-\alpha \t}\,\Big(\left\|\nabla \wtilde{u} \right\|_{L^{\infty}}\,
\left\|\wtilde{u}\right\|_{B^{s}_{p,r}} \\ 
&\qquad\qquad\qquad\qquad\qquad\qquad +\,e^{\alpha \t}\,\frac{1}{\rho_*}\,\left\|\nabla\wtilde\Pi\right\|_{B^s_{p,r}}\,
+\,e^{\alpha \t}\,\left\|\nabla\wtilde\Pi\right\|_{L^\infty}\,\left\|\nabla\rho\right\|_{B^{s-1}_{p,r}}\Big)\,\dt\,. 
\end{align*}

At this point, we make use of assumption \eqref{ass:decay-low} in the previous estimates: taking advantage also of the inequality of Lemma \ref{lem:pres:est:Besov}
applied to the pressure equation \eqref{eq:Pi_g=1}, we obtain
\begin{align*}
\left\|\nabla \rho(t)\right\|_{B^{s-1}_{p,r}}\,&\lesssim\,\left\|\nabla\rho_0\right\|_{B^{s-1}_{p,r}}\,+\,
\int^t_0 e^{-\alpha \t}\,\left( \left\|\nabla \rho\right\|_{B^{s-1}_{p,r}}\, +\, \left\|\wtilde{u}\right\|_{B^{s}_{p,r}}\right)\,\dt \\
\left\|\wtilde{u}(t)\right\|_{B^{s}_{p,r}}\,&\lesssim\,\left\|\wtilde{u}_0\right\|_{B^{s}_{p,r}}\,+\,\int^t_0 e^{-\alpha \t}\,
\Big(\left\|\wtilde{u}\right\|_{B^{s}_{p,r}} \, + \, \left\|\nabla\rho\right\|_{B^{s-1}_{p,r}}\Big)\,\dt\,. 
\end{align*}
Then, the Gr\"onwall lemma implies the desired estimate: for all $t\geq0$, one has
\begin{equation*}
\left\|\nabla \rho(t)\right\|_{B^{s-1}_{p,r}}\, +\, \left\|\wtilde{u}(t)\right\|_{B^{s}_{p,r}}\, \lesssim\,
\left\|\nabla \rho_0 \right\|_{B^{s-1}_{p,r}}\, +\, \left\|\wtilde{u}_0 \right\|_{B^{s}_{p,r}}\,.
\end{equation*}
Finally, estimate \eqref{gen:pres:ine} applied to \eqref{eq:Pi_g=1} implies the exponential decay also for the pressure term.
The lemma is thus completely proved.
\end{proof}

\section{The case $\g=0$: proof of Theorem \ref{th:g=0_d}} \label{s:g=0}

The proof of the global existence in the case $\g=0$ is more delicate. As a matter of fact, when $\g=0$, the momentum equation reads
\begin{equation} \label{eq:mom_g=0}
\rho\,\d_tu\,+\,\rho\,u\cdot\nabla u\,+\,\nabla\Pi\,+\,\alpha\,u\,=\,0\,.
\end{equation}
Thus, even though the term $\mf D^0_\alpha(\rho,u)\,=\,\alpha\,u$ is still a damping term, 
it creates complications when looking for higher order estimates for the solution. As a matter of fact, for obtaining higher order estimates on $u$,
we divide the momentum equation by $\rho$, but, when doing so, $\mf D^0_\alpha(\rho,u)$ presents a variable coefficient $\frac{1}{\rho}$ in front,
keep in mind equation \eqref{eq:mom_2}. As a consequence, it is not clear which is the optimal decay rate for the solution of \eqref{eq:mom_g=0}.

Recall that, given an initial datum $\big(\rho_0,u_0\big)$ verifying the assumptions of Theorem \ref{th:g=0_d}, we are able to construct a local solution
$\big(\rho,u,\nabla\Pi\big)$ defined on some maximal time interval $[0,T[\,$. As done in Section \ref{s:g=1}, our goal is to prove that $T=+\infty$
by resorting to the continuation criterion of Lemma \ref{lem:cont:pri}. By contradiction, in what follows we assume that $T<+\infty$ instead.

\subsection{Preliminaries and reformulation} \label{ss:prel-ref}

To begin with, we observe that the bound \eqref{est:rho-inf} still holds true for $\rho$, owing to the transport structure of the mass conservation equation:
\begin{equation} \label{est:rho-inf_0}
\forall\,t\in[0,T[\;,\qquad\qquad \rho_*\,\leq\,\rho(t)\,\leq\,\rho^*\,.
\end{equation}

Next, using that $ \mathfrak{D}_{\alpha}^{0}(\rho, u ) = \alpha u $ in \eqref{ene:est:glob:form}, we deduce
\[
\frac{1}{2}\,\frac{\dd}{\dt}\int_{\R^d}\rho\,|u|^2\,\dx\,+\,\alpha\,\int_{\R^d}|u|^2\,\dx\,=\,0.
\]
Thus, after defining
\[
 \alpha_*\,:=\,\frac{\alpha}{\rho^*}\,,
\]
we easily find
\[
\frac{1}{2}\,\frac{\dd}{\dt}\left\|\sqrt{\rho}\,u\right\|_{L^2}^2\,+\,\alpha_*\,\left\|\sqrt{\rho}\,u\right\|_{L^2}^2\,\leq\,0\,,
\]
which in turn implies, thanks also to the use of \eqref{est:rho-inf_0}, the inequality
\begin{equation} \label{est:u-L^2_0}
\forall\,t\in[0,T[\;,\qquad\qquad
\left\|u(t)\right\|_{L^2}\,\lesssim\,e^{-\alpha_*t}\,\left\|u_0\right\|_{L^2}\,.
\end{equation}

Next, we turn our attention to the higher order estimates. The two key ideas in order to bypass the difficulty mentioned above are the following ones. First of all,
we decompose the damping term $\mf D^0_\alpha(\rho,u)$ into two parts: one of them will give an exponential decay (even though with a non-optimal rate), while
the other one will be a positive term, yielding a contribution with the right sign in the estimates.
Once implemented in the Besov functional framework, this procedure creates remainder terms, which are however of lower order:
the second main ingredient of the proof is then to use an interpolation argument to make a $L^2$ norm of $u$ appear, for which one can use the (optimal) decay
estimate \eqref{est:u-L^2_0}.

Let us make explicit what we have just announced. Given any
\[
0\,<\,\beta\,\leq\,\alpha_*\,:=\,\frac{\alpha}{\rho^*}\,,
\]
we can write the momentum equation \eqref{eq:mom_g=0} as
\[
\d_tu\,+\,u\cdot\nabla u\,+\,\frac{1}{\rho}\,\nabla\Pi\,+\,\left(\frac{\alpha}{\rho}\,-\,\beta\right)\,u\,+\,\beta\,u\,=\,0\,.
\]
In other terms, if we define
\[
u_\beta\,:=\,e^{\bt\,t}\,u\qquad\qquad \mbox{ and }\qquad\qquad \Pi_\bt\,:=\,e^{\bt\,t}\,\Pi\,,
\]
we get the equation
\begin{equation} \label{eq:u_beta}
\d_tu_\bt\,+\,e^{-\bt t}\,u_\bt\cdot\nabla u_\bt\,+\,\frac{1}{\rho}\,\nabla\Pi_\bt\,+\,\left(\frac{\alpha}{\rho}\,-\,\beta\right)\,u_\bt\,=\,0\,.
\end{equation}
Analogously, the mass conservation equation can be recasted as
\begin{equation} \label{eq:mass-beta}
\d_t\rho\,+\,e^{-\bt t}\,u_\bt\cdot\nabla\rho\,=\,0\,.
\end{equation}

Now, similarly to what done in Section \ref{s:g=1}, we observe that the failure of condition \eqref{est:Du} of Lemma \ref{lem:cont:pri}
implies that
\begin{equation} \label{est:to-prove_0}
\int^T_0\left\|u(t)\right\|_{B^1_{\infty,1}}\,\dt\,=\,\int^T_0e^{-\,\bt\,t}\,\left\|u_\bt(t)\right\|_{B^1_{\infty,1}}\,\dt\,=\,+\,\infty\,.
\end{equation}
Thus, we are going to bound the Besov norm of $u_\bt$ and show that, under the smallness condition of Theorem \ref{th:g=0_d}, this cannot happen.

\subsection{Estimates for the Besov norm of the velocity field} \label{ss:est-Besov}

In this part, we are going to bound the Besov norm of $u_\bt$, where we recall that the value of $\beta\in\,]0,\alpha_*]$ has not been fixed yet.
We have to proceed carefully, in order to exploit (as anticipated above) the positive sign of the last term appearing in the left-hand side of equation \eqref{eq:u_beta}.

Applying the operator $\Delta_j$ of a Littlewood-Paley decomposition (see Section \ref{s:tools}) to relation \eqref{eq:u_beta}, we find
\begin{equation} \label{eq:D_ju_beta}
\d_t\Delta_ju_\bt\,+\,e^{-\bt t}\,u_\bt\cdot\nabla\Delta_ju_\bt\,+\,\left(\frac{\alpha}{\rho}\,-\,\bt\right)\,\Delta_ju_\bt\,=\,
-\,\Delta_j\mc P_\bt\,+\,\mbb C^1_j\,+\,\mbb C^2_j\,,
\end{equation}
where we have defined the pressure-density term
\[
 \mc P_\bt\,:=\,\frac{1}{\rho}\,\nabla\Pi_\bt
\]
and the two commutator terms
\begin{align*}
 \mbb C^1_j\,:=\,e^{-\bt t}\,\left[u_\bt\cdot\nabla\,,\,\Delta_j\right]u_\bt\qquad\qquad \mbox{ and }\qquad\qquad
 \mbb C^2_j\,:=\,\alpha\,\left[\frac{1}{\rho}\,,\,\Delta_j\right]u_\bt\,.
\end{align*}
Now, we perform a $L^q$ estimate, for $2\leq q<+\infty$, on the transport-type equation \eqref{eq:D_ju_beta}: 
we get
\begin{align*}
&\left\|\Delta_ju_\bt(t)\right\|_{L^q}\,+\,
\int^t_0\int_{\R^d}\left(\frac{\alpha}{\rho}\,-\,\bt\right)\,\frac{\left|\Delta_ju_\bt\right|^q}{\left\|\Delta_ju_\bt\right\|^{q-1}_{L^q}}\,\dx\,\dd\t \\
&\qquad\qquad\qquad \lesssim\,\left\|\Delta_ju_\bt(0)\right\|_{L^q} \,+\,
\int^t_0\Big(\left\|\Delta_j\mc P_\bt\right\|_{L^q} \,+\,\left\|\mbb C^1_j\right\|_{L^q} \,+\,\left\|\mbb C^2_j\right\|_{L^q} \Big)\,\dd\t\,
\end{align*}
for any $t\in[0,T[\,$. Using the positivity property $\frac{\alpha}{\rho}\,-\,\bt\geq0$ and sending $q\to+\infty$ yield
\[
\left\|\Delta_ju_\bt(t)\right\|_{L^\infty}\,\lesssim\,\left\|\Delta_ju_\bt(0)\right\|_{L^\infty}\,+\,
\int^t_0\Big(\left\|\Delta_j\mc P_\bt\right\|_{L^\infty}\,+\,\left\|\mbb C^1_j\right\|_{L^\infty}\,+\,\left\|\mbb C^2_j\right\|_{L^\infty}\Big)\,\dd\t\,.
\]
From the previous bound, by multiplying both sides by $2^j$ and summing up over $j\geq-1$, in turn we deduce an estimate for the $B^1_{\infty,1}$ norm of $u_\bt$:
for any $t\in[0,T[\,$, we obtain
\begin{align} \label{est:u_bt-partial}
\left\|u_\bt(t)\right\|_{B^1_{\infty,1}}\,\lesssim\,\left\|u_0\right\|_{B^1_{\infty,1}}\,+\,\int^t_0\left(\left\|\mc P_\bt\right\|_{B^1_{\infty,1}}\,+\,
\sum_{j=-1}^{+\infty}2^j\,\left\|\mbb C^1_j\right\|_{L^\infty}\,+\,\sum_{j=-1}^{+\infty}2^j\,\left\|\mbb C^2_j\right\|_{L^\infty}\right)\,\dd\t\,.
\end{align}

Let us now bound one by one the terms appearing under the integral on the right-hand side of the previous inequality.

\subsubsection*{The commutator terms}

The estimate of the first commutator term is classical: by using Lemma \ref{l:CommBCD}, we get
\[
\sum_{j=-1}^{+\infty}2^j\,\left\|\mbb C^1_j\right\|_{L^\infty}\,\lesssim\,e^{-\bt t}\,\left\|\nabla u_\bt\right\|_{L^\infty}\,\left\|u_\bt\right\|_{B^1_{\infty,1}}\,.
\]

Next, let us focus on the second commutator term appearing in \eqref{est:u_bt-partial}.
By Lemma \ref{com:damping} and the paraliearisation result of Proposition \ref{p:comp},
we can bound 
it as follows:
\begin{align*}
\sum_{j=-1}^{+\infty}2^j\,\left\|\mbb C^2_j\right\|_{L^\infty}\,&\lesssim\,\alpha\,\left\|\rho-1\right\|_{B^1_{\infty,1}}\,\left\|u_\bt\right\|_{B^{0}_{\infty,1}}  \\
&\lesssim\,\alpha\,\left\|\rho-1\right\|_{B^1_{\infty,1}}\,\left\|u_\bt\right\|^\theta_{L^2}\,\left\|u_\bt\right\|^{1-\theta}_{B^{1}_{\infty,1}} \\
&\lesssim\,\alpha\,
\left\|\rho-1\right\|_{B^1_{\infty,1}}\,e^{-\theta(\alpha_*-\bt)t}\,\left\|u_0\right\|^\theta_{L^2}\,\left\|u_\bt\right\|^{1-\theta}_{B^{1}_{\infty,1}}\,,
\end{align*}
for a suitable exponent $\theta=\theta(d)\in\,]0,1[\,$ only depending on the space dimension $d\geq2$.
Indeed, one has the following chain of continuous embeddings:
$B^1_{\infty,1}\hookrightarrow B^0_{\infty,1}\hookrightarrow B^{-d/2-\de}_{\infty,1}$, where $\de>0$ is arbitrarily small.
Hence, noticing that $L^2\hookrightarrow B^{-d/2}_{\infty,2}\hookrightarrow B^{-d/2-\de}_{\infty,1}$, we can apply an interpolation argument
and deduce the existence of such an exponent $\theta$. In addition, in passing from the penultimate inequality to the last one,
we have used the exponential decay given by \eqref{est:u-L^2_0}.

In order to simplify the notation in the argument below, from now on we fix the value of $\beta=\bt_*\in\,]0,\alpha_*[\,$ such that
\begin{equation} \label{eq:beta-theta}
\beta_*\,=\,\theta\,\big(\alpha_*\,-\,\bstar\big)\qquad\qquad 
\Longrightarrow\qquad\qquad  \bt_*\,=\,\frac{\theta}{1+\theta}\,\alpha_*\,.
\end{equation}
This choice implies that the two commutator terms have the same exponential decay.
Then, plugging the previous bounds into \eqref{est:u_bt-partial}, we infer the inequality
\begin{align}
\label{est:u_*_partial}
\left\|u_{\bstar}(t)\right\|_{B^1_{\infty,1}}\,&\lesssim\,\left\|u_0\right\|_{B^1_{\infty,1}}\,+\,
\int^t_0\Bigg(\left\|\rho\right\|_{B^1_{\infty,1}}\,\left\|\nabla\Pi_{\bstar}\right\|_{B^1_{\infty,1}}\,+\,
e^{-\bt_*\t}\,\left\|\nabla u_{\bstar}\right\|_{L^\infty}\,\left\|u_{\bstar}\right\|_{B^1_{\infty,1}} \\
\nonumber
&\qquad\qquad\qquad\qquad\qquad\qquad \,+\,\alpha\,e^{-\bt_*\t}\,
\left\|\rho-1\right\|_{B^1_{\infty,1}}\,\left\|u_0\right\|^\theta_{L^2}\,\left\|u_{\bt_*}\right\|^{1-\theta}_{B^1_{\infty,1}}\Bigg)\,\dd\t
\end{align}
for all $t\in[0,T[\,$, where we have also used that
\[
\left\|\mc P_\bt\right\|_{B^1_{\infty,1}}\,\lesssim\,\left\|\rho\right\|_{B^1_{\infty,1}}\,\left\|\nabla\Pi_\bt\right\|_{B^1_{\infty,1}}\,,
\]
as a consequence of the fact that $B^1_{\infty,1}$ is an algebra and of Proposition \ref{p:comp}.

\subsubsection*{Bounds for the density and the pressure} 

We still have to bound the density and the pressure gradient appearing in estimate \eqref{est:u_*_partial}.
The analysis is similar to the one performed in the previous section, so we will only sketch it.

First of all, we observe that, from the mass equation \eqref{eq:mass-beta} and classical transport estimates in Besov spaces, we have,
for any $t\in[0,T[\,$, the inequality
\begin{equation} \label{est:dens_*}
\left\|\rho(t)-1\right\|_{B^1_{\infty,1}}\,\lesssim\,\left\|\rho_0-1\right\|_{B^1_{\infty,1}}\,
\exp\left(C\int^t_0e^{-\bt_*\t}\,\left\|u_{\beta_*}(\t)\right\|_{B^1_{\infty,1}}\,\dd\t\right)\,.
\end{equation}

As for the pressure, we start by writing an equation for $\Pi_{\bstar}$: by applying the divergence operator to both sides of relation \eqref{eq:u_beta}, we find
\[
-\,\div\left(\frac{1}{\rho}\,\nabla\Pi_{\bstar}\right)\,=\,\div\left(e^{-\bt_*t}\,u_{\bstar}\cdot\nabla u_{\bstar}\,+\,\frac{\alpha}{\rho}\,u_{\bstar}\right)\,,
\]
owing to the fact that $\div u_{\bstar}\,=\,0$.
Before applying estimate \eqref{gen:pres:ine} with $ F\, =\,e^{-\bt_*t}\,u_{\bstar}\cdot\nabla u_{\bstar}\,+\,\frac{\alpha}{\rho}\,u_{\bstar} $, let us notice that
\begin{align}
\label{L2:F}
\left\| e^{-\bt_*t}\,u_{\bstar}\cdot\nabla u_{\bstar}\,+\,\frac{\alpha}{\rho}\,u_{\bstar} \right\|_{L^2}\, &\lesssim\, e^{-\bt_*t}\,\left\|u_{\bstar}\right\|_{L^2}\,
\left\|\nabla u_{\bstar}\right\|_{L^{\infty}}\,+\, \alpha\, \left\|u_{\bstar}\right\|_{L^2} \\
\nonumber
&\lesssim \, e^{-\bt_*t}\,\left\|u_{\bstar}\right\|_{L^2}\,\left\|\nabla u_{\bstar}\right\|_{L^{\infty}}\,+\, \alpha\,  e^{-\bt_*t}\, \left\|u_{0}\right\|_{L^2}\,,
\end{align}
where in the last inequality we have used \eqref{est:u-L^2_0} and the fact that  $ \bstar = \theta(\alpha_ *-\bstar) \leq (\alpha_ *-\bstar) $. Moreover,
by arguing similarly to Section \ref{ss:g=1_d=2}, see below equation \eqref{eq:Pi_g=1}, one can bound
\begin{equation}
\label{B1:F}
\left\| \rho\, \div(e^{-\bt_*t}\,u_{\bstar}\cdot\nabla u_{\bstar}) \right\|_{B^0_{\infty,1}}\, \lesssim\,
e^{-\bt_*t}\,\| \rho \|_{B^{1}_{\infty,1}}\, \left\|u_{\bstar}\right\|_{B^{1}_{\infty,1}}^2\,.
\end{equation}
%
%
Finally, we notice that the following equalities hold true:
\[
 \rho\,\div\left(\frac{\alpha}{\rho}\,u_{\bstar}\right)\,=\,-\,\alpha\,\nabla\log\rho\cdot u_{\bstar} \,=\,-\,\alpha\,\div\big(\log\rho\;u_{\bstar}\big)\,.
\]
Using \eqref{algebra:B^0_2} and the paralinearization theorem, we deduce that 
\begin{align}
\label{B2:F}
\left\|\rho\,\div\left(\frac{\alpha}{\rho}\,u_{\bstar}\right)\right\|_{B^0_{\infty,1}}\,\lesssim\,\alpha \, \left\|\rho \right\|_{B^1_{\infty,1}}\,
\left\|u_{\bt_*}\right\|_{B^{0}_{\infty,1}}\,\lesssim \, \alpha \,e^{-\bt_*t}\,\left\|\rho \right\|_{B^1_{\infty,1}}\,\left\|u_0\right\|_{L^2}^\theta\,
\left\|u_{\bstar}\right\|_{B^1_{\infty,1}}^{1-\theta}\,,
\end{align}
where $\theta\in\,]0,1[\,$ is the same exponent appearing above.
Estimates \eqref{L2:F}, \eqref{B1:F} and \eqref{B2:F} together with \eqref{gen:pres:ine} imply the bound
\begin{align} \label{est:Pi_g=0_Bes}
\left\|\nabla\Pi_{\bstar}\right\|_{B^1_{\infty,1}}\,
&\lesssim\,e^{-\bt_*t}\,\left(1\,+\,\|\rho-1\|^\eta_{B^1_{\infty,1}}\right) \\
\nonumber
&\qquad\qquad\times\,
\left(\alpha \left\|u_{0}\right\|_{L^2}\,+\,\left\| u_{\bstar}\right\|_{L^2\cap B^1_{\infty,1}}^2\,+\,\alpha \left\|u_0\right\|_{L^2}^\theta\,
\left\|u_{\bstar}\right\|_{B^1_{\infty,1}}^{1-\theta}\right) \,.
\end{align}

\subsection{End of the argument} \label{ss:end}

We can now plug estimates \eqref{est:dens_*} and \eqref{est:Pi_g=0_Bes} into \eqref{est:u_*_partial}.
After defining
\[
 R_2\,:=\,
 1\,+\,\|\rho_0-1\|^{\eta+1}_{B^1_{\infty,1}}\,,
\]
we get that the following inequalities hold true, for any $t\in[0,T[\,$:
\begin{align*} 
\left\|u_{\bstar}(t)\right\|_{B^1_{\infty,1}}\,&\lesssim\,\left\|u_0\right\|_{B^1_{\infty,1}}\,+\,
\int^t_0e^{-\bt_*\t}\,\left(1\,+\,\|\rho-1\|^{\eta+1}_{B^1_{\infty,1}}\right) \\
\nonumber
&\qquad\qquad\qquad\qquad
\times\,\left( \alpha \left\|u_{0}\right\|_{L^2}\,+\,\left\| u_{\bstar}\right\|_{L^2\cap B^1_{\infty,1}}^2\,+\,\alpha \left\|u_0\right\|_{L^2}^\theta\,
\left\|u_{\bstar}\right\|_{B^1_{\infty,1}}^{1-\theta} \right)\,\dd\t \\
&\lesssim\,\left\|u_0\right\|_{B^1_{\infty,1}}\,+\,R_2
\int^t_0e^{-\bt_*\t}\,\exp\left(C\,(\eta+1)\,\int^\t_0e^{-\bt_*\s}\,\left\|u_{\bstar}\right\|_{B^1_{\infty,1}}\,\dd\s\right) \\
\nonumber
&\qquad\qquad\qquad\qquad
\times\,\left(\alpha \left\|u_{0}\right\|_{L^2}\,+\,\left\| u_{\bstar}\right\|_{L^2\cap B^1_{\infty,1}}^2\,+\,\alpha \left\|u_0\right\|_{L^2}^\theta\,
\left\|u_{\bstar}\right\|_{B^1_{\infty,1}}^{1-\theta} \right)\,\dd\t \\
&\lesssim\,\left\|u_0\right\|_{B^1_{\infty,1}}\,+ \alpha \|u_0\|_{L^2} \,R_2
\int^t_0e^{-\bt_*\t}\,\exp\left(C\,(\eta+1)\,\int^\t_0e^{-\bt_*\s}\,\left\|u_{\bstar}\right\|_{B^1_{\infty,1}}\,\dd\s\right) \\
& \qquad\qquad    + \,R_2 \int^t_0e^{-\bt_*\t}\,\exp\left(C\,(\eta+1)\,\int^\t_0e^{-\bt_*\s}\,\left\|u_{\bstar}\right\|_{B^1_{\infty,1}}\,\dd\s\right) \nonumber \\
\nonumber
&\qquad\qquad\qquad\qquad\qquad\qquad\qquad\qquad\qquad
\times\,\left( \,\left\| u_{\bstar}\right\|_{L^2\cap B^1_{\infty,1}}^2\,+\,\alpha \,
\left\|u_{\bstar}\right\|_{B^1_{\infty,1}} \right)\,\dd\t\,,
\end{align*}
where, for passing from the penultimate bound to the last one, we have used the Young inequality.

At this point, the argument can be closed in a very similar fashion as done in Subsections \ref{ss:g=1_d=2} and \ref{ss:g=1_d-geq3} above.
Let us introduce the notation
\[
U_{\bstar}(t)\,:=\,\left\|u_{\bstar}(t)\right\|_{L^2}\,+\,\left\|u_{\bstar}(t)\right\|_{B^1_{\infty,1}}\,,
\]
with obvious changes when $t=0$.
Then, from the previous estimate for the Besov norm of $u_{\bstar}$ and an application of the Gr\"onwall inequality, for any $t\in[0,T[\,$ we infer
that
\begin{align}
\label{est:U_*_provv}
U_{\bstar}(t)\,&\lesssim\,e^{\alpha\,R_2\,h(t)}\,\Bigg( U_{\bstar}(0)\,\big( 1 \,+ \alpha \,R_2\,h(t)\big)
\\
\nonumber
&\qquad\qquad\qquad\qquad\qquad
+\,R_2\int^t_0 e^{-\bstar\t}\,e^{C\int^\t_0e^{-\bstar\s}\,U_{\bstar}(\s)\,\dd\s}\,\big(U_{\bstar}(\t)\big)^2 \,\dd\t\, \Bigg)\,,
\end{align}
for a new constant $C>0$, where we have defined
\[
 h(t)\,:=\,\int^t_0 e^{-\bstar\t}\,e^{C\int^\t_0e^{-\bstar\s}\,U_{\bstar}(\s)\,\dd\s}\,\dd\t\,.
\]

Now, let us define the time $T_2>0$ as
\begin{align*}
T_2\,:=\,\sup\Bigg\{ \, t\in[0,T[\;\Big|\;  & \, \int^t_0e^{-\bstar\t}\,U_{\bstar}(\t)\,\dd\t\,\leq \, 2 \\ 
&  \quad \text{ and  } \quad \,R_2\int^t_0 e^{-\bstar\t}\,e^{C\int^\t_0e^{-\bstar\s}\,U_{\bstar}(\s)\,\dd\s}\,\big(U_{\bstar}(\t)\big)^2 \,\dd\t\,\leq\,4\,U_{\bstar}(0)\Bigg\}\,.
\end{align*}
%
%
First of all, we notice that, for any $t\in[0,T_2]$, one has
\begin{align*}
\alpha \,R_2\,h(t)
\, \leq\, R_2\, \frac{\alpha}{\bstar}\, e^{2C}\, =\, R_2\, \frac{1+\theta}{\theta}\,\rho^{*}\, e^{2C}\, =\, R_{2}\, \mathfrak{C}\, , 
\end{align*}
where we have used the definition of $ \bstar$, $ \alpha_{*} $ and $ \alpha $. Observe that the new constant $ \mathfrak{C} $
depends only on the fixed quantities $ \rho_{*} $ and $ \rho^*$ appearing in \eqref{eq:vacuum}, but not on $\alpha$, $ \bstar $ and $ u_0 $.
 
By the previous inequality and inequality \eqref{est:U_*_provv}, we gather, for any $t\in[0,T_2]\,$, the following bound:
\begin{equation} \label{est:U_*-bound}
U_{\bstar}(t)\,\lesssim\,\left(5\, +\, R_2\, \mathfrak{C}\right)\,e^{R_2\,\mathfrak{C}} \, U_{\bstar}(0)\,.
\end{equation}
In particular, this implies on the one hand that, for all $ t\in[0,T_2] $, we have   
\begin{align*}
 \int^t_0e^{-\bstar\t}\,U_{\bstar}(\t)\,\dd\t\, \lesssim \, \left(5\, +\, R_2\, \mathfrak{C}\right)\,e^{R_2\, \mathfrak{C}} \, \frac{U_{\bstar}(0)}{\beta_{*}}
\end{align*}
and, on the other hand, that
\begin{align*}
\,R_2\int^t_0 e^{-\bstar\t}\,e^{C\int^\t_0e^{-\bstar\s}\,U_{\bstar}(\s)\,\dd\s}\,\big(U_{\bstar}(\t)\big)^2 \,\dd\t\,\lesssim\,
 \frac{\big(U_{\bstar}(0)\big)^2}{\beta_{*}}\, R_2 \, e^{2\,C}\, \left(5\, +\, R_2\, \mathfrak{C}\right)^2\,e^{2\,R_2\, \mathfrak{C}}\,.
\end{align*}
Therefore, by definition of the time $T_2$, we see that one must have
\[
\mbox{ either }\qquad 2\,\lesssim  \, \frac{U_{\bstar}(0)}{\beta_{*}}\,\left(5 + R_2 \mathfrak{C}\right)\,e^{R_2\,\mathfrak{C}} \qquad \mbox{ or }\qquad
4\,\lesssim \, \frac{U_{\bstar}(0)}{\beta_{*}} \,   R_2\, e^{2C}\,\left(5 + R_2 \mathfrak{C}\right)^2\,e^{2\,R_2\,\mathfrak{C}}\,.
\]
None of those inequalities can be satisfied, if the quantity $ U_{\bstar}(0)/ \beta_{*} $ is small enough, according to the condition imposed in Theorem \ref{th:g=0_d},
where we take $\mf R\,=\,R_2$.
This implies that one must have $T_2=T$. Hence, by \eqref{est:U_*-bound}, we see that $U_{\bstar}$ remains bounded in $[0,T[\,$, which then contradicts
the blow-up property stated in \eqref{est:to-prove_0}.

We thus deduce that $T=+\infty$, as sought.
In addition, arguing as done at the end of Subsection \ref{ss:g=1_d-geq3}, we recover the time decay of the critical Besov norms of the solution:
it follows from \eqref{est:U_*-bound} and \eqref{est:Pi_g=0_Bes} that
\[
\forall\,t\geq0\,,\qquad\qquad
\left\|u(t)\right\|_{L^2\cap B^1_{\infty,1}}\,+\,\left\|\nabla\Pi(t)\right\|_{L^2\cap B^1_{\infty,1}}\,\leq\,C\,e^{-\,\bstar\,t}\,,
\]
for a suitable constant $C>0$.

In order to complete the proof of Theorem \ref{th:g=0_d}, it remains us to show the time decay of the higher order Besov norms of the solution: this will be done
in the next subsection.

\subsection{Decay of higher regularity norms when $\g=0$} \label{ss:decay-higher_0}

Our next goal is to prove the decay of the high regularity norms of the solution, recall estimate \eqref{decay:est:2}.
This will follow from the counterpart of Lemma \ref{per:of:norm} in the case $\g=0$: the precise statement is the following one.

\begin{lemma}
\label{per:of:norm_0}
Let $d\geq2$, $\g=0$ and $\alpha>0$ fixed. Take indices $(s,p,r)\,\in\,\R\times[1,+\infty]\times[1,+\infty]$ such that $p\geq2$ and on of the two conditions in \eqref{cond:Lipschitz}
is verified.
Take an initial datum $\big(\rho_0,u_0\big)$ such that conditions \eqref{eq:in-datum}-\eqref{eq:vacuum} are satisfied, for two suitable constants
$0<\rho_*\leq\rho^*$. Assume that $u_0\in L^2\cap B^s_{p,r}$ and $\nabla\rho_0\in B^{s-1}_{p,r}$ and that, in addition, there exists a unique
global in time solution $\big(\rho,u,\nabla\Pi\big)$ to system \eqref{eq:dd-E} associated to that initial datum,
verifying the regularity properties (i)-(ii)-(iii) of Theorem \ref{th:g=1_d}.
Finally, assume that there exists a constant $C>0$ such that
\begin{equation} \label{ass:decay-low_0}
\forall\,t\geq 0\,,\qquad\qquad \left\|\nabla\rho(t)\right\|_{B^0_{\infty,1}}\,+\,e^{\bstar\,t}\,\left\|u(t)\right\|_{L^2\cap B^1_{\infty,1}}\,+\,
e^{\bstar\,t}\,\left\|\nabla\Pi(t)\right\|_{L^2\cap B^1_{\infty,1}}\,\leq\,C\,,
\end{equation}
where $\beta_*$ is the exponent defined in \eqref{eq:beta-theta}.

Then, for a possibly different constant $\wtilde C>0$ and for a new exponent $0<\beta_0<\beta_*$, defined in relation \eqref{def:bt_0} below,
one has
\[
\forall\,t\geq 0\,,\qquad\qquad
\left\|\nabla\rho(t)\right\|_{B^{s-1}_{p,r}}\,+\,e^{\beta_0\,t}\,\left\|u(t)\right\|_{L^2\cap B^s_{p,r}}\,+\,
e^{\beta_0\,t}\,\left\|\nabla\Pi(t)\right\|_{L^2\cap B^s_{p,r}}\,\leq\,\wtilde{C}\,.
\]
\end{lemma}

\begin{proof}
Notice that we can safely assume that $p<+\infty$, or $p=+\infty$ and $s>1$ here.
As for the proof of Lemma \ref{per:of:norm} above, the strategy is to reproduce fine estimates (in the same spirit of the continuation criterion)
on the Besov norm of the solution, and take care of making lower order Besov norms (for which the decay is guaranteed by assumption) appear.

By going along the lines of the computations of Subsection \ref{ss:est-Besov}, one easily sees that it is natural to define
\begin{equation} \label{def:bt_0}
\beta_0\,:=\,\frac{\theta_0}{1+\theta_0}\,\alpha_*\,, 
\qquad\qquad \mbox{ \tsl{i.e.} }\qquad \beta_0\,=\,\theta_0\,\big(\alpha_*\,-\,\beta_0\big)\,,
\end{equation}
where $\alpha_*\,:=\,\alpha/\rho^*$ as above and
$\theta_0$ is the exponent related to the interpolation
\[
 \left\|f\right\|_{B^{s-1}_{p,r}}\,\lesssim\,\left\|f\right\|^{\theta_0}_{L^2}\,\left\|f\right\|^{1-\theta_0}_{B^{s}_{p,r}}\,.
\]
Observe that, if we fix the index $\s=d/2+\de$ (where $\de>0$ can be taken arbitrarily small) in order to
have the embeddings $B^s_{p,r}\hookrightarrow B^{s-1}_{p,r}\hookrightarrow B^0_{\infty,1}\hookrightarrow B^{-\s}_{\infty,1}$,
we find that
\[
\theta\,=\,\frac{1}{1+\s}\qquad\qquad \mbox{ and }\qquad\qquad \theta_0\,=\,\frac{1}{s+\s}\,,
\]
where $\theta$ is the exponent appearing in \eqref{eq:beta-theta}. Then, as easy computation shows that $\beta_*/\beta_0>1$, namely
\[
0\,<\,\beta_0\,<\,\beta_*\,.
\]

This having been pointed out, our next goal is to perform $B^s_{p,r}$ estimates for $u_{\beta_0}$. The starting point is the analogous of \eqref{est:u_bt-partial}
in the $B^s_{p,r}$ framework, namely
\[ 
\left\|u_{\bt_0}(t)\right\|_{B^s_{p,r}}\,\lesssim\,\left\|u_0\right\|_{B^s_{p,r}}\,+\,\int^t_0\left(\left\|\mc P_{\bt_0}\right\|_{B^s_{p,r}}\,+\,
\left\|\left(2^{js}\,\left\|\mbb C^1_j\right\|_{L^p}\right)_j\right\|_{\ell^r}\,+\,
\left\|\left(2^{js}\,\left\|\mbb C^2_j\right\|_{L^p}\right)_j\right\|_{\ell^r}\right)\,\dd\t\,.
\] 

To begin with, we notice that, similarly to inequalities \eqref{B1:F} and \eqref{B2:F}, one has the estimates
\begin{align*}
\left\| \rho\, \div(e^{-\bt_0t}\,u_{\bt_0}\cdot\nabla u_{\bt_0}) \right\|_{B^{s-1}_{p,r}}\,&\lesssim\, 
e^{-\bt_0t}\,\Big(\|\nabla\rho \|_{B^{s-1}_{p,r}}\,\left\|u_{\bt_0}\right\|_{L^{\infty}}\,\left\|\nabla u_{\bt_0}\right\|_{L^{\infty}} \\
&\qquad\qquad\qquad\qquad\qquad\qquad +\,\rho^*\,\left\| \nabla u_{\bt_0} \right\|_{L^{\infty}}\, \left\| u_{\bt_0} \right\|_{B^{s}_{p,r}} \Big) \\
\left\|\rho\,\div\left(\frac{\alpha}{\rho}\,u_{\bt_0}\right)\right\|_{B^{s-1}_{p,r}}\, &\lesssim \,
\left\|\rho \right\|_{B^{s-1}_{p,r}}\, \left\|u_{\bt_0}\right\|_{L^{\infty}}\, +\,\left\|\nabla\rho\right\|_{L^\infty}\, \left\|u_{\bt_0}\right\|_{B^{s-1}_{p,r}} \\ 
&\lesssim\, \left\|\rho \right\|_{B^{s-1}_{p,r}}\, \left\|u_{\bt_0}\right\|_{L^{\infty}}\, +\, 
\left\|\nabla\rho\right\|_{L^\infty}\,\left\|u_{\bt_0}\right\|_{L^2}^{\theta_0}\,
\left\|u_{\bt_0}\right\|_{B^s_{p,r}}^{1-\theta_0} \\
&\lesssim\, \left\|\rho \right\|_{B^{s-1}_{p,r}}\, \left\|u_{\bt_0}\right\|_{L^{\infty}}\, +\, e^{-(\alpha_*-\bt_0)\theta_0t}\,\left\|\nabla\rho\right\|_{L^\infty}\,
\left\|u_{\bt_0}\right\|_{B^s_{p,r}}^{1-\theta_0}\,.
\end{align*}
Then, using also Lemma \ref{lem:pres:est:Besov}
and the definition of $\bt_0$ in \eqref{def:bt_0}, 
we deduce that
\begin{align*}
\left\| \nabla \Pi_{\beta_0}\right\|_{B^s_{p,r}}\,& \lesssim \,
\left(1 +\| \nabla \rho\|_{L^{\infty}}^{\eta}\right)\,\left( e^{-t \beta_0}\,\left\| u_{\beta_0}\right\|_{L^2}\,\left\| \nabla u_{\beta_0} \right\|_{L^{\infty}}
\,+\,\left\|u_{\bt_0}\right\|_{L^2}\right)\,+\,\left\| \nabla \rho \right\|_{B^{s-1}_{p,r}}\, \left\| \nabla \Pi_{\beta_0}\right\|_{L^{\infty}} \\
&\qquad\qquad +\,e^{-\bt_0t}\,\Big(\|\nabla\rho \|_{B^{s-1}_{p,r}}\,\left\|u_{\bt_0}\right\|_{L^{\infty}}\,\left\|\nabla u_{\bt_0}\right\|_{L^{\infty}}\,+\,
\left\| \nabla u_{\bt_0} \right\|_{L^{\infty}}\, \left\| u_{\bt_0} \right\|_{B^{s}_{p,r}}\Big)
\\
&\qquad\qquad\qquad\qquad  +\, e^{-\bt_0t}\,\left\|\nabla\rho\right\|_{L^\infty}\,
\left\|u_{\bt_0}\right\|_{B^s_{p,r}}^{1-\theta_0} 
\,+\,
\left(1\,+\,\left\|\nabla \rho \right\|_{B^{s-1}_{p,r}}\right)\,\left\|u_{\bt_0}\right\|_{L^{\infty}}\,.
\end{align*}
Using assumption \eqref{ass:decay-low_0} and recalling that $\bt_0<\bt^*$, in turn we find
\begin{align}
\label{0002}
\left\| \nabla \Pi_{\beta_0}\right\|_{B^s_{p,r}}\,&\lesssim \, e^{-t \beta_0}\,
\Big(\left\| u_{\bt_0} \right\|_{B^{s}_{p,r}}\, +\, \left\|  \nabla \rho \right\|_{B^{s-1}_{p,r}}\,+\,1\Big) \\
\nonumber
&\qquad +\,
e^{-(\alpha_*-\beta_0)\theta_0t}\,\left(1\,+\,\left\|u_{\bt_0}\right\|_{B^s_{p,r}}\right)\,+\,e^{-(\beta_*-\beta_0)t}\,
\left(1\,+\,\left\|\nabla \rho \right\|_{B^{s-1}_{p,r}}\right)\,,
\end{align}
where the implicit multiplicative constant depends also on $\|u_0\|_{L^2}$.
Then, making use of assumption \eqref{ass:decay-low_0} again, we infer the bound
\begin{align*}
\left\|\mc P_{\bt_0}\right\|_{B^s_{p,r}}\,&\lesssim\,\left\|\nabla \Pi_{\bt_0}\right\|_{B^s_{p,r}}\,+\,
\|\nabla \rho\|_{B^{s-1}_{p,r}}\,\left\|\nabla\Pi_{\bt_0}\right\|_{L^\infty}\,
\lesssim\,e^{-t \lambda}\,
\Big(\left\| u_{\bt_0} \right\|_{B^{s}_{p,r}}\, +\, \left\|  \nabla \rho \right\|_{B^{s-1}_{p,r}}\,+\,1\Big)\,,
\end{align*}
where we have denoted by $\lambda$ the minimum value among $\bt_0$, $(\alpha_*-\bt_0)\theta_0$ and $\bt_*-\bt_0$.

Let us now bound the commutator terms. On the one hand, an immediate application of Lemma \ref{l:CommBCD} yields
\begin{align*}
\left\|\left(2^{js}\,\left\|\mbb C^1_j\right\|_{L^p}\right)_j\right\|_{\ell^r}\,&\lesssim\,e^{-\bt_0t}\,\left\|\nabla u_{\bt_0}\right\|_{L^\infty}\,
\left\|u_{\bt_0}\right\|_{B^s_{p,r}}\lesssim\,e^{-\bt_0t}\,\left\|u_{\bt_0}\right\|_{B^s_{p,r}}\,\,.
\end{align*}
On the other hand, Lemma \ref{com:damping} implies that
\begin{align*}
\left\|\left(2^{js}\,\left\|\mbb C^2_j\right\|_{L^p}\right)_j\right\|_{\ell^r}\,&\lesssim\,
\| \rho \|_{B^1_{\infty,1}}\,\left\| u_{\beta_0} \right\|_{B^{s-1}_{p,r}}\, +\, \left\|\nabla \rho\right\|_{B^{s-1}_{p,r}}\, \left\|u_{\beta_0}\right\|_{B^1_{\infty,1}}\,.
\end{align*}

Putting all those estimates together, we finally arrive at the following inequality:
\begin{align*}
\left\|u_{\beta_0}(t)\right\|_{B^{s}_{p,r}}\,&\,\lesssim\,
\left\|u_{0}\right\|_{B^{s}_{p,r}}\,+\,\int^t_0e^{-\lambda\t}\left(
\left\| u_{\beta_0}\right\|_{B^{s}_{p,r}}\, +\,\left\|\nabla\rho\right\|_{B^{s-1}_{p,r}}\,+\,1\right) \,\dd\t \\
&\qquad\qquad\qquad\qquad +\,\int^t_0\bigg(\| \rho \|_{B^1_{\infty,1}}\,\left\| u_{\beta_0}\right\|_{B^{s-1}_{p,r}}\,+\,
\left\|\nabla \rho\right\|_{B^{s-1}_{p,r}}\,\left\|u_{\beta_0}\right\|_{B^1_{\infty,1}}\bigg)\,\dt\,. 
\end{align*}
At this point, we use the following two facts,
\begin{align*}
\left\|u_{\beta_0}\right\|_{B^1_{\infty,1}}\,&=\,e^{-t(\bt_*-\beta_0)}\,\left\|u_{\bt_*}\right\|_{B^1_{\infty,1}}\,\lesssim\,e^{-t(\beta_*-\bt_0)} \\
\left\| u_{\beta_0}\right\|_{B^{s-1}_{p,r}}\,&\lesssim\,\left\|u_{\bt_0}\right\|_{L^2}^{\theta_0}\,\left\|u_{\bt_0}\right\|_{B^s_{p,r}}^{1-\theta_0}\,\lesssim\,
e^{-(\bt_*-\bt_0)\theta_0t}\,\left(1\,+\,\left\|u_{\bt_0}\right\|_{B^s_{p,r}}\right)\,,
\end{align*}
which hold true owing to \eqref{ass:decay-low_0} and \eqref{est:u-L^2_0}. Then, up to taking a smaller $\lambda>0$ if necessary, we obtain the bound
\begin{align*}
\left\|u_{\beta_0}(t)\right\|_{B^{s}_{p,r}}\,&\,\lesssim\,
\left\|u_{0}\right\|_{B^{s}_{p,r}}\,+\,\int^t_0e^{-\lambda\t}\left(
\left\| u_{\beta_0}\right\|_{B^{s}_{p,r}}\, +\,\left\|\nabla\rho\right\|_{B^{s-1}_{p,r}}\,+\,1\right) \,\dd\t\,. 
\end{align*}

On the other hand, it follows from inequality \eqref{lin:est:nab:a} and the fact that, under our assumptions, one has $s>1$ always, that
\begin{equation*}
\left\|\nabla \rho\right\|_{B^{s-1}_{p,r}}\,\lesssim\,\left\|\nabla\rho_0\right\|_{B^{s-1}_{p,r}}\,+\,
\int^t_0 e^{-\beta_0 \t}\,\left( \left\|\nabla  u_{\beta_0} \right\|_{L^{\infty}}\,\left\|\nabla \rho\right\|_{B^{s-1}_{p,r}}\, +\,
\left\|\nabla \rho \right\|_{L^{\infty}}\, \left\|u_{\beta_0}\right\|_{B^{s}_{p,r}}\right)\,\dt\,.
\end{equation*}
Therefore, defining
\[
 \mc U_{\bt_0}(t)\,:=\,\left\|u_{\beta_0}(t)\right\|_{B^{s}_{p,r}}\,+\,\left\|\nabla \rho\right\|_{B^{s-1}_{p,r}}
\]
and using \eqref{ass:decay-low_0} again, we easily find
\begin{align*}
 \mc U_{\bt_0}(t)\,\lesssim\,\mc U_{\bt_0}(0)\,+\,\int^t_0e^{-\lambda\t}\,\Big(\mc U_{\bt_0}(\t)\,+\,1\Big)\,\dd\t \,.
\end{align*}
We then conclude, by an application of the Gr\"onwall lemma, that
\begin{equation*}
\forall\,t\geq0\,,\qquad\qquad 
\left\|\nabla \rho(t)\right\|_{B^{s-1}_{p,r}}\, +\, \left\|u_{\bt_0}(t)\right\|_{B^{s}_{p,r}}\,<\,+\,\infty\,.
\end{equation*}
Inserting this bound into \eqref{0002}, also the sought decay for the pressure gradient follows.
The proof of the lemma is thus completed.
\end{proof}


\appendix

\section{Appendix: proof of some technical lemmas} \label{a:app}

In this appendix, we show the proof of some technical lemmas from Section \ref{s:tools}, which have been needed in the course of our proof.

We start by presenting the proof of Lemma \ref{com:damping}.

\begin{proof}[Proof of Lemma \ref{com:damping}]
The beginning of the proof follows the main lines of the proof of classical commutator estimates (see Chapter 2 of \cite{BCD} for details).
More precisely, we use Bony's paraproduct decomposition to write
\begin{align*}
[f,\Delta_j ] v\, &=\, \left[\mathcal{T}_f, \Delta_j \right] v\, +\,\mc T_{\Delta_jv}f\, -\,\Delta_j\mc T_vf\,+\,\mc R\big(f,\Delta_jv\big)\,-\,\Delta_j \mathcal{R}(f, v)\,=\,
\sum_{n=1}^5A_n^{(j)}
\end{align*}
We are going to estimate the terms $ A_n^{(j)} $ separately.

The bound for $A_1^{(j)}\,=\,\left[\mathcal{T}_f, \Delta_j \right] v$ is analogous to the one stated in Lemma \ref{l:ParaComm} (in fact, by spectral localisation,
it would be easy to reduce the proof to that lemma). Indeed, by using Lemma 2.97 of \cite{BCD}, we can write
\begin{align*}
2^{js}\,\left\|A_1^{(j)}\right\|_{L^{p}}\,&\lesssim \, \sum_{|i-j|\leq 5} 2^{js}\, 2^{-j}\, \left\| \nabla S_{i-1} f\right\|_{L^{\infty}}\,
\| \Delta_{i} v \|_{L^{p}}\,\lesssim\,
2^{j(s-1)}\, \| \nabla f \|_{L^{\infty}} \, \sum_{|i-j|\leq 5} \| \Delta_{i} v \|_{L^{p}} \\
&\lesssim \,  c_j^1\, \|\nabla f\|_{L^{\infty}}\, \| v\|_{B^{s-1}_{p,r}}\,,
\end{align*}
for a suitable sequence $\big(c_j^1\big)_j \in \ell^r $ of unitary norm.

We now focus on the second term
\[
 A_2^{(j)}\,=\,\mc T_{\Delta_jv}f\,=\,
 \sum_{i \geq j-3} \Delta_j S_{i-1} v\, \Delta_i f\,. 
\]
As only the high frequencies of $f$ appear in the previous expression, we can make use of the second Bernstein inequality to estimate
\begin{align*}
2^{js}\,\left\| A_2^{(j)}\right\|_{L^{p}}\, &\lesssim\,  2^{js}
\sum_{i\geq j+1}2^{-i}\,2^i\, \| \Delta_i f \|_{L^{\infty}}\,\| \Delta_j S_{i-1} v \|_{L^{p}}	\\
&\lesssim\, 2^{j(s-1)}\, \| \nabla f\|_{B^0_{\infty,1}}\, \|\Delta_{j} v \|_{L^{p}}\,\lesssim\, c_j^2\, \|\nabla f\|_{B^0_{\infty,1}}\,\| v\|_{B^{s-1}_{p,r}},
\end{align*}
with the sequence $ \big(c_j^2\big)_j$ belonging to the unit ball of $\ell^r $.

Next, for the third term $A_3^{(j)}\,=\,-\,\Delta_j\mc T_vf$, we notice that
\[
\left\|\left(2^{js}\,\left\|A_3^{(j)}\right\|_{L^p}\right)_j\right\|_{\ell^r}\,\approx\,\left\|T_vf\right\|_{B^s_{p,r}}\,.
\]
Hence, we can use the continuity properties of the paraproduct operator, see Proposition \ref{p:op}, together with Remark 2.83 oc \cite{BCD}, to bound
\begin{align*}
2^{js}\,\left\|A_3^{(j)}\right\|_{L^{p}}\,\lesssim\,c_j^3\, \|\nabla f\|_{B^{s-1}_{p,r}}\,\|v\|_{L^{\infty}}\,,
\end{align*}
with  $\big(c_j^3\big)_j \in \ell^r $ of unitary norm.

As for $A_4^{(j)}$, we use spectral localisation properties to rewrite 
\[
 A_4^{(j)}\,=\,\sum_{|l-j| \leq 1 } \sum_{|l-k|\leq 1} \Delta_k f\, \Delta_j \Delta_l v\, =\,  \sum_{|k-j| \leq 2 } \sum_{|l-k|\leq 1} \Delta_k f\, \Delta_j \Delta_l v\,.
\]
Then, we can bound
\begin{align*}
2^{js}\,\left\| A_4^{(j)}\right\|_{L^{p}}\, &\lesssim \, 2^{js}\,\sum_{|k-j| \leq 2 }\sum_{|l-k|\leq1}\left\|\Delta_k f\,\Delta_j\Delta_lv\right\|_{L^{p}}\,\lesssim\,
c_j^4\, \left(\|\nabla f\|_{B^{s-1}_{p,r}}\,\|v\|_{L^{\infty}}\,+\,\|f\|_{L^{\infty}}\,\|v\|_{L^{\infty}}\right)\,,
\end{align*}
for a suitable sequence  $\big(c_j^4\big)_j \in \ell^r $ of unitary norm. Notice that the presence of the term $\|f\|_{L^{\infty}}\,\|v\|_{L^{\infty}}$
is needed in order to control the terms for which $k=l=0$, so that no high frequencies are involved
(nor of $f$, nor of $v$, so one cannot make a gradient appear in the estimates).

Finally, we consider the last term $A_5^{(j)}\,=\,-\Delta_j\mc R(f,v)$. For controlling it, we can directly apply Proposition \ref{p:op}
and get
\begin{align*}
2^{js}\,\left\|A_5^{(j)}\right\|_{L^{p}}\,\lesssim\,c_j^5\,\left\|\mc R(f,v)\right\|_{B^s_{p,r}}\,\lesssim\,
\| f\|_{B^{1}_{\infty,\infty}} \| v\|_{B^{s-1}_{p,r}}\,,
\end{align*}
where, again, the sequence $\big(c^5_j\big)_j$ belongs to $\ell^r$ and has unitary norm.

In the end, after adding all the estimates for the terms $A_n^{(j)}$, we deduce the proof of the lemma by setting $ c_j\, :=\,\sum_{n=1}^5 c_j^n\,\in \,\ell^r $.  
\end{proof}

Now, we present the proof of Lemma \ref{lem:pres:est:Besov}, concerning fine estimates for the pressure gradient in Besov spaces.
The guidelines of the proof being similar to \cite{D:F, F_2012}, we limit ourselves to give a sketch of the argument here, indicating only the main modifications
in order to get the estimates used in our work.

\begin{proof}[Proof of Lemma \ref{lem:pres:est:Besov}]
We are going to argue in a similar way as in Section 5 of \cite{D:F}, so we will be a bit sketchy here and refer to that paper for details.
In addition, for simplicity we will consider here only the case $ (s,p,r) = (1, \infty, 1) $, beacause is the one that we will use;
besides, this case is the most involved one, as we will need to estimates non-linear term in the space $B^0_{\infty,1}$ where, as already remarked several times, the product is not continuous.

First of all notice that we can rewrite equation \eqref{eq:elliptic} as follows:
\[ 
-\,\Delta \Pi \,=\, - \,\nabla \log (\rho) \cdot \nabla \Pi \,+\,   \rho\,  \div F\,.
\] 
Hence,  by cutting $ \nabla \Pi $ into low and high frequency and using that $ p \geq 2$, we have
\begin{align*}
\left\|\nabla \Pi\right\|_{B^1_{\infty,1}}\,&\lesssim\,\|\nabla \Pi\|_{L^2}\,+\,\left\|\Delta \Pi\right\|_{B^0_{\infty,1}} \\
&\lesssim\,\|\nabla\Pi\|_{L^2}
\,+\,\left\|\nabla\log\rho\cdot\nabla \Pi\right\|_{B^0_{\infty,1}}
\,+\,\left\|\rho\,\div F \right\|_{B^0_{\infty,1}}\,.
\end{align*}
Using inequality \eqref{algebra:B^0} and Proposition \ref{p:comp}, we get 
\begin{equation*}
\left\|\nabla\log\rho\cdot\nabla \Pi\right\|_{B^0_{\infty,1}}\,\lesssim\,\left\|\nabla(\rho-1)\right\|_{B^0_{\infty,1}}\,\left\|\nabla \Pi\right\|_{B^{1/2}_{\infty,1}}\,.
\end{equation*}
However, this estimate makes a non-linear dependence appear on the Besov norm of $\rho-1$ and this is not suitable for our scopes.
Therefore, we rather use a Bony's paraproduct decomposition, together with Proposition \ref{p:comp} again, and bound
\begin{align*}
\left\|\nabla\log\rho\cdot\nabla \Pi\right\|_{B^0_{\infty,1}}\,&\lesssim\,\left\|\nabla(\rho-1)\right\|_{B^0_{\infty,1}}\,\left\|\nabla \Pi\right\|_{L^{\infty}}\,+\,
\left\|\nabla(\rho-1)\right\|_{L^{\infty}}\,\left\|\nabla \Pi\right\|_{B^{1/2}_{\infty,1}} \\
&\lesssim\,\left\|\nabla(\rho-1)\right\|_{B^0_{\infty,1}}\,\left\|\nabla \Pi\right\|_{L^{\infty}}\,+\,
\left\|\nabla(\rho-1)\right\|_{L^{\infty}}\,\left\|\nabla \Pi \right\|^{\de}_{L^2}\,\left\|\nabla \Pi\right\|^{1-\de}_{B^{1}_{\infty,1}}\,,
\end{align*} 
for a suitable $\de\in\,]0,1[\,$ coming from interpolation. After an application of the Young inequality, that estimate in turn yields the bound
\begin{align*}
\left\|\nabla \Pi\right\|_{B^1_{\infty,1}}\,&\lesssim\,\left(1\,+\,\|\nabla\rho\|_{L^\infty}^\eta\right)\,\|\nabla\Pi\|_{L^2}\,+\,
\left\|\nabla(\rho-1)\right\|_{B^0_{\infty,1}}\,\left\|\nabla \Pi\right\|_{L^{\infty}}\,+\,
 \left\| \rho\,  \div(F)  \right\|_{B^{1}_{\infty,1}}\,, 
\end{align*}
where we have set $\eta\,=\,1+1/\de$. Finally, Lemma \ref{l:laxmilgram} implies the sought inequality \eqref{gen:pres:ine}. 
\end{proof}



\addcontentsline{toc}{section}{References}
{\small

\begin{thebibliography}{xxx}

\bibitem{Korean} H. Bae, W. Lee, J. Shin:
{\it A blow-up criterion for the inhomogeneous incompressible Euler equations}.
Nonlinear Anal., {\bf 196} (2020), 111774, 9 pp.

\bibitem{BCD} H. Bahouri, J.-Y. Chemin, R. Danchin:
{\it ``Fourier analysis and nonlinear partial differential equations''}.
Grundlehren der mathematischen Wissenschaften [Fundamental Principles of Mathematical Sciences], 343. Springer, Heidelberg, 2011.

\bibitem{B-F-P} E. Bocchi, F. Fanelli, C. Prange:
{\it Anisotropy and stratification effects in the dynamics of fast rotating compressible fluids}.
Ann. Inst. H. Poincar\'e C Anal. Non Lin\'eaire, {\bf 39} (2022), n. 3, 647-704.

\bibitem{Bony} J.-M. Bony: 
{\it Calcul symbolique et propagation des singularit\'es pour les \'equations aux d\'eriv\'ees partielles non lin\'eaires}.
Ann. Sci. \'Ecole Norm. Sup., {\bf 14} (1981), n. 2, 209-246.

\bibitem{B:Z} K. Beauchard, E. Zuazua:
{\it Large time asymptotics for partially dissipative hyperbolic systems}.
Arch. Ration. Mech. Anal., {\bf 199} (2011), n. 1, 177-227.

\bibitem{B-CB-P} R. Bianchini, T. Crin-Barat, M. Paicu:
{\it Relaxation approximation and asymptotic stability of stratified solutions to the IPM equation}.
Arch. Ration. Mech. Anal., {\bf 248} (2024), n. 1, Paper No. 25, 35 pp.

\bibitem{Brav-F} M. Bravin, F. Fanelli:
{\it Fast rotating non-homogeneous fluids in thin domains and the Ekman pumping effect}.
J. Math. Fluid Mech., {\bf 25} (2023), n. 4, Paper No. 83, 41 pp. 


\bibitem{Ca-Co-L} \'A. Castro, D. C\'ordoba, D. Lear:
{\it On the asymptotic stability of stratified solutions for the 2D Boussinesq equations with a velocity damping term}.
Math. Models Methods Appl. Sci., {\bf 29} (2019), n. 7, 1227-1277.

\bibitem{C-D-G-G} J.-Y. Chemin, B. Desjardins, I. Gallagher, E. Grenier:
\textit{``Mathematical geophysics. An introduction to rotating fluids and the Navier-Stokes equations''}.
Oxford Lecture Series in Mathematics and its Applications, Oxford University Press, Oxford, 2006.

\bibitem{CWZZ} Q. Chen, D. Wei, P. Zhang, Z. Zhang:
{\it Nonlinear inviscid damping for $2$-D inhomogeneous incompressible Euler equations}.
Submitted (2023), arxiv preprint \texttt{2303.14858v1}.


\bibitem{CB:D2} T. Crin-Barat, R. Danchin:
{\it Partially dissipative hyperbolic systems in the critical regularity setting: the multi-dimensional case}.
J. Math. Pures Appl. (9), {\bf 165} (2022), 1-41.

\bibitem{CB:D3} T. Crin-Barat, R. Danchin:
{\it Global existence for partially dissipative hyperbolic systems in the $L^p$ framework, and relaxation limit}.
Math. Ann., {\bf 386} (2023), n. 3-4, 2159-2206.


\bibitem{CB:Z2} T. Crin-Barat, L.-Y. Shou, E. Zuazua:
{\it Large time asymptotics for partially dissipative hyperbolic systems without Fourier analysis: application to the nonlinearly damped $p$-system}.
Submitted (2023), arxiv preprint \texttt{2308.08280}.

\bibitem{CR-B} B. Cushman-Roisin, J.-M. Beckers:
{\it ``Introduction to geophysical fluid dynamics''}.
International Geophysics Series, Elsevier/Academic Press, Amsterdam, 2011.

\bibitem{D1} R. Danchin:
{\it On the well-posedness of the incompressible density-dependent Euler equations in the $L^p$ framework}.
J. Differential Equations, {\bf 248} (2010), n. 8, 2130-2170.

\bibitem{D:F} R. Danchin, F. Fanelli:
{\it The well-posedness issue for the density-dependent Euler equations in endpoint Besov spaces}.
J. Math. Pures Appl. (9), {\bf 96} (2011), n. 3, 253-278.

\bibitem{F_2012} F. Fanelli:
{\it Conservation of geometric structures for non-homogeneous inviscid incompressible fluids}.
Comm. Partial Differential Equations, {\bf 37} (2012), n. 9, 1553-1595.

\bibitem{F-GB-S} F. Fanelli, R. Granero-Belinch\'on, S. Scrobogna:
{\it Well-posedness theory for non-homogeneous incompressible fluids with odd viscosity}.
J. Math. Pures Appl. (9), {\bf 187} (2024), 58-137.

\bibitem{F-V} F. Fanelli, A. F. Vasseur:
{\it Effective velocity and $L^\infty$-based well-posedness for incompressible fluids with odd viscosity}.
SIAM J. Math. Anal., {\bf 57} (2025), n. 1, 153-189.

\bibitem{Gr-Masm} E. Grenier, N. Masmoudi:
{\it Ekman layers of rotating fluids, the case of well prepared initial data}.
Comm. Partial Differential Equations, {\bf 22} (1997), n. 5-6, p. 953-975.

\bibitem{HK} T. Hmidi,  S. Keraani:
{\it Incompressible viscous flows in borderline Besov spaces}.
Arch. Ration. Mech. Anal., {\bf 189} (2008), n. 2,  283-300.

\bibitem{Masm} N. Masmoudi:
{\it Ekman layers of rotating fluids: the case of general initial data}.
Comm. Pure Appl. Math., {\bf 53} (2000), n. 4, p. 432-483.

\bibitem{Ped} J. Pedlosky:
{\it ``Geophysical fluid dynamics''}.
Springer-Verlag, New York, 1987.

\bibitem{S-T-W} T. Sideris, B. Thomases, D. Wang:
{\it Long time behaviour of solutions to the 3D compressible Euler equations with damping}.
Comm. Partial Differential Equations, {\bf 28} (2003), n. 3-4, 795-816. 

\bibitem{Vis} M. Vishik:
{\it Hydrodynamics in Besov spaces}.
Arch. Ration. Mech. Anal., {\bf 145} (1998), n. 3, 197-214.


\end{thebibliography}

}
\end{document}